\documentclass{article}
\usepackage{iclr2023_conference,times}
\usepackage{algorithm}
\usepackage[noend]{algpseudocode}

\usepackage{amsthm}
\usepackage{latexsym}
\usepackage{amsmath}
\usepackage{amssymb}
\usepackage{bm}
\usepackage{mathrsfs}
\usepackage{url}
\usepackage{bbm}
\usepackage{enumitem}
\usepackage{nicefrac}

\newcommand{\newterm}[1]{{\bf #1}}

\def\Figref#1{Figure~\ref{#1}}

\def\eqref#1{(\ref{#1})}

\def\Algref#1{Algorithm~\ref{#1}}

\def\norm#1{\left\|#1\right\|}
\def\inner#1{\left\langle#1\right\rangle}

\def\eps{\varepsilon}

\def\vzero{{\bm{0}}}

\def\va{{\bm{a}}}
\def\vb{{\bm{b}}}

\def\vg{{\bm{g}}}
\def\vh{{\bm{h}}}

\def\vm{{\bm{m}}}

\def\vp{{\bm{p}}}

\def\vt{{\bm{t}}}
\def\vu{{\bm{u}}}
\def\vv{{\bm{v}}}
\def\vw{{\bm{w}}}
\def\vx{{\bm{x}}}
\def\vy{{\bm{y}}}
\def\vz{{\bm{z}}}

\def\mA{{\bm{A}}}
\def\mB{{\bm{B}}}
\def\mC{{\bm{C}}}

\def\mH{{\bm{H}}}
\def\mI{{\bm{I}}}

\def\mM{{\bm{M}}}
\def\mN{{\bm{N}}}

\def\mQ{{\bm{Q}}}

\def\mU{{\bm{U}}}
\def\mV{{\bm{V}}}

\DeclareMathAlphabet{\mathsfit}{\encodingdefault}{\sfdefault}{m}{sl}
\SetMathAlphabet{\mathsfit}{bold}{\encodingdefault}{\sfdefault}{bx}{n}

\def\gB{{\mathcal{B}}}

\def\gF{{\mathcal{F}}}

\def\gL{{\mathcal{L}}}

\def\gO{{\mathcal{O}}}

\def\gS{{\mathcal{S}}}

\def\gW{{\mathcal{W}}}
\def\gX{{\mathcal{X}}}
\def\gY{{\mathcal{Y}}}
\def\gZ{{\mathcal{Z}}}

\def\sN{{\mathbb{N}}}

\def\sS{{\mathbb{S}}}

\newcommand{\unif}{\mathrm{Unif}}

\newcommand{\E}{\mathbb{E}}

\newcommand{\R}{\mathbb{R}}

\DeclareMathOperator*{\argmax}{arg\,max}
\DeclareMathOperator*{\argmin}{arg\,min}

\newtheorem{theorem}{Theorem}
\newtheorem{lemma}[theorem]{Lemma}
\newtheorem*{lemma*}{Lemma}
\newtheorem{corollary}{Corollary}

\newtheorem{proposition}[theorem]{Proposition}
\newtheorem{definition}{Definition}

\newtheorem{assumption}{Assumption}
\theoremstyle{definition}

\newcommand{\eg}{\emph{e.g.}}
\newcommand{\ie}{\emph{i.e.}}

\usepackage{enumitem}
\usepackage[pagebackref=true, backref=true]{hyperref}
\usepackage{url}
\usepackage{color,colortbl}
\definecolor{darkblue}{rgb}{0.0,0.0,0.65}
\definecolor{darkred}{rgb}{0.65,0.0,0.0}
\definecolor{darkgreen}{rgb}{0.0,0.5,0.0}
\definecolor{tab:blue}{RGB}{31,119,180}  %
\definecolor{tab:red}{RGB}{214,39,40}  %
\definecolor{tab:green}{RGB}{44,160,44}  %
\definecolor{tab:orange}{RGB}{255,127,14}  %
\hypersetup{
	colorlinks = true,
	citecolor  = darkblue,
	linkcolor  = darkred,
	filecolor  = darkblue,
	urlcolor   = darkblue,
}
\usepackage{graphicx, grffile}
\usepackage{caption}
\usepackage{subcaption}

\title{SGDA with shuffling: faster convergence for nonconvex-P{\L} minimax optimization}

\author{Hanseul Cho, Chulhee Yun \\
Kim Jaechul Graduate School of AI, KAIST\\
\texttt{\{jhs4015, chulhee.yun\}@kaist.ac.kr} \\
}

\iclrfinalcopy
\begin{document}

\maketitle

\begin{abstract}
Stochastic gradient descent-ascent (SGDA) is one of the main workhorses for solving finite-sum minimax optimization problems. Most practical implementations of SGDA randomly reshuffle components and sequentially use them (\ie, without-replacement sampling); however, there are few theoretical results on this approach for minimax algorithms, especially outside the easier-to-analyze (strongly-)monotone setups. To narrow this gap, we study the convergence bounds of SGDA with random reshuffling (\newterm{SGDA-RR}) for smooth nonconvex-nonconcave objectives with Polyak-{\L}ojasiewicz (P{\L}) geometry. We analyze both simultaneous and alternating SGDA-RR for \emph{nonconvex-P{\L}} and \emph{primal-P{\L}-P{\L}} objectives, and obtain convergence rates faster than with-replacement SGDA. Our rates extend to mini-batch SGDA-RR, recovering known rates for full-batch gradient descent-ascent (GDA). Lastly, we present a comprehensive lower bound for GDA with an arbitrary step-size ratio, which matches the full-batch upper bound for the primal-P{\L}-P{\L} case.
\end{abstract}

\section{Introduction}
\label{sec:intro}
A finite-sum minimax optimization problem aims to solve the following:%
\begin{equation}
	\min_{\vx\in\gX}\max_{\vy \in \gY} f(\vx; \vy) := \frac{1}{n}\sum_{i=1}^n f_i (\vx; \vy), %
    \label{eq:minimax}
\end{equation}
where $f_i$ denotes the $i$-th component function. In plain language, we want to minimize the average of $n$ component functions for $\vx$, while maximizing it for $\vy$ given $\vx$.
There are many important areas in modern machine learning that fall within the minimax problem, including
generative adversarial networks (GANs) \citep{goodfellow2020generative}, %
adversarial attack and robust optimization \citep{madry2018towards, sinha2018certifiable},
multi-agent reinforcement learning (MARL) \citep{li2019robust},
AUC maximization \citep{ying2016stochastic, Liu2020Stochastic, yuan2021large}, and many more. In most cases, the objective $f$ is usually nonconvex-nonconcave, \ie, neither convex in $\vx$ nor concave in $\vy$. Since general nonconvex-nonconcave problems are known to be intractable, we would like to tackle the problems with some additional structures, such as smoothness and Polyak-{\L}ojasiewicz (P{\L}) condition(s). %
We elaborate the detailed settings for our analysis, \emph{nonconvex-P{\L}} and \emph{primal-P{\L}-P{\L}} (or, \emph{P{\L}(${\it \Phi}$)-P{\L}}), in Section~\ref{sec:setup}.

One of the simplest and most popular algorithms to solve the problem~\eqref{eq:minimax} would be \emph{stochastic gradient descent-ascent} (\newterm{SGDA}). This naturally extends the idea of stochastic gradient descent (SGD) used for minimization problems. Given an initial iterate $(\vx_0; \vy_0)$, at time $t\in \sN$, SGDA (randomly) chooses an index $i(t) \in \{1, \dots, n\}$ and accesses the $i(t)$-th component to perform a pair of updates
\begin{align*}
    \bigg[~
    \begin{split}
    	&\vx_{t} = \vx_{t-1} - \alpha \nabla_1\, f_{i(t)}(\vx_{t-1}; \vy_{t-1}), \\
    	&\vy_{t} = \vy_{t-1} + \beta \nabla_2\, f_{i(t)}(\vx'; \vy_{t-1}),
    \end{split} 
    \qquad \text{where } \vx' =
    \begin{cases}
        \vx_{t-1}, & \text{(\emph{simSGDA}), or} \\
        \vx_{t}, & \text{(\emph{altSGDA}).}
    \end{cases}
\end{align*}
Here, $\alpha>0$ and $\beta>0$ are the step sizes and $\nabla_j$ denotes the gradient with respect to $j$-th argument for $f_{i(t)}$ ($j=1,2$).
As shown in the update equations above, there are two widely used versions of SGDA: \emph{simultaneous SGDA} (\newterm{simSGDA}), and \emph{alternating SGDA} (\newterm{altSGDA}).

In such stochastic gradient methods, there are two main categories of sampling schemes for the component indices $i(t)$. 
One way is to sample $i(t)$ independently (in time) and uniformly at random from $\{1, \dots, n\}$, which is called \emph{with-replacement sampling}. This scheme is widely adopted in theory papers because it makes analysis of stochastic methods amenable: the noisy gradients $\nabla f_{i(t)}$ are independent over time $t$ and are unbiased estimators of the full-batch gradient $\nabla f$. In contrast, the vast majority of practical implementations employ \emph{without-replacement sampling}, indicating a huge theory-practice gap. In without-replacement sampling, we sample each index precisely once at each epoch. Perhaps the most popular of such schemes is \emph{random reshuffling} (\newterm{RR}), which uniformly randomly shuffles the order of indices at the beginning of every epoch. Unfortunately, it is well-known that without-replacement methods are much more difficult to analyze theoretically, largely because the sampled indices in each epoch are no longer independent of each other.

Interestingly, for minimization problems, several recent works overcome this obstacle and show that SGD using without-replacement sampling leads to faster convergence, given that the number of epochs is large enough
\citep{nagaraj2019sgd, ahn2020sgd, mishchenko2020random, rajput2020closing, nguyen2021unified, yun2021open, yun2022minibatch}.
On the other hand, for minimax problems like \eqref{eq:minimax}, the majority of the studies still assume with-replacement sampling and/or rely on independent unbiased gradient oracles \citep{nouiehed2019solving, guo2020fast,lin2020gradient,yan2020optimal,yang2020global, loizou2021stochastic, beznosikov2022stochastic}. There are very few results on minimax algorithms using without-replacement sampling; even most of the existing ones take advantage of (strong-)convexity (in $\vx$) and/or (strong-)concavity (in $\vy$) \citep{das2022sampling, maheshwari2022zeroth,yu2022fast}. Detailed comparative analysis of these works is conducted in Section~\ref{sec:discussion}.

Putting all these issues into consideration, our main question is the following.
\vspace*{-5pt}
\begin{center}
    \it Does SGDA using without-replacement component sampling provably converge fast, \\ even on smooth nonconvex-nonconcave objective $f$ with P{\L} structures?
\end{center}
\vspace*{-5pt}

\subsection{Summary of our contributions}
To answer the question, we analyze the convergence of SGDA with random reshuffling (\newterm{SGDA-RR}, \Algref{alg:SGDA wo replacement}).
We analyze both the simultaneous and alternating versions of SGDA-RR and prove convergence theorems for the following two regimes.
Here we denote the step size ratio as {\small $r=\beta/\alpha$}.
\vspace{-10pt} %
\begin{itemize}
    \item When $-f(\vx;\vy)$ satisfies $\mu_2$-P{\L} condition in $\vy$ (\newterm{nonconvex-P{\L}}) and component function $f_i$'s are $L$-smooth, we prove that SGDA-RR with $r\gtrsim (L/\mu_2)^2$ converges to $\eps$-stationarity in expectation after $\gO\left(nrL\eps^{-2} + \sqrt{n}r^{1.5}L\eps^{-3} \right)$ gradient evaluations (Theorem~\ref{thm:NCPL}).
    \item Further assuming $\mu_1$-P{\L} condition on
    $\Phi(\cdot):=\max_\vy f(\cdot;\vy)$ (\newterm{primal-P{\L}-P{\L}}, or \newterm{P{\L}($\Phi$)-P{\L}}), we prove that SGDA-RR with $r\gtrsim(L/\mu_2)^2$ converges within $\eps$-accuracy in expectation after {\small $\tilde{\gO}\left(\frac{nLr}{\mu_1}\log(\eps^{-1}) + \sqrt{n}L(\frac{r}{\mu_1})^{1.5}\eps^{-1}\right)$} gradient evaluations (Theorem~\ref{thm:2PL}). %
\end{itemize}
\vspace{-5pt} %
As will be discussed in Section~\ref{sec:discussion}, the rates shown above are \emph{faster} than existing results on with-replacement SGDA. In fact, Theorems~\ref{thm:NCPL}~\&~\ref{thm:2PL} are special cases ($b=1$) of our extended theorems (Theorems~\ref{thm:NCPL 2}~\&~\ref{thm:2PL 2} in Appendix~\ref{sec:minibatch}) that analyze \emph{mini-batch} SGDA-RR of batch size $b$; by setting $b=n$, we also recover known convergence rates for full-batch gradient descent ascent (GDA). Hence, our analysis covers the entire spectrum between vanilla SGDA-RR ($b=1$) and GDA ($b=n$).%
\vspace*{-5pt} %
\begin{itemize}
    \item Additionally, we provide complexity lower bounds for solving strongly-convex-strongly-concave (SC-SC) minimax problems using full-batch simultaneous GDA with an arbitrarily fixed step size ratio $r = \beta/\alpha$.
    Perhaps surprisingly, we find that the lower bound for SC-SC functions matches the convergence upper bound for a much larger class of primal-P{\L}-P{\L} functions when the step size ratio satisfies $r \gtrsim L^2/\mu_2^2$ (Theorem~\ref{thm:LB}).
\end{itemize}

\section{Problem setup}
\label{sec:setup}
\vspace*{-5pt} %
\subsection{Notation}
\vspace*{-5pt} %
In our problem~\eqref{eq:minimax}, the domain of every $f_i$ is $\gZ = \gX\times\gY$, where $\gX=\R^{d_x}$, $\gY=\R^{d_y}$, and $\gZ=\R^d$: we concern unconstrained problems for simplicity.
We denote the Euclidean norm and the standard inner product by $\norm{\cdot}$ and $\inner{\cdot,\cdot}$, respectively.
We often use an abbreviated notation $\vz= (\vx; \vy) \in \gZ$ for $\vx\in \gX$ and $\vy\in\gY$. Even when $\vz$ or $(\vx;\vy)$ is followed by superscripts and/or subscripts, we use the symbols interchangeably; \eg, $\vz_i^k = (\vx_i^k; \vy_i^k)$. 
Note that we split the arguments $\vx$ (for minimization) and $\vy$ (for maximization) by a semicolon (`;').
We use $\nabla_1$ and $\nabla_2$ to denote the gradients with respect to first and second arguments, respectively.%
Accordingly, we can write the full gradient as, \eg, $\nabla g = [\nabla_1 g^\top;\nabla_2 g^\top]^\top$.
For a positive integer $N$, we denote $[N]:=\{1, \dots,N\}$. Let the set $\sS_N$ be a symmetric group of degree $N$. That is, each \emph{permutation} $\sigma\in\sS_N$ is a bijection from $[N]$ to itself, or equivalently, a re-arrangement of $[N]$.
Lastly, we use the usual $\gO$/$\Omega$/$\Theta$ notation for bounds, where $\tilde{\gO}$/$\tilde{\Omega}$/$\tilde{\Theta}$ are used for hiding some logarithmic factors, respectively.%
\subsection{Algorithms: simSGDA-RR \& altSGDA-RR}
\label{sec:setup:algorithms}
As we explained in Section~\ref{sec:intro}, we consider simSGDA and altSGDA combined with RR, a without-replacement sampling scheme. We call them \newterm{simSGDA-RR} and \newterm{altSGDA-RR}, respectively. We present a detailed description of the methods in \Algref{alg:SGDA wo replacement}. For completeness, we also provide an extended version that uses mini-batches of size $\ge 1$ (\Algref{alg:SGDA wo replacement minibatch}) in Appendix \ref{sec:minibatch}. For comparison, we call the SGDA algorithms using \emph{with-replacement} sampling by just \newterm{simSGDA} and \newterm{altSGDA}.

The quantities $\alpha,\beta>0$ are step sizes associated with $\vx$ and $\vy$, respectively.
We use two separate symbols $\alpha$ and $\beta$ to allow the two step sizes to be different. Such algorithms are sometimes called \emph{two-time-scale} algorithms, in a broader sense, and they are adopted in nonconvex minimax optimization problems~\citep{heusel2017gans, lin2020gradient, yang2020global}. In fact, a recent result \citep{li2022convergence} shows that having $\alpha \neq \beta$ is sometimes \emph{necessary} for convergence.

\begin{algorithm}[t]
\caption{{\color{teal}sim}SGDA/{\color{brown}alt}SGDA-{\color{purple} RR}}
\label{alg:SGDA wo replacement}
\begin{algorithmic}[1]
\State {\bf Given:} The number of components $n$; the number of epochs $K$; step sizes $\alpha, \beta>0$ 
\State {\bf Initialize:} $(\vx_0^1; \vy_0^1)\in \R^{d_x}\times\R^{d_y}$
\For{$k\in [K]$}
    \State Sample $\sigma_k \sim {\color{purple}\unif(\sS_n)}$ \Comment RR: uniformly randomly shuffle the indices every epoch
    \For{$i\in[n]$}
        \State $\vx_{i}^k = \vx_{i-1}^k - \alpha \nabla_1 f_{\sigma_k(i)}(\vx_{i-1}^k; \vy_{i-1}^k)$
        \If{{\color{teal}simSGDA}-RR}
        \State $\vy_{i}^k = \vy_{i-1}^k + \beta \nabla_2 f_{\sigma_k(i)}({\color{teal}\vx_{i-1}^k}; \vy_{i-1}^k)$ \Comment simultaneous update: $\vx$ \& $\vy$ 
        \ElsIf{{\color{brown}altSGDA}-RR}
        \State $\vy_{i}^k = \vy_{i-1}^k + \beta \nabla_2 f_{\sigma_k(i)}({\color{brown}\vx_{i}^k}; \vy_{i-1}^k)$ \Comment alternating update: $\vx \rightarrow \vy$
        \EndIf
    \EndFor
    \State $(\vx_0^{k+1}; \vy_0^{k+1}) = (\vx_n^k; \vy_n^k)$
\EndFor
\end{algorithmic}
\end{algorithm}

\subsection{Assumptions and definitions}
To define the function classes that we are interested in solving, we introduce a few assumptions. 
\begin{assumption}[Component smoothness] \label{ass:smooth}
	Every $i$-th component $f_i:\gZ\rightarrow \R$ is $L$-\newterm{smooth}, \ie, $f_i$ is differentiable and $\nabla\! f_i$ is $L$-Lipschitz continuous: $\norm{\nabla\! f_i(\vz)\!-\!\nabla\! f_i(\bar{\vz})}\le L\norm{\vz-\bar{\vz}}$.
	As a result, $f_i(\bar{\vz})\!-\!f_i(\vz) \le \inner{\nabla f_i(\vz), \bar{\vz}\!-\!\vz} + \frac{L}{2}\!\norm{\bar{\vz}\!-\!\vz}^2$ ($\forall\vz, \bar{\vz}$) and the average $f$ of $f_i$'s is also $L$-smooth.\footnote{
    As we noted, Assumption~\ref{ass:smooth} directly implies the \emph{average smoothness} which is a common requirement in the analysis with unbiased gradient oracles. Nevertheless, we claim that Assumption~\ref{ass:smooth} is not more crucial than without-replacement sampling to obtain faster convergence rates: see Appendix~\ref{sec:LB with-replacement} for details and proofs.
    }
\end{assumption}
\begin{assumption}[Component gradient variance] \label{ass:bddvar}
	There exist constants $A, B\ge0$ such that, for any $\vz=(\vx;\vy)\in \gZ$ and $j \in \{1,2\}$, we have
	$\frac{1}{n}\sum_{i=1}^{n} \norm{\nabla_j\, f_i(\vz) - \nabla_j\, f(\vz)}^2 \le A\norm{\nabla_j\, f(\vz)}^2 + B$.%
\end{assumption}
\begin{assumption} \label{ass:primal,dual}
    For a function $f:\gX\times\gY \rightarrow \R$, the \newterm{primal function} $\Phi:\gX \rightarrow \R$
    is well-defined as $\Phi(\vx) := \max_{\vy'\in\gY} f(\vx; \vy')$.
    For each $\vx\in \gX$, the set
    $\gY_{\vx}^* := \argmax_{\vy'\in\gY} f(\vx; \vy')$
    is non-empty and closed. Moreover, we assume $\Phi(\vx)$ is bounded below by $\Phi^* = \inf_{\vx'\in \gX} \Phi(\vx') > -\infty$.
\end{assumption}
Note that Assumption~\ref{ass:bddvar} controls the discrepancy between the objective function $f$ and its components $f_i$'s; it is similar to Assumption~2 of \citet{nguyen2021unified}, adapted to minimax problems. 
Letting $A=0$ recovers a common assumption of the uniformly bounded variance of component gradients; thus, our assumption is a relaxation. Also, note that $A=B=0$ when $n=1$.

We now add an additional structure to our objective function, which is called Polyak-{\L}ojasiewicz (P{\L}) condition. A function $g:\R^d\rightarrow\R$ is said to be $\mu$-\newterm{P{\L}} if it has a minimum value $g^*$ and satisfies
\[\norm{\nabla g(\vt)}^2 \ge 2\mu(g(\vt)-g^*). \quad (\forall~ \vt\in \R^d) \tag{$\mu$-P{\L}}\]
Readers could find several studies and applications that the condition involves, in the papers by \citet{karimi2016linear,nouiehed2019solving,yang2020global,Liu2020Stochastic}, and more. 
Note that every $\mu$-strongly convex\footnote{We say a function $g:\R^d\rightarrow\R$ is $\mu$-strongly convex for some $\mu>0$ if it holds $g(\vx')\ge g(\vx) + \inner{\nabla g (\vx), \vx'-\vx} + (\mu/2)\norm{\vx'-\vx}^2$ ($\forall \vx,\vx'$); we say $g$ is $\mu$-strongly concave if $-g$ is $\mu$-strongly convex.}
function satisfies $\mu$-P{\L} condition, whereas a P{\L} function does not need to be convex.
Hence, $\mu$-P{\L} is a strict generalization of $\mu$-strong convexity. In addition, every stationary point of a P{\L} function is a global optimum, which is a benign property for optimization. 

We are interested in the case where our objective function $f(\vx; \vy)$ has such a structure in terms of $\vy$ (Assumption~\ref{ass:NCPL}). Sometimes, we further assume the primal function $\Phi$ is also P{\L} (Assumption~\ref{ass:primalPL}). We emphasize that we do not necessarily assume the P{\L} conditions for the individual $f_i$'s. 
\begin{assumption}[$\vy$-side P{\L}] \label{ass:NCPL}
    For each (fixed) $\vx\in\gX$, $-f(\vx; \cdot)$ is $\mu_2$-P{\L}, \ie, for every $(\vx; \vy)\in\gZ$, $\norm{\nabla_2\, f(\vx; \vy)}^2 \ge 2\mu_2 (\Phi(\vx)-f(\vx; \vy)),$
    where $\Phi$ is the primal function associated with $f$.
\end{assumption}
\begin{assumption}[Primal P{\L}, or P{\L}($\Phi$)] \label{ass:primalPL}
    The primal function $\Phi(\cdot)=\max_{\vy'} f(\vx; \vy')$ of $f$ is $\mu_1$-P{\L}, \ie, for every $\vx\in\gX$, $\norm{\nabla\, \Phi(\vx)}^2 \ge 2\mu_1 (\Phi(\vx)-\Phi^*),$
    where $\Phi^*=\min_{\vx} \Phi(\vx)$ is well-defined.
\end{assumption}
We say the function $f$ is \newterm{nonconvex-P{\L}} when it satisfies Assumption~\ref{ass:NCPL}. Since we do not assume any convexity/concavity, it is generally hard to reach global optima. Due to the $\vy$-side P{\L} condition, we can guarantee that the primal function $\Phi$ is differentiable and even $L_\Phi$-smooth with $L_\Phi \le L + L^2/\mu_2$ (Proposition~\ref{prop:Phi smooth} in Appendix~\ref{sec:propositions}).
Since the problem~\eqref{eq:minimax} can be reformulated as the minimization problem of $\Phi$ (when we can always find $\vy$ well that maximizes $f(\vx;\vy)$ given $\vx$), we could aim to find an approximate first-order stationary point of $\Phi$, by making the norm of the gradient of $\Phi$ small. 

On top of that, if $f$ satisfies both Assumptions~\ref{ass:NCPL}~and~\ref{ass:primalPL}, the function is said to be \newterm{primal-P{\L}-P{\L}}, or \newterm{P{\L}($\Phi$)-P{\L}} for short.\footnote{The P{\L}($\Phi$)-P{\L} condition is much weaker than \newterm{two-sided P{\L}} condition assuming ``$\vx$-side'' P{\L} condition: see Proposition~\ref{prop:Phi PL}. As pointed out by \citet{guo2020fast}, there exist a  P{\L}($\Phi$)-P{\L} function $g(\vx;\vy)$ that is not $\vx$-side $\mu$-P{\L} for any $\mu>0$ but even strongly \emph{concave} in $\vx$.}
In this case, we directly aim not only to decrease the primal function $\Phi$ associated with the objective function $f$ but also to increase the function value $f(\vx;\vy)$ in terms of $\vy$. To evaluate how close we are to our goal, we define a \emph{potential function} $V_\lambda$ later in Section~\ref{sec:main results}.
When we attain $V_\lambda(\vx^*, \vy^*)=0$, it implies that we arrive at a global minimax point: $f(\vx^*,\vy^*)=\Phi(\vx^*)=\Phi^*$. The function $V_\lambda$ enables us to develop a unified analysis for nonconvex-P{\L} and P{\L}($\Phi$)-P{\L} objective functions; we discuss this in greater detail in Section~\ref{sec:main results}.
\vspace*{-5pt} %

\section{Main results}
\label{sec:main results}
\vspace*{-5pt} %
Based on the assumptions stated in the previous section, we present the convergence results for both smooth nonconvex-P{\L} objectives and smooth P{\L}($\Phi$)-P{\L} objectives. Before stating the main theorems, we first introduce the most important tool for our analyses: the \emph{potential function}.
\vspace*{-5pt} %
\subsection{Potential function \(V_\lambda\)}\label{sec:main results:potential}
\vspace*{-5pt} %
For our convergence analyses, we utilize a function $V_\lambda : \gX\times \gY \rightarrow \R$ defined as
\begin{equation}
	V_\lambda(\vx; \vy):= \lambda (\Phi(\vx)-\Phi^*) + (\Phi(\vx)-f(\vx; \vy)), \label{eq:Lyapunov}
\end{equation}
where $\lambda>0$ is a constant. We borrow inspiration from \cite{yang2020global} and \cite{das2022sampling} to come up with this function, although the placement of $\lambda$ of ours is different. %
In fact, the convergence to a neighborhood of a global minimax point (if it exists) implies the reduction of this function.
For each $\vx$, a non-negative term $\Phi(\vx)-f(\vx; \vy)$ gets smaller as $\vy$ makes $f(\vx; \vy)$ larger.
The term becomes zero when $\vy = \vy^*(\vx)$ for some $\vy^*(\vx)\in \gY_{\vx}^*$, since $\Phi(\vx) = f(\vx; \vy^*(\vx))$.
Also, another non-negative term $\Phi(\vx)-\Phi^*$ gets smaller as $\vx$ makes $\Phi(\vx)$ smaller.
Thus, as $(\vx; \vy)$ approaches to a minimax optimal point, $V_\lambda(\vx; \vy)$ decreases to near zero.
In general, $V_\lambda$ is not guaranteed to attain exact zero, especially when the objective function $f(\vx;\vy)$ is nonconvex in $\vx$ (\eg, $f$ is nonconvex-P{\L}). Nevertheless, the potential function is still useful for deriving our convergence results.
\vspace*{-5pt} %
\subsection{Main theorems: upper bounds of convergence rates}
\vspace*{-5pt} %
Now, we present our main results.
We provide a detailed comparison of our theorems against existing results in Section~\ref{sec:discussion}.
We present the full proof in Appendices~\ref{sec:proofs simSGDA RR} and \ref{sec:proofs altSGDA RR}. We remark that both Theorems~\ref{thm:NCPL} and \ref{thm:2PL} are special cases (for mini-batch size $b = 1$) of their \emph{mini-batch} extensions: Theorems~\ref{thm:NCPL 2} and \ref{thm:2PL 2} in Appendix~\ref{sec:minibatch}.

\begin{theorem}[Nonconvex-P{\L}]\label{thm:NCPL}
	Suppose that $f$ satisfies Assumptions~\ref{ass:smooth}, \ref{ass:bddvar}, \ref{ass:primal,dual}, and \ref{ass:NCPL}. Let $\kappa_2 = L/\mu_2$, where $\mu_2$ is P{\L} constant of $-f(\vx; \cdot)$ at all $\vx$. Let $\lambda=4$. Choose the step sizes $\alpha$ and $\beta$ such that
    \vspace*{-3pt} %
	\[\beta = \min\left\{\frac{1}{6L\sqrt{n(n+A)}},\,\, \gO\left(\left( \frac{ V_\lambda(\vz_0^1)}{Bn^2K} \right)^{\frac{1}{3}}\right)\right\} \quad\text{and}\quad \alpha=\frac{\beta}{r},\]
	for some $r\ge 14\kappa_2^2$. Then, both simSGDA-RR and altSGDA-RR (\Algref{alg:SGDA wo replacement}) satisfy
    \vspace*{-3pt} %
 	\begin{align*}
        \frac{1}{K}\sum_{k=1}^K \E \norm{\nabla \Phi(\vx_0^k)}^2
        \le\gO\!\left(\frac{r L V_\lambda(\vz_0^1)}{K}\sqrt{1\!+\!\frac{A}{n}} \!+\! r \left(\frac{L^2 B V_\lambda(\vz_0^1)^2}{nK^2}\right)^{\!\nicefrac{1}{3}} \right).
 	\end{align*}
\end{theorem}
\vspace*{-5pt} %

{\noindent \bf Upper bound on gradient complexity.} To achieve $\eps$-stationarity of the primal function, \ie, $\frac{1}{K}\sum_{k=1}^K\E\norm{\nabla \Phi(\vx_0^k)}^2\le \eps^2$, a sufficient number of gradient evaluations (denoted by $T_\eps=nK$) is
\[T_\eps = \gO\left(\frac{r L V_\lambda(\vz_0^1)}{\eps^2} \max\left\{\sqrt{n^2+nA}, \frac{\sqrt{rnB}}{\eps}\right\}\right).\]
\vspace*{-5pt} %

\begin{theorem}[P{\L}($\Phi$)-P{\L}]\label{thm:2PL}
	Suppose that $f$ satisfies Assumptions~\ref{ass:smooth}, \ref{ass:bddvar}, \ref{ass:primal,dual}, \ref{ass:NCPL}, and \ref{ass:primalPL}. Let $\kappa_1 = L/\mu_1$ and $\kappa_2 = L/\mu_2$, where $\mu_1$ and $\mu_2$ are P{\L} constants of $\Phi(\cdot)$ and $-f(\vx; \cdot)$ (at all $\vx$), respectively. Let $\lambda=4$. Choose appropriate step sizes $\alpha$ and $\beta$ such that
    \vspace*{-5pt} %
	\begin{equation*}
		\beta = \min \left\{ \frac{1}{6L\sqrt{n(n+A)}},\,\, \tilde{\gO} \left(\frac{\kappa_2^2}{\mu_1 n K}\right)\right\} \quad\text{and}\quad \alpha=\frac{\beta}{r},
	\end{equation*}
    \vspace*{-5pt} %
	for some $r\ge 14\kappa_2^2$. Then, both simSGDA-RR and altSGDA-RR (\Algref{alg:SGDA wo replacement}) satisfy
	\[\E[V_\lambda(\vz_0^{K+1})] \le \gO\left( V_\lambda(\vz_0^1)\cdot\exp\left(-\frac{K}{12\kappa_1r\sqrt{1+\frac{A}{n}}}\right)\right)+\tilde{\gO}\left(\frac{\kappa_1^2r^3 B}{\mu_1 n K^2}\right).\]
\end{theorem}
\vspace*{-5pt} %
{\bf Upper bound on gradient complexity.} To achieve $\eps^2$-accuracy on expectation of $V_\lambda(\vz_n^K)$, \ie, $\E [V_\lambda(\vz_n^K)]\le \eps^2$, a sufficient number of gradient evaluations (denoted by $T'_\eps=nK$) is
\[T'_\eps = \max\left\{\gO\left(\kappa_1r\sqrt{n^2+nA} \cdot \log \left(\frac{V_\lambda(\vz_0^1)}{\eps}\right)\right), \tilde{\gO}\left(\frac{\kappa_1r^{3/2}}{\eps}\sqrt{\frac{nB}{\mu_1}}\right)\right\}.\]
\vspace*{-10pt} %

{\bf Remark on step size ratio.} In both theorems, we use the step sizes of ratio $r=\beta/\alpha \gtrsim \kappa_2^2$. It is common to use such a step size scheme $r = \Theta(\kappa_2^2)$ to analyze two-time-scale (S)GDA for nonconvex minimax problems \citep{jin2020local, lin2020gradient, yang2020global}. 

{\bf Remark on the parameter $\lambda$.} In our convergence analyses, we arbitrarily choose $\lambda=4$ which makes the numerical calculations easier. The value of $\lambda>0$ does not matter for the equivalence between the equation $V_\lambda(\vx^*; \vy^*)=0$ and global minimax condition  (Proposition~\ref{prop:optimality} in Appendix~\ref{sec:propositions}). Also, the choice of $\lambda$ in both theorems can be arbitrary as long as $\lambda>1$; our logic does not fall apart if other appropriate step sizes for that $\lambda$ are chosen.
That is to say, we can show that the sequence $V_\lambda(\vz_0^k)$ almost monotonically decreases, ignoring some small variance terms.

\section{Comparison with related works}
\label{sec:discussion}
\vspace*{-5pt} %
\subsection{Comparison with stochastic with-replacement setting} \label{sec:discussion:with-replacement}
\vspace*{-5pt} %
First of all, we confirm that SGDA with random reshuffling (RR) has \emph{faster} convergence rates (\ie, fewer gradient computations) than SGDA based on with-replacement sampling. 
In particular, we compare our results with the analyses on the purely stochastic minimax settings which assume that every stochastic gradient oracle is \emph{independently sampled and unbiased}: this assumption is naturally satisfied by with-replacement sampling for the finite-sum settings we consider.
To make the comparisons fair and easy, we simply let \underline{$r=\beta/\alpha=\Theta(\kappa_2^2)$, $A=0$, and $B=\tau^2$}.

\citet[Theorem~4.5]{lin2020gradient} present a convergence rate for with-replacement simSGDA with $r\!=\!\Theta(\kappa_2^2)$ run on nonconvex $\mu_2$-strongly-concave problems with a convex \emph{bounded} constraint set $\gY$ for dual variable $\vy$.
Their gradient complexity to achieve $\frac{1}{T}\!\sum_{t=1}^T \E \norm{\nabla \Phi(\vx_t)}^2\!\le\!\eps^2$ (where $T$ is the number of iterations) is written as
\(T_\eps=\gO\left(\frac{\kappa_2^2L\Delta_\Phi +\kappa_2 L^2 D^2}{\eps^2}\max\left\{1, \frac{\kappa_2\tau^2}{\eps^2}\right\}\right),\)
where $\kappa_2=L/\mu_2$, $\Delta_{\Phi} = \Phi(\vx_0) - \Phi^*$, $D=\operatorname{diam}\gY$, and $\tau^2$ is the variance of the (unbiased) stochastic gradient oracles.
Their complexity can be simplified as $\gO(\kappa_2^3\tau^2\eps^{-4})$, treating other factors as constants.
In contrast, our Theorem~\ref{thm:NCPL} has a better gradient complexity in terms of $\eps$ and $\tau$, thanks to shuffling:
\[\gO\left(\frac{\kappa_2^2 L V_\lambda(\vz_0^1)}{\eps^2} \max\left\{n, \frac{\kappa_2\tau\sqrt{n}}{\eps}\right\}\right), \tag{Ours, from Theorem~\ref{thm:NCPL}}\]
or simply $\gO(\kappa_2^3\tau\sqrt{n}\eps^{-3})$.
Thus, our gradient complexity for both simSGDA-RR \emph{and} altSGDA-RR is better than that of with-replacement simSGDA when $\eps$ is small as $\eps\le \gO(\tau/\sqrt{n}).$
Our rate has three more strengths:
(i) we do not require strong concavity in $\vy$, which is a strictly stronger assumption than requiring $\vy$-side PL condition;
(ii) we do not require the constraint set $\gY$ to be bounded;
(iii) our result can easily extend to the case of \emph{any} mini-batch sizes, whereas \citet{lin2020gradient} need a particular choice of mini-batch size $M=\gO(\kappa_2 \tau^2/\eps)$ to ensure convergence.

For nonconvex-P{\L} objectives, \citet[Theorem~3.1]{yang2022faster} provide a convergence rate for with-replacement altSGDA with $r=\Theta(\kappa_2^2)$. Their rate can be translated to a gradient complexity for achieving $\frac{1}{T}\!\sum_{t=1}^T \E \norm{\nabla \Phi(\vx_t)}^2\le \eps^2$, written as
\(\gO\left(\frac{\kappa_2^2 L V_\lambda(\vz_0)}{\eps^2} \left(1 + \frac{\kappa_2^2 V_\lambda(\vz_0)^2 \tau^2}{\Delta_\Phi \eps^2}\right)\right)\)
or simply $\gO(\kappa_2^4\tau^2\eps^{-4})$.
Therefore, our gradient complexity for both altSGDA-RR \emph{and} simSGDA-RR is better when $\eps$ is small as $\eps\le \gO(\kappa_2\tau/\sqrt{n})$.
 
For P{\L}($\Phi$)-P{\L} objectives, \citet[Theorem~3.3]{yang2020global} obtain a convergence rate for with-replacement altSGDA with $r=\Theta(\kappa_2^2)$.\footnote{Although they consider two-sided P{\L} problems, their analysis applies to P{\L}($\Phi$)-P{\L} problems as well.} They apply diminishing step sizes ($\gO(1/t)$, $t\in\sN$) to derive a gradient complexity bound $\gO\left(\frac{\kappa_1\kappa_2^4\tau^2}{\mu_1\eps^2}\right)$ to achieve $\E [V_\lambda (\vz_T)] \le \eps^2$. 
One can apply the constant step sizes depending on the total number $T$ of iterations to their analysis and derive a similar complexity with only deterioration in a logarithmic factor.
In contrast, our gradient complexity for both sim/altSGDA-RR using constant step sizes can be written as, for small enough $\eps$,
\[\tilde{\gO}\left(\frac{\kappa_1\kappa_2^3\tau\sqrt{n}}{\eps\sqrt{\mu_1}}\right). \tag{Ours, from Theorem~\ref{thm:2PL}}\]
This is a better complexity in $\eps$ and $\kappa_2$, especially when $\eps\le \tilde{\gO}\left(\kappa_2\tau/\!\sqrt{n\mu_1}\right)$, even without the requirement of diminishing step size.
\vspace*{-5pt} %

\subsection{Comparison with other works on stochastic without-replacement setting} \label{sec:discussion:without-replacement}
One of the most relevant works to this paper is \citet[Theorem~3]{das2022sampling}.
The authors obtain a similar convergence rate to us for the two-sided P{\L} objective, based on linearization of gradients, but for a dissimilar algorithm which they refer to as \emph{AGDA-RR}.
The algorithm can be also thought of as \emph{epoch-wise}-alternating SGDA-RR, whereas our algorithm (altSGDA-RR) can be called as \emph{step-wise}-alternating SGDA-RR.
In epoch $k$, their algorithm ({\it i}) performs updates only on $\vx$ ($\vx_0^k,\ldots,\vx_n^{k}$) while fixing $\vy$ to $\vy_0^k$, and then ({\it ii}) performs updates only on $\vy$ ($\vy_0^k,\ldots,\vy_n^k$) while fixing $\vx$ to $\vx_0^{k+1} = \vx_n^k$.
We believe that our step-wise algorithm is closer to practice, especially when $n$ is large.
Because of the distinction between algorithms, the proof techniques are also different.

\citet[Theorem~3]{xie2021efficient} present a convergence rate of \emph{CD-MA}, an extension of simSGDA to the cross-device federated learning setup, on nonconvex-P{\L} setting.
Their convergence result for \emph{CD-MA} also assumes mini-batch sampling by random reshuffling.
As a consequence, they yield a rate analogous to our Theorem~\ref{thm:NCPL} if we reduce their result to the single-machine setup.
Nevertheless, our convergence bound contains a term that shrinks with the number of components or mini-batches, whereas theirs does not. For a more detailed comparison, please refer to Appendix~\ref{sec:omitted comparison}.

There are also some works on RR-based (constrained) minimax optimization algorithms other than SGDA, but for convex-concave problems.
\citet{maheshwari2022zeroth} present \emph{OGDA-RR}, a gradient-free RR-based optimistic GDA algorithm.
\citet{yu2022fast} study \emph{stochastic proximal point with RR}, consisting of double-loop epochs.
Their analyses exploit convex-concavity and Lipschitz continuity of their objective, based on the arguments by \citet{nagaraj2019sgd}.
This enables a direct usage of the duality gap, the difference between primal function $\Phi(\cdot)$ and dual function $\Psi(\cdot) = \min_{\vx} f(\vx;\cdot)$,
as a criterion for optimality.
On the contrary, our work relies on a different structure of the functions, which in turn differentiates the constructions of convergence rates.

\subsection{Comparison with deterministic setting} \label{sec:discussion:deterministic}
Here, we compare our rates with (full-batch) \emph{gradient descent-ascent} (GDA):
\begin{align*}
    \bigg[~
    \begin{split}
    	&\vx_{k} = \vx_{k-1} - \alpha \nabla_1\, f(\vx_{k-1}; \vy_{k-1}), \\
    	&\vy_{k} = \vy_{k-1} + \beta \nabla_2\, f(\vx'; \vy_{k-1}),
    \end{split} 
    \qquad \text{where } \vx' =
    \begin{cases}
        \vx_{k-1}, & \text{(\emph{simGDA}), or} \\
        \vx_{k}, & \text{(\emph{altGDA}).}
    \end{cases}
\end{align*}
It uses the whole information of the objective $f$ at every iteration without any noise. For comparison with GDA, we utilize our extended theorems for arbitrary mini-batch size $b$ (Theorems~\ref{thm:NCPL 2}~and~\ref{thm:2PL 2} in Appendix~\ref{sec:minibatch}). By letting $b=n$ and matching our iterate $\vz_0^k=(\vx_0^k;\vy_0^k)$ to a GDA iterate $\vz_k=(\vx_k;\vy_k)$, our results reduce to upper convergence  bounds for simGDA and altGDA.

For nonconvex-P{\L} problems (Theorems~\ref{thm:NCPL}\,\&\,\ref{thm:NCPL 2}), the convergence rate and \emph{iteration} complexity (\ie, sufficient number of iterations $K_\eps$) become
\begin{equation}
    \min_{k\in [K]} \norm{\nabla \Phi(\vx_k)}^2 \le \gO\left(\frac{\kappa_2^2L V_\lambda(\vz_1) }{K}\right); \quad \ie,~ K_\eps = \gO\left(\frac{\kappa_2^2 L V_\lambda(\vz_1)}{\eps^2}\right), \label{eq:NCPL GDA}
\end{equation}
when $r=\Theta(\kappa_2^2)$.
This is similar to a known rate of simGDA with $r=\Theta(\kappa_2^2)$ for nonconvex-strongly-concave problems by \citet[Theorem~4.4]{lin2020gradient} as a special case. 
Their iteration complexity is written as $\gO((\kappa_2^2 L \Delta_{\Phi} + \kappa_2 L^2 D^2)/\eps^2)$,
where the symbols are already defined in Section~\ref{sec:discussion:with-replacement}. To see how the two bounds compare in terms of the factors other than $\eps$, notice that we have $\Phi(\vx)-f(\vx; \vy) \le \tfrac{L}{2}\norm{\vy-\vy^*(\vx)}^2$ for any $(\vx; \vy)$, due to the $L$-smoothness of $-f$. 
Here, $\vy^*(\vx)$ is an element of $\gY_{\vx}^* = \argmax_y f(\vx; \vy)$.
Thus, we have $V_\lambda (\vz_1) = \lambda [\Phi(\vx_1)-\Phi^*] + [\Phi(\vx_1)-f(\vz_1)] \le \lambda \Delta_{\Phi} + LD^2/2$. 
As a result, we could \emph{loosely} translate our iteration complexity \eqref{eq:NCPL GDA} to $\gO((\kappa_2^2 L \Delta_{\Phi} + \kappa_2^2 L^2 D^2)/\eps^2)$.
We suspect that the discrepancy in terms of $\kappa_2$ comes from the fact that our analysis does not require the (strong) concavity in terms of $\vy$ or a bounded constraint $\gY$: these requirements made a considerable difference in proofs.

For P{\L}($\Phi$)-P{\L} problems (Theorems~\ref{thm:2PL}\,\&\,\ref{thm:2PL 2}), the rate and iteration complexity ($K'_\eps$) become
\begin{equation}
    V_\lambda(\vz_{K+1}) \le V_\lambda(\vz_1)\cdot\exp\left(-\frac{K}{C\kappa_1\kappa_2^2}\right); \quad \ie,~ K'_\eps = \gO\left(\kappa_1\kappa_2^2 \log(1/\eps)\right) \label{eq:2PL GDA}
\end{equation}
where $r=\Theta(\kappa_2^2)$ and $C$ is a numerical constant.
This recovers the linear convergence by \citet[Theorem~3.2]{yang2020global} as a special case, where they prove convergence of altGDA with step size ratio $r=\Theta(\kappa_2^2)$ for two-sided P{\L} problem. 
Following the proof of \citep[Theorem~3.2]{yang2020global}, 
one can show that the bound \eqref{eq:2PL GDA} indeed implies the actual convergence to a global minimax point $\vz^*$,
in the sense that we can achieve $\norm{\vz_k-\vz^*}\le \eps$ in $\gO\left(\kappa_1 \kappa_2^2 \log (1/\eps)\right)$ iterations.

\section{Lower bound for (full-batch) simGDA using separate step sizes}
\label{sec:LB GDA}
As an extension of the discussion from Section~\ref{sec:discussion:deterministic}, we characterize a lower complexity bound of deterministic simGDA with separate step sizes ($\alpha,\beta$) of arbitrary ratio $r=\beta/\alpha$, for smooth strongly-convex-strongly-concave (SC-SC) cases.
Surprisingly, at least for $r \gtrsim \kappa_2^2$, our lower bound matches the upper complexity bound of GDA for a much wider class of smooth P{\L}($\Phi$)-P{\L} problems,\footnote{
strongly-convex-strongly-concave (SC-SC) $\subset$ two-sided~P{\L} $\subset$ P{\L}($\Phi$)-P{\L} $\subset$ nonconvex-P{\L}.}
which is quite surprising.

For a smooth P{\L}($\Phi$)-P{\L} problems, simGDA with at least $r=\Omega(\kappa_2^2)$ has an upper complexity bound $K=\gO(\kappa_1r\log(1/\eps))$ for a \emph{global} $\eps$-convergence $V_\lambda(\vz_K) \le \eps^2$ in terms of potential function. This means that the lowest complexity is $\gO(\kappa_1 \kappa_2^2 \log(1/\epsilon))$ achieved when $r = \Theta(\kappa_2^2)$. On the other hand, for a $L$-smooth $\mu$-SC-SC problem with saddle point $\vz^*$, it is well-known that the simGDA with a single step-size ($\alpha=\beta$) has a tight upper/lower complexity $K=\Theta(\kappa^2\log(1/\eps))$ to achieve $\norm{\vz_K-\vz^*}^2\le\eps^2$, where $\kappa = L/\mu$ (\eg, \citet[Theorem\,C.1]{das2022sampling}). The difference of complexity bounds in condition number ($\kappa_1 \kappa_2^2$ v.s.\ $\kappa^2$) is somewhat questionable because, at least in smooth minimization problems, strongly convex problems and P{\L} problems have identical gradient descent (GD) iteration complexity $\gO(\kappa\log(1/\eps))$~\citep[Theorem~1]{karimi2016linear}.

One could ask where the discrepancy in terms of $\kappa$ comes from: is it due to ({\it i}) the criteria ($V_\lambda(\vz_K)$ v.s.\ $\norm{\vz_K-\vz^*}^2$) for $\eps$-accuracy, ({\it ii}) the function classes (P{\L}($\Phi$)-P{\L} v.s.\ SC-SC), or ({\it iii}) the step size ratios ($\Omega(\kappa_2^2)$ v.s.\ 1)? We answer the question by showing the following theorem: the discrepancy in $\kappa$ comes from the step size ratio difference. We defer the proof to Appendix~\ref{sec:proofs LB GDA}.

\begin{theorem}[Lower bound, ratio-specific]\label{thm:LB}
    Consider a class $\gF(L, \mu_1, \mu_2)$ of functions $f(\vx;\vy)$ with two arguments $\vx$ and $\vy$, which is $L$-smooth, $\mu_1$-strongly-convex in $\vx$, and $\mu_2$-strongly-concave in $\vy$. Suppose $\kappa_1 = L/\mu_1\ge c$ and $\kappa_2 = L/\mu_2\ge c$ for some constant $c>1$. Then, for any step size ratio $r = \beta/\alpha >0$, there exists a function $f\in\gF(L, \mu_1, \mu_2)$ with a unique saddle point $\vz^*$, for which simGDA with any step sizes $(\alpha, \beta) = (\beta/r,\beta)$ requires at least
    \[K=\left\{\begin{array}{ll}
    \Omega\left(\kappa_1 r\log(1/\eps)\right), &\text{\rm if}~ r\ge \kappa_2/c,\\
    \Omega\left(\kappa_1\kappa_2\log(1/\eps)\right), & \text{\rm if}~ c/\kappa_1\le r\le \kappa_2/c,\\
    \Omega(\left(\kappa_2/r\right)\log(1/\eps)), & \text{\rm if}~ 0<r\le c/\kappa_1
    \end{array}\right.\]
    iterations to achieve either $\norm{\vz_K-\vz^*}^2 \le \eps^2$ or $V_\lambda(\vz_K)\le \eps^2$.
\end{theorem}

Thanks to Theorem~\ref{thm:LB}, we can say from Theorem~\ref{thm:2PL 2} that for any step size ratio $r \gtrsim \kappa_2^2$, we have a \emph{tight} upper bound on the iteration complexity $K=\gO(\kappa_1r\log(1/\eps))$ of simGDA for general P{\L}($\Phi$)-P{\L} problems. 
Note that Theorem~\ref{thm:LB} also subsumes the existing lower bound of the equal-step-size ($r = 1$) simGDA for $\mu$-SC-SC problems.

Given the tightness of bounds for $r \gtrsim \kappa_2^2$, a natural next step is to discuss $1\lesssim r \lesssim \kappa_2^2$.
Recent work by \citet{li2022convergence} also discusses the step size ratio of simGDA.
In \citet[Theorem~4.1]{li2022convergence}, the authors construct a $y$-side strongly-concave function\footnote{\label{footnote:PLPL x-concave} $g(x;y)=-\frac{L}{2}x^2+Lxy-\frac{\mu}{2} y^2$, where $L/\mu>1$: its primal function is strongly convex.
}
and show that simGDA with a step size ratio $r\le \kappa_2$ is \emph{impossible} to converge. The \emph{necessity} of $r\gtrsim\kappa_2$ implied by this theorem also applies to the P{\L}($\Phi$)-P{\L} case. Thus, there is no hope for showing an upper convergence bound of simGDA with $1\lesssim r\lesssim\kappa_2$ for general nonconvex-P{\L} problems.
We remark that their theorem does not contradict nor subsume Theorem~\ref{thm:LB} because we consider a much smaller function class (SC-SC) to construct the lower bounds.

On the \emph{sufficiency} of $r \gtrsim \kappa_2$ for convergence, \citet[Theorem~4.2]{li2022convergence} show that simGDA with $r \ge c\kappa$ (for some $c>1$) can \emph{locally} converge at the iteration complexity $\gO(\kappa_1 r \log(1/\eps))$ for some nonconvex-strongly-concave problems, which matches the bound in Theorem~\ref{thm:LB}.
Our upper bounds (Theorems~\ref{thm:NCPL 2} and \ref{thm:2PL 2}) do require $r\gtrsim \kappa_2^2$, which may look suboptimal, but we claim that our results are not necessarily weaker. One reason is that our convergence guarantee is \emph{global}, \ie, independent of the initialization.
Another reason is that their analysis is only valid when a \emph{differential Stackelberg equilibrium}\footnote{Loosely speaking, a differential Stackelberg equilibrium is a stationary point $(\vx^*; \vy^*)$ where $f(\vx^*; \cdot)$ is locally strongly concave near $\vy^*$ and $\Phi(\cdot)$ is locally strongly convex near $\vx^*$.} exists, whereas a general P{\L}($\Phi$)-P{\L} function may not have such an equilibrium (for an example, see Proposition~\ref{prop:2PL wo stackelberg} in Appendix~\ref{sec:propositions}). 

As far as we know, it is still an open problem whether a \emph{global} convergence bound for simGDA on nonconvex-P{\L} problems can be shown when the step size ratio $r$ is between $\Omega(\kappa_2)$ and $\gO(\kappa_2^2)$.%

\section{Experiments}
\label{sec:experiments}
To validate our main theoretical findings, here we present some numerical results. We focus on the primal-P{\L}-strongly-concave (or P{\L}($\Phi$)-SC, which is P{\L}($\Phi$)-P{\L} as well) quadratic games of the form
\begin{equation}
\label{eq:quadraticgame}
\begin{aligned}
    \min_{\vx \in \R^d}\max_{\vy \in \R^d} f(\vx;\vy) &\textstyle= \frac{1}{2}\vx^\top \mA \vx + \vx^\top \mB \vy - \frac{1}{2}\vy^\top\mC\vy= \frac{1}{n}\sum_{i=1}^n f_i(\vx;\vy),\\
    \text{where} \quad f_i(\vx;\vy) &\textstyle= \frac{1}{2}\vx^\top \mA_i \vx + \vx^\top \mB_i \vy - \frac{1}{2}\vy^\top\mC_i\vy + \vu_i^\top \vx- \vv_i^\top \vy.
\end{aligned}
\end{equation}
This toy example is often used to numerically evaluate the minimax algorithms \citep{yang2020global,loizou2021stochastic,das2022sampling} and appears in various domains such as AUC maximization \citep{ying2016stochastic}, policy evaluation \citep{du2017stochastic}, and imitation learning \citep{cai2019global}%

To make the game in Equation~\eqref{eq:quadraticgame} satisfy P{\L}($\Phi$)-SC and component $L$-smoothness, we should sample the coefficient matrices and vectors carefully. First, they need to be $\norm{\mA_i}_2,\norm{\mB_i}_2,\norm{\mC_i}_2\le L$ and $\sum_{i=1}^n \vu_i = \sum_{i=1}^n \vv_i = \vzero$. To make the primal function $\Phi$ a well-defined real-valued function for any $\vx\in \R^d$, we choose $\mC=\frac{1}{n}\sum_{i=1}^n \mC_i$ to be positive definite, \ie, $\mu\mI\preceq\mC$ for an identity matrix $\mI$ and $\mu>0$. Then, the primal function can be explicitly written as
\begin{align*}
    \textstyle\Phi(\vx) =\max_{\vy\in \R^d} f(\vx;\vy) =\frac{1}{2}\vx^\top \left(\mA + \mB\mC^{-1}\mB^\top\right)\vx := \frac{1}{2}\vx^\top \mM\vx.
\end{align*}
We construct a matrix $\mM:=\mA + \mB\mC^{-1}\mB^\top$ to be rank-deficient positive semi-definite. Letting the smallest nonzero eigenvalue of $\mM$ by $\mu$, we ensure that $\Phi$ is $\mu$-P{\L} but not strongly convex.  We emphasize that the objective function $f$ is not even (strongly-)convex in $\vx$ in general.

We compare six algorithms in total: simSGDA-RR, altSGDA-RR, AGDA-RR (as defined in \citet{das2022sampling}), and the with-replacement counterparts of these three algorithms. To this end, on 5 different randomly-generated quadratic games and under 2 random seeds per game (\ie, 10 runs per algorithm), we run each algorithm for the same number of epochs using constant step sizes of ratio $\beta/\alpha = c\kappa_2^2$ for some constant $c$ and $\kappa_2=L/\mu$. 

\begin{figure}[t]
    \centering
    \begin{subfigure}[b]{0.32\textwidth}
         \centering
         \includegraphics[width=\textwidth]{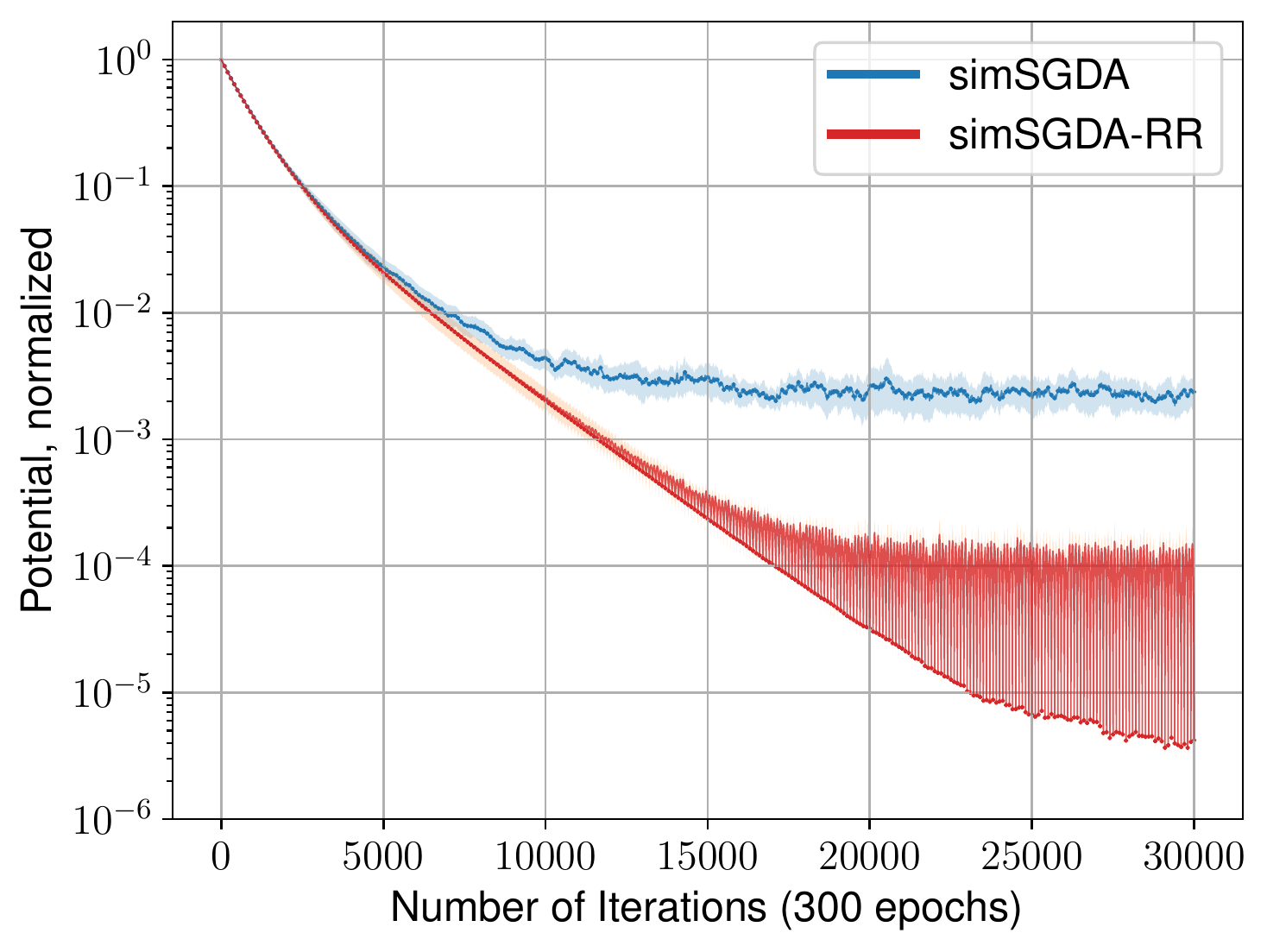}
         \caption{{\color{tab:blue}simSGDA} v.s. {\color{tab:red}simSGDA-RR}.}
         \label{fig:Experiment:simSGDA}
     \end{subfigure}
     ~
     \begin{subfigure}[b]{0.32\textwidth}
         \centering
         \includegraphics[width=\textwidth]{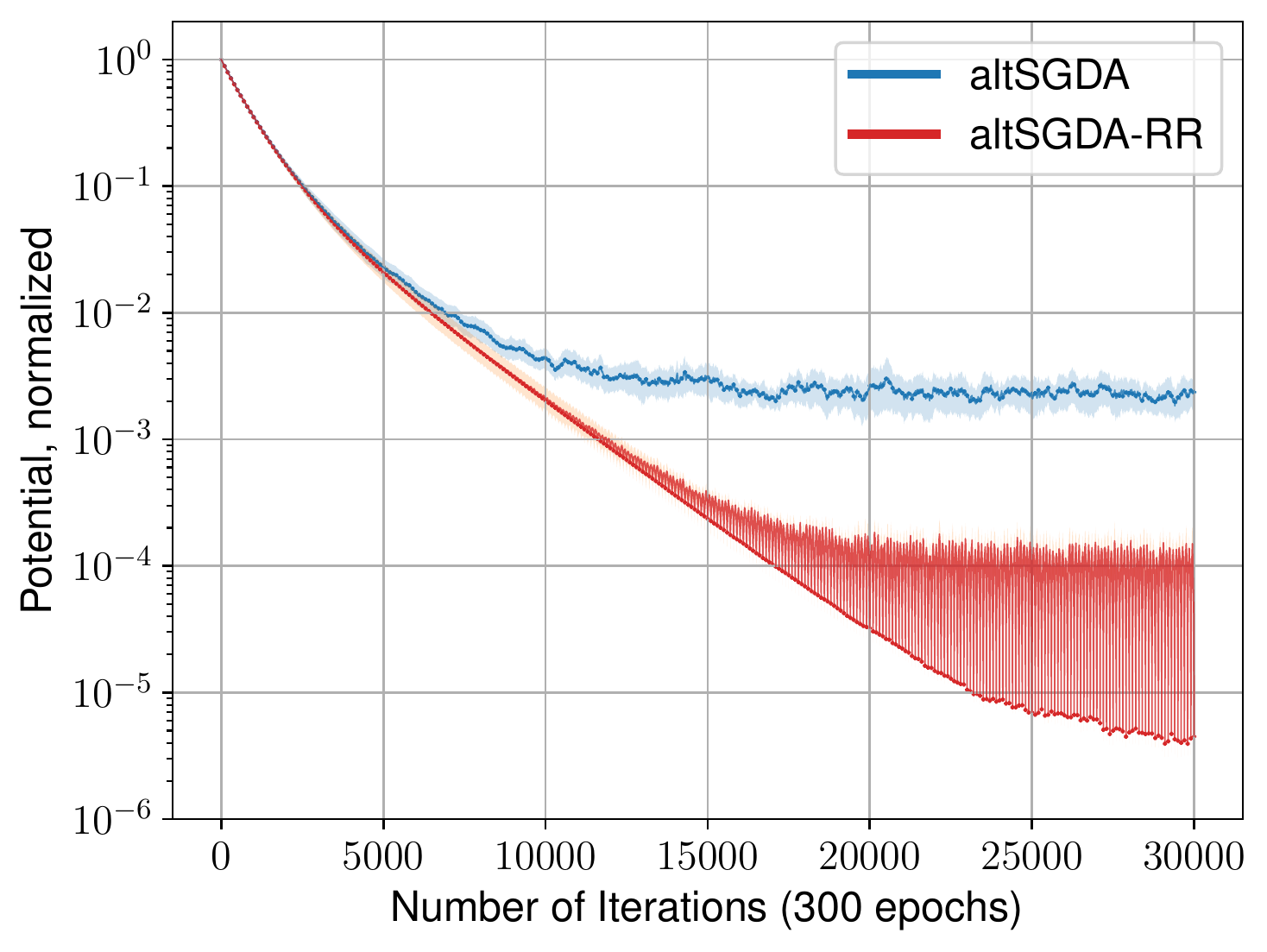}
         \caption{{\color{tab:blue}altSGDA} v.s. {\color{tab:red}altSGDA-RR}.}
         \label{fig:Experiment:altSGDA}
     \end{subfigure}
     ~
     \begin{subfigure}[b]{0.32\textwidth}
         \centering
         \includegraphics[width=\textwidth]{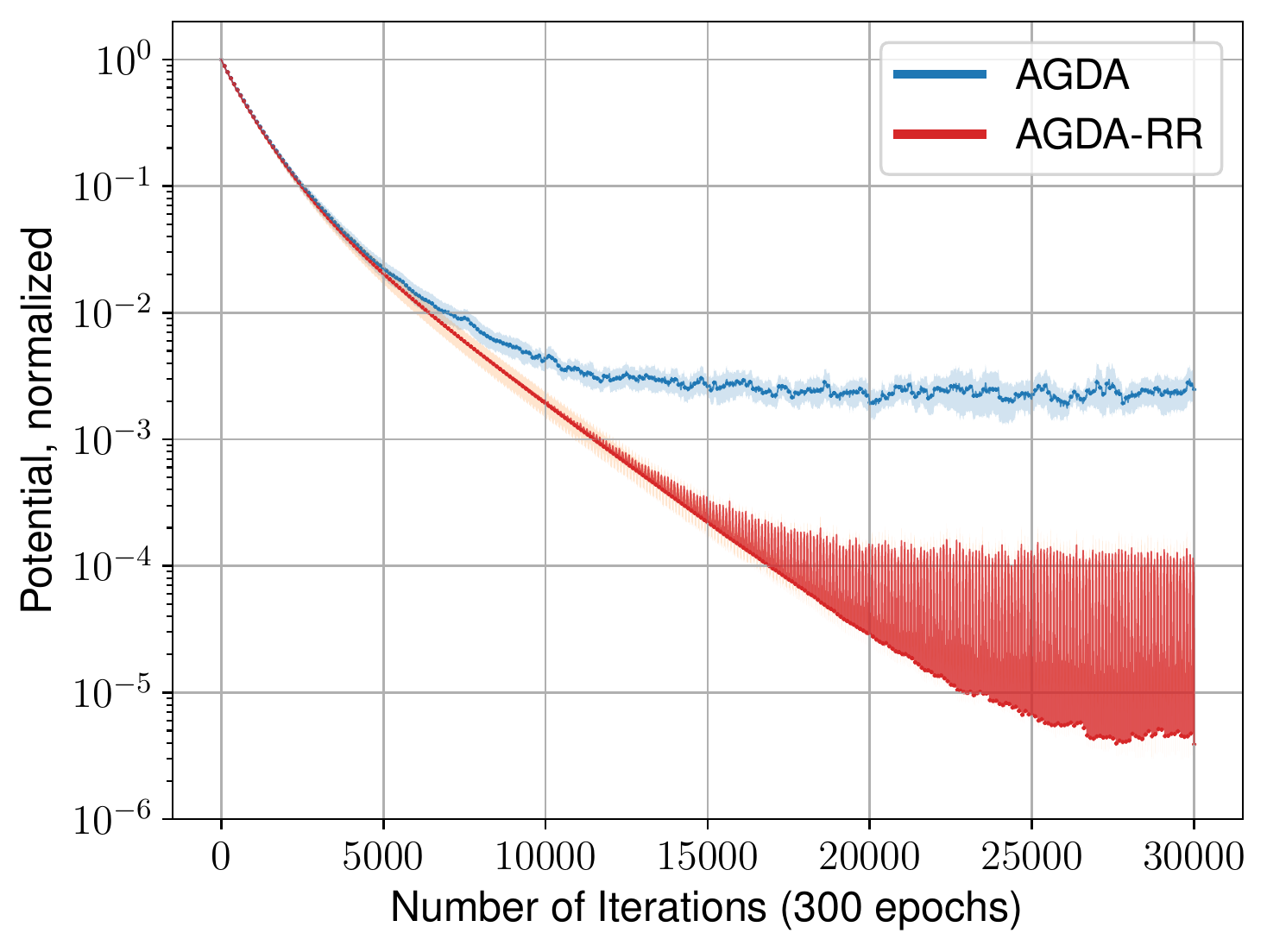}
         \caption{{\color{tab:blue}AGDA} v.s. {\color{tab:red}AGDA-RR}.}
         \label{fig:Experiment:AGDA}
     \end{subfigure}
    \caption{Experimental results on quadratic games \eqref{eq:quadraticgame}. Solid lines: average across 10 different runs. Shaded regions: 95\% confidence intervals ($\pm 1.96$ std). Dots: start/end of epochs.
    The vertical axes are on a \emph{logarithmic scale}.}
    \label{fig:Experiment}
\end{figure}

We report the potential function values ($V_\lambda$, defined in Equation~\eqref{eq:Lyapunov})
at every iteration.\footnote{As described in Section~\ref{sec:discussion:without-replacement}, AGDA-RR uses only one-side gradient ($\nabla_1$ or $\nabla_2$) at each iteration; given a fixed budget of gradient computations, it should access components twice as many times as SGDA-RR. Hence, we report the values at every other iteration of AGDA \& AGDA-RR, for a fair comparison.}
Results are presented in \Figref{fig:Experiment}: the values are normalized by dividing them by the initial value. As we discussed in Section~\ref{sec:discussion:with-replacement}, we observe that the random reshuffling considerably accelerates the convergence of the algorithms. Furthermore, all three algorithms with random reshuffling show more or less the same performance. Specifically, the plots for simSGDA ({resp.} simSGDA-RR) and altSGDA ({resp.} altSGDA-RR) are almost identical. We believe this is because we choose a random seed for each of the 10 different runs and share it across different algorithms.

Please refer to Appendix~\ref{sec:experimental details:quadratic} for more detailed construction, discussion, and comparative study of the experimental results.

\section{Conclusion}
\label{sec:conclusion}
We investigated stochastic algorithms based on without-replacement component sampling, called simSGDA-RR and altSGDA-RR, for solving smooth nonconvex finite-sum minimax optimization problems.
We established convergence rates under the $\vy$-side P{\L} condition (nonconvex-P{\L}) and, additionally, the primal P{\L} condition (P{\L}($\Phi$)-P{\L}).
We ascertain that the SGDA-RR can achieve a faster rate than its with-replacement counterpart, which agrees with the existing theory on without-replacement SGD for minimization. 
Lastly, we provided complexity lower bounds for simGDA with an arbitrarily fixed step size ratio $r$, demonstrating that the full-batch upper bound with $r\gtrsim\kappa_2^2$ for P{\L}($\Phi$)-P{\L} functions is tight.

Possible future directions include widening our results beyond sim/altSGDA (\eg, extra-gradient or optimistic GDA) and beyond RR (\eg, single/adversarial shuffling). As also discussed in Section~\ref{sec:LB GDA}, an interesting open question remains open: can we identify tight convergence rates for stochastic (with-/without-replacement) and/or deterministic GDA with step size ratio $r$ satisfying $\kappa_2\lesssim r \lesssim \kappa_2^2$, for general nonconvex-P{\L} problems?

\subsubsection*{Acknowledgments}
This work was supported by Institute of Information \& communications Technology Planning \& evaluation (IITP) grant (No.\ 2019-0-00075, Artificial Intelligence Graduate School Program (KAIST)) funded by the Korea government (MSIT).
The work was also supported by the National Research Foundation of Korea (NRF) grant (No.\ NRF-2019R1A5A1028324) funded by the Korea government (MSIT). CY acknowledges support from a grant funded by Samsung Electronics Co., Ltd.

\bibliography{iclr2023_conference}

\begin{thebibliography}{33}
\providecommand{\natexlab}[1]{#1}
\providecommand{\url}[1]{\texttt{#1}}
\expandafter\ifx\csname urlstyle\endcsname\relax
  \providecommand{\doi}[1]{doi: #1}\else
  \providecommand{\doi}{doi: \begingroup \urlstyle{rm}\Url}\fi

\bibitem[Ahn et~al.(2020)Ahn, Yun, and Sra]{ahn2020sgd}
Kwangjun Ahn, Chulhee Yun, and Suvrit Sra.
\newblock {SGD} with shuffling: optimal rates without component convexity and
  large epoch requirements.
\newblock \emph{Advances in Neural Information Processing Systems},
  33:\penalty0 17526--17535, 2020.

\bibitem[Beznosikov et~al.(2022)Beznosikov, Gorbunov, Berard, and
  Loizou]{beznosikov2022stochastic}
Aleksandr Beznosikov, Eduard Gorbunov, Hugo Berard, and Nicolas Loizou.
\newblock Stochastic gradient descent-ascent: Unified theory and new efficient
  methods.
\newblock \emph{arXiv preprint arXiv:2202.07262}, 2022.

\bibitem[Cai et~al.(2019)Cai, Hong, Chen, and Wang]{cai2019global}
Qi~Cai, Mingyi Hong, Yongxin Chen, and Zhaoran Wang.
\newblock On the global convergence of imitation learning: A case for linear
  quadratic regulator.
\newblock \emph{arXiv preprint arXiv:1901.03674}, 2019.

\bibitem[Das et~al.(2022)Das, Sch{\"o}lkopf, and Muehlebach]{das2022sampling}
Aniket Das, Bernhard Sch{\"o}lkopf, and Michael Muehlebach.
\newblock Sampling without replacement leads to faster rates in finite-sum
  minimax optimization.
\newblock \emph{arXiv preprint arXiv:2206.02953}, 2022.

\bibitem[Du et~al.(2017)Du, Chen, Li, Xiao, and Zhou]{du2017stochastic}
Simon~S Du, Jianshu Chen, Lihong Li, Lin Xiao, and Dengyong Zhou.
\newblock Stochastic variance reduction methods for policy evaluation.
\newblock In \emph{International Conference on Machine Learning}, pp.\
  1049--1058. PMLR, 2017.

\bibitem[Goodfellow et~al.(2020)Goodfellow, Pouget-Abadie, Mirza, Xu,
  Warde-Farley, Ozair, Courville, and Bengio]{goodfellow2020generative}
Ian Goodfellow, Jean Pouget-Abadie, Mehdi Mirza, Bing Xu, David Warde-Farley,
  Sherjil Ozair, Aaron Courville, and Yoshua Bengio.
\newblock Generative adversarial networks.
\newblock \emph{Communications of the ACM}, 63\penalty0 (11):\penalty0
  139--144, 2020.

\bibitem[Guo et~al.(2020)Guo, Yuan, Yan, and Yang]{guo2020fast}
Zhishuai Guo, Zhuoning Yuan, Yan Yan, and Tianbao Yang.
\newblock Fast objective \& duality gap convergence for
  nonconvex-strongly-concave min-max problems.
\newblock \emph{arXiv preprint arXiv:2006.06889}, 2020.

\bibitem[Heusel et~al.(2017)Heusel, Ramsauer, Unterthiner, Nessler, and
  Hochreiter]{heusel2017gans}
Martin Heusel, Hubert Ramsauer, Thomas Unterthiner, Bernhard Nessler, and Sepp
  Hochreiter.
\newblock {GAN}s trained by a two time-scale update rule converge to a local
  nash equilibrium.
\newblock \emph{Advances in neural information processing systems}, 30, 2017.

\bibitem[Horn \& Johnson(2012)Horn and Johnson]{Horn2012-du}
Roger~A Horn and Charles~R Johnson.
\newblock \emph{Matrix Analysis}.
\newblock Cambridge University Press, Cambridge, England, 2 edition, October
  2012.

\bibitem[Jin et~al.(2020)Jin, Netrapalli, and Jordan]{jin2020local}
Chi Jin, Praneeth Netrapalli, and Michael Jordan.
\newblock What is local optimality in nonconvex-nonconcave minimax
  optimization?
\newblock In \emph{International conference on machine learning}, pp.\
  4880--4889. PMLR, 2020.

\bibitem[Karimi et~al.(2016)Karimi, Nutini, and Schmidt]{karimi2016linear}
Hamed Karimi, Julie Nutini, and Mark Schmidt.
\newblock Linear convergence of gradient and proximal-gradient methods under
  the {P}olyak-{\l}ojasiewicz condition.
\newblock In \emph{Joint European conference on machine learning and knowledge
  discovery in databases}, pp.\  795--811. Springer, 2016.

\bibitem[Li et~al.(2022)Li, Farnia, Das, and Jadbabaie]{li2022convergence}
Haochuan Li, Farzan Farnia, Subhro Das, and Ali Jadbabaie.
\newblock On convergence of gradient descent ascent: A tight local analysis.
\newblock In \emph{International Conference on Machine Learning}, pp.\
  12717--12740. PMLR, 2022.

\bibitem[Li et~al.(2019)Li, Wu, Cui, Dong, Fang, and Russell]{li2019robust}
Shihui Li, Yi~Wu, Xinyue Cui, Honghua Dong, Fei Fang, and Stuart Russell.
\newblock Robust multi-agent reinforcement learning via minimax deep
  deterministic policy gradient.
\newblock In \emph{Proceedings of the AAAI Conference on Artificial
  Intelligence}, volume~33, pp.\  4213--4220, 2019.

\bibitem[Lin et~al.(2020)Lin, Jin, and Jordan]{lin2020gradient}
Tianyi Lin, Chi Jin, and Michael Jordan.
\newblock On gradient descent ascent for nonconvex-concave minimax problems.
\newblock In \emph{International Conference on Machine Learning}, pp.\
  6083--6093. PMLR, 2020.

\bibitem[Liu et~al.(2020)Liu, Yuan, Ying, and Yang]{Liu2020Stochastic}
Mingrui Liu, Zhuoning Yuan, Yiming Ying, and Tianbao Yang.
\newblock Stochastic {AUC} maximization with deep neural networks.
\newblock In \emph{International Conference on Learning Representations}, 2020.
\newblock URL \url{https://openreview.net/forum?id=HJepXaVYDr}.

\bibitem[Loizou et~al.(2021)Loizou, Berard, Gidel, Mitliagkas, and
  Lacoste-Julien]{loizou2021stochastic}
Nicolas Loizou, Hugo Berard, Gauthier Gidel, Ioannis Mitliagkas, and Simon
  Lacoste-Julien.
\newblock Stochastic gradient descent-ascent and consensus optimization for
  smooth games: Convergence analysis under expected co-coercivity.
\newblock \emph{Advances in Neural Information Processing Systems},
  34:\penalty0 19095--19108, 2021.

\bibitem[Madry et~al.(2018)Madry, Makelov, Schmidt, Tsipras, and
  Vladu]{madry2018towards}
Aleksander Madry, Aleksandar Makelov, Ludwig Schmidt, Dimitris Tsipras, and
  Adrian Vladu.
\newblock Towards deep learning models resistant to adversarial attacks.
\newblock In \emph{International Conference on Learning Representations}.
  OpenReview.net, 2018.

\bibitem[Maheshwari et~al.(2022)Maheshwari, Chiu, Mazumdar, Sastry, and
  Ratliff]{maheshwari2022zeroth}
Chinmay Maheshwari, Chih-Yuan Chiu, Eric Mazumdar, Shankar Sastry, and Lillian
  Ratliff.
\newblock Zeroth-order methods for convex-concave min-max problems:
  Applications to decision-dependent risk minimization.
\newblock In \emph{International Conference on Artificial Intelligence and
  Statistics}, pp.\  6702--6734. PMLR, 2022.

\bibitem[Mishchenko et~al.(2020)Mishchenko, Khaled, and
  Richt{\'a}rik]{mishchenko2020random}
Konstantin Mishchenko, Ahmed Khaled, and Peter Richt{\'a}rik.
\newblock Random reshuffling: Simple analysis with vast improvements.
\newblock \emph{Advances in Neural Information Processing Systems},
  33:\penalty0 17309--17320, 2020.

\bibitem[Nagaraj et~al.(2019)Nagaraj, Jain, and Netrapalli]{nagaraj2019sgd}
Dheeraj Nagaraj, Prateek Jain, and Praneeth Netrapalli.
\newblock {SGD} without replacement: Sharper rates for general smooth convex
  functions.
\newblock In \emph{International Conference on Machine Learning (ICML)}, pp.\
  4703--4711. PMLR, 2019.

\bibitem[Nguyen et~al.(2021)Nguyen, Tran-Dinh, Phan, Nguyen, and
  Van~Dijk]{nguyen2021unified}
Lam~M Nguyen, Quoc Tran-Dinh, Dzung~T Phan, Phuong~Ha Nguyen, and Marten
  Van~Dijk.
\newblock A unified convergence analysis for shuffling-type gradient methods.
\newblock \emph{The Journal of Machine Learning Research}, 22\penalty0
  (1):\penalty0 9397--9440, 2021.

\bibitem[Nouiehed et~al.(2019)Nouiehed, Sanjabi, Huang, Lee, and
  Razaviyayn]{nouiehed2019solving}
Maher Nouiehed, Maziar Sanjabi, Tianjian Huang, Jason~D Lee, and Meisam
  Razaviyayn.
\newblock Solving a class of non-convex min-max games using iterative first
  order methods.
\newblock \emph{Advances in Neural Information Processing Systems (NeurIPS)},
  32, 2019.

\bibitem[Rajput et~al.(2020)Rajput, Gupta, and
  Papailiopoulos]{rajput2020closing}
Shashank Rajput, Anant Gupta, and Dimitris Papailiopoulos.
\newblock Closing the convergence gap of {SGD} without replacement.
\newblock In \emph{International Conference on Machine Learning}, pp.\
  7964--7973. PMLR, 2020.

\bibitem[Sinha et~al.(2018)Sinha, Namkoong, and Duchi]{sinha2018certifiable}
Aman Sinha, Hongseok Namkoong, and John Duchi.
\newblock Certifiable distributional robustness with principled adversarial
  training.
\newblock In \emph{International Conference on Learning Representations}, 2018.
\newblock URL \url{https://openreview.net/forum?id=Hk6kPgZA-}.

\bibitem[Xie et~al.(2021)Xie, Zhang, Shen, Liu, and Qian]{xie2021efficient}
Jiahao Xie, Chao Zhang, Zebang Shen, Weijie Liu, and Hui Qian.
\newblock Efficient cross-device federated learning algorithms for minimax
  problems.
\newblock \emph{arXiv e-prints}, pp.\  arXiv--2105, 2021.

\bibitem[Yan et~al.(2020)Yan, Xu, Lin, Liu, and Yang]{yan2020optimal}
Yan Yan, Yi~Xu, Qihang Lin, Wei Liu, and Tianbao Yang.
\newblock Optimal epoch stochastic gradient descent ascent methods for min-max
  optimization.
\newblock \emph{Advances in Neural Information Processing Systems},
  33:\penalty0 5789--5800, 2020.

\bibitem[Yang et~al.(2020)Yang, Kiyavash, and He]{yang2020global}
Junchi Yang, Negar Kiyavash, and Niao He.
\newblock Global convergence and variance reduction for a class of
  nonconvex-nonconcave minimax problems.
\newblock \emph{Advances in Neural Information Processing Systems},
  33:\penalty0 1153--1165, 2020.

\bibitem[Yang et~al.(2022)Yang, Orvieto, Lucchi, and He]{yang2022faster}
Junchi Yang, Antonio Orvieto, Aurelien Lucchi, and Niao He.
\newblock Faster single-loop algorithms for minimax optimization without strong
  concavity.
\newblock In \emph{International Conference on Artificial Intelligence and
  Statistics}, pp.\  5485--5517. PMLR, 2022.

\bibitem[Ying et~al.(2016)Ying, Wen, and Lyu]{ying2016stochastic}
Yiming Ying, Longyin Wen, and Siwei Lyu.
\newblock Stochastic online {AUC} maximization.
\newblock \emph{Advances in neural information processing systems}, 29, 2016.

\bibitem[Yu et~al.(2022)Yu, Lin, Mazumdar, and Jordan]{yu2022fast}
Yaodong Yu, Tianyi Lin, Eric~V Mazumdar, and Michael Jordan.
\newblock Fast distributionally robust learning with variance-reduced min-max
  optimization.
\newblock In \emph{International Conference on Artificial Intelligence and
  Statistics}, pp.\  1219--1250. PMLR, 2022.

\bibitem[Yuan et~al.(2021)Yuan, Yan, Sonka, and Yang]{yuan2021large}
Zhuoning Yuan, Yan Yan, Milan Sonka, and Tianbao Yang.
\newblock Large-scale robust deep {AUC} maximization: A new surrogate loss and
  empirical studies on medical image classification.
\newblock In \emph{Proceedings of the IEEE/CVF International Conference on
  Computer Vision}, pp.\  3040--3049, 2021.

\bibitem[Yun et~al.(2021)Yun, Sra, and Jadbabaie]{yun2021open}
Chulhee Yun, Suvrit Sra, and Ali Jadbabaie.
\newblock Open problem: Can single-shuffle {SGD} be better than reshuffling
  {SGD} and {GD}?
\newblock In \emph{Conference on Learning Theory}, pp.\  4653--4658. PMLR,
  2021.

\bibitem[Yun et~al.(2022)Yun, Rajput, and Sra]{yun2022minibatch}
Chulhee Yun, Shashank Rajput, and Suvrit Sra.
\newblock Minibatch vs local {SGD} with shuffling: Tight convergence bounds and
  beyond.
\newblock In \emph{International Conference on Learning Representations}.
  OpenReview.net, 2022.

\end{thebibliography}
\bibliographystyle{iclr2023_conference}

\newpage
\tableofcontents
\newpage

\appendix

\section{Mini-batch SGDA-RR and convergence rates}
\label{sec:minibatch}
In this appendix, we present an algorithm that extends simSGDA-RR and altSGDA-RR by using mini-batches of size $b\ge 1$. For simplicity, we assume that the number of components $n$ is an integer multiple of the mini-batch size $b$ in our analysis; \ie, $n=bq$ for some integer $q\ge1$. One can extend this to the case when $n$ is not necessarily a multiple of $b$ (\eg, $n=b(q-1)+s$, where $q\ge 1$, $s\in [b]$) so that there are $q-1$ mini-batches of size $b$ and one more mini-batch of size $s\le b$.

\begin{algorithm}[h]
\caption{{\color{violet}Mini-batch} {\color{teal}sim}SGDA/{\color{brown}alt}SGDA-{\color{purple} RR}}
\label{alg:SGDA wo replacement minibatch}
\begin{algorithmic}[1]
\State {\bf Given:} The number of components $n={\color{violet} b}(q-1)+s$ ($q$: number of iterations per epoch); {\color{violet} mini-batch size $b$}; the number of epochs $K$; step sizes $\alpha, \beta>0$ 
\State {\bf Initialize:} $(\vx_0^1; \vy_0^1)\in \R^{d_x}\times\R^{d_y}$
\For{$k\in [K]$}
    \State Sample $\sigma_k \sim {\color{purple}\unif(\sS_n)}$ \Comment RR: uniformly randomly shuffle the indices every epoch
    \For{$t\in[q]$}
        \State ${\color{violet}\gB^k_t := \{\sigma_k(j):b(t-1)<j\le bt, j\in [n]\}}$ \Comment Mini-batch : a set of component indices
        \vspace{3pt}
        \State $\vx_{t}^k = \vx_{t-1}^k - \frac{\alpha}{\color{violet}b} \sum_{i\in {\color{violet}\gB^k_t}} \nabla_1 f_{i}(\vx_{t-1}^k; \vy_{t-1}^k)$
        \If{{\color{teal}simSGDA}-RR}
        \State $\vy_{t}^k = \vy_{t-1}^k + \frac{\beta}{\color{violet}b} \sum_{i\in {\color{violet}\gB^k_t}} \nabla_2 f_{i}({\color{teal}\vx_{t-1}^k}; \vy_{t-1}^k)$ \Comment simultaneous update: $\vx$ \& $\vy$ 
        \ElsIf{{\color{brown}altSGDA}-RR}
        \State $\vy_{t}^k = \vy_{t-1}^k + \frac{\beta}{\color{violet}b} \sum_{i\in {\color{violet}\gB^k_t}} \nabla_2 f_{i}({\color{brown}\vx_{t}^k}; \vy_{t-1}^k)$ \Comment alternating update: $\vx \rightarrow \vy$
        \EndIf
    \EndFor
    \State $(\vx_0^{k+1}; \vy_0^{k+1}) = (\vx_{n/b}^k; \vy_{n/b}^k)$
\EndFor
\end{algorithmic}
\end{algorithm}

Next, we illustrate the generalized versions of our main results (Theorems~\ref{thm:NCPL} and \ref{thm:2PL}) for \Algref{alg:SGDA wo replacement minibatch} with mini-batches of size $b\ge1$. Let us assume $n\ge 2$ because the case $n=1$ trivially boils down to simGDA or altGDA. We defer the proofs for simultaneous updates to Appendix~\ref{sec:proofs simSGDA RR}. We present the parts that change in the proof for alternating updates in Appendix~\ref{sec:proofs altSGDA RR}.

\begin{theorem}[Nonconvex-P{\L}, mini-batch SGDA-RR] \label{thm:NCPL 2}
	Suppose $f$ satisfies Assumptions~\ref{ass:smooth},~\ref{ass:bddvar},~\ref{ass:primal,dual},~and~\ref{ass:NCPL}. 
	Let $\lambda=4$. Choose the step sizes $\alpha$ and $\beta$ by $\alpha=\beta/r$ for some $r\ge 14\kappa_2^2$ and
	\[\beta = b\cdot\min\left\{\frac{1}{6Ln \sqrt{1 + \frac{n-b}{n-1}\cdot\frac{A}{n}}},\,\, \frac{1}{c}\left( \frac{V_\lambda(\vz_0^1)}{Ln^2(\frac{n-b}{n-1})BK} \right)^{\frac{1}{3}}\right\},\]
	for some numerical constant $c>0$. Then, mini-batch simSGDA-RR and altSGDA-RR with mini-batch size $b$ (a divisor of $n$) satisfy
	\begin{align*}
		\frac{1}{K}\sum_{k=1}^K \E\norm{\nabla \Phi (\vx_0^k)}^2 \le \frac{6r L V_\lambda(\vz_0^1)}{K}\sqrt{1+\left( \frac{n-b}{n-1} \right)\frac{A}{n}} + 2cr \left(\frac{L^2 B\, V_\lambda(\vz_0^1)^2}{nK^2}\cdot\frac{n-b}{n-1}\right)^{1/3}.
	\end{align*}
\end{theorem}

\begin{theorem}[P{\L}($\Phi$)-P{\L}, mini-batch SGDA-RR] \label{thm:2PL 2}
	Suppose $f$ satisfies Assumptions~\ref{ass:smooth},~\ref{ass:bddvar},~\ref{ass:primal,dual},~\ref{ass:NCPL},~and~\ref{ass:primalPL}. Let $\lambda=4$. Choose the step sizes $\alpha$ and $\beta$ by $\alpha=\beta/r$ for some $r\ge 14\kappa_2^2$ and
	\[\beta = b\cdot \min \left\{ \frac{1}{6Ln \sqrt{1 + \frac{n-b}{n-1}\cdot\frac{A}{n}}},\,\, \frac{2r}{\mu_1 n K} \max\left\{1, \,\,\log \left(\frac{V_\lambda(\vz_0^1) \mu_1 n K^2}{8c^3 \kappa_1^2r^3 \left( \frac{n-b}{n-1}\right)B}\right)\right\}\right\},\]
	for some numerical constant $c>0$. Then, mini-batch simSGDA-RR and altSGDA-RR with mini-batch size $b$ (a divisor of $n$) satisfy
	\[\E[V_\lambda(\vz_0^{K+1})] \le \gO\left(V_\lambda(\vz_0^1)\cdot\exp\left(-\frac{K}{12\kappa_1r\sqrt{1+ \frac{n-b}{n-1}\frac{A}{n}}}\right)\right)+ \tilde{\gO}\left(\frac{\kappa_1^2r^3 B}{\mu_1 n K^2}\right) \cdot \frac{n-b}{n-1}.\]
\end{theorem}

As a side remark, some works consider a sampling method called \emph{$b$-minibatch sampling} where all the elements in each mini-batch are distinct (\ie, without-replacement component sampling \underline{per mini-batch}), \eg, \citet[Definition~2.1]{loizou2021stochastic}. However, there is a significant gap between this method and ours: any two distinct mini-batches sampled by the $b$-minibatch sampling can intersect with each other (\ie, mini-batches are sampled with replacement), whereas, in each epoch of our \Algref{alg:SGDA wo replacement minibatch}, all the mini-batches are mutually disjoint.

\section{Technical propositions}
\label{sec:propositions}
{\noindent \bf Notation.} Throughout this appendix, we use $\gX=\R^{d_x}$ and $\gY=\R^{d_y}$. Given a closed set $\gS\subset \R^d$, we denote the set of all projection(s) of $\vv\in \R^d$ onto $\gS$, \ie, the nearest point(s) in $\gS$ from $\vv$, by $\Pi_\gS (\vv) := \argmin_{\vw\in\gS} \norm{\vv-\vw}$.

\subsection{Function classes: P{\L} condition, smoothness, and more}

\begin{proposition}[$\kappa\ge 1$]\label{prop:kappa}
	Let $g$ be an $L$-smooth function which is bounded below by $g^*$. Then, for any $\vx$,
	\[\norm{\nabla g(\vx)}^2 \le 2L\left[g(\vx)-g^*\right]. \]
	If $g$ is $\mu$-P{\L} as well, then $\mu\le L$. Consequently, the condition number $\kappa:=L/\mu$ of $g$ is $\ge 1$.
\end{proposition}
\begin{proof}
	Since $g$ is $L$-smooth, for any $\vx$ and $\vy$,
	\begin{equation}
		g^* \le g(\vy) \le g(\vx)+\inner{\nabla g(\vx), \vy-\vx}+\frac{L}{2}\norm{\vy-\vx}^2. \label{eq:lower bdd smooth}
	\end{equation}
	Now define a convex quadratic function $h_x(\vy)$ of $\vy$ as 
	\[h_x(\vy):=g(\vx)+\inner{\nabla g(\vx), \vy-\vx}+\frac{L}{2}\norm{\vy-\vx}^2.\]
	Since its gradient is
	\[\nabla h_x (\vy) = \nabla g(\vx) + L(\vy-\vx),\]
	$\vy^* := \vx-\frac{1}{L}\nabla g(\vx)$ is a minimum of $h_x$. 
	Plugging $\vy=\vy^*$ to the equation~\eqref{eq:lower bdd smooth}, we get
	\[g^*\le g(\vx)+\inner{\nabla g(\vx), -\frac{1}{L}\nabla g(\vx)} + \frac{L}{2}\norm{-\frac{1}{L}\nabla g(\vx)}^2 = g(\vx)-\frac{1}{2L}\norm{\nabla g(\vx)}^2.\]
	Rearranging the terms,
	\[\norm{\nabla g(\vx)}^2 \le 2L \left[g(\vx)-g^*\right].\]
	If we additionally utilize P{\L} inequality with $g^* := \min g(\vx)$,
	\[\norm{\nabla g(\vx)}^2 \ge 2\mu \left[g(\vx)-g^*\right],\]
	we directly yield $\mu\le L$ and thus $\kappa=L/\mu\ge 1$.
\end{proof}

\begin{definition}[\cite{karimi2016linear}] \label{def:PL related}
    Consider $g:\gX \rightarrow \R$. Let $\vx_p \in \Pi_{\gX^*}(\vx)$ be a projection of $\vx$ onto the optimal set $\gX^* = \argmin_{\vx\in \gX} g(\vx)$.
    \begin{enumerate}[label={(\arabic*)}]
        \item \label{def:SC} We say $g$ satisfies $\mu$-strong convexity (SC) if $g(\vx')\ge g(\vx) + \inner{\nabla g (\vx), \vx'-\vx} + \frac{\mu}{2}\norm{\vx'-\vx}^2$ for any $\vx, \vx'\in\gX$.
        \item \label{def:RSI} We say $g$ satisfies $\mu$-restricted secant inequality (RSI) if $\inner{\nabla g (\vx), \vx-\vx_p} \ge \mu \norm{\vx_p -\vx}^2$ for any $\vx\in\gX$.
		\item \label{def:EB} We say $g$ satisfies $\mu$-error bound (EB) condition if $\norm{\nabla g(\vx)}\ge \mu \norm{\vx_p -\vx}$ for any $\vx\in\gX$.
		\item \label{def:QG} We say $g$ satisfies $\mu$-quadratic growth (QG) condition if $g(\vx)\!-\!\min_{\vx'} g(\vx')\ge \frac{\mu}{2}\norm{\vx_p\!-\!\vx}^2$ for any $\vx\in\gX$.
	\end{enumerate}
\end{definition}

\begin{proposition} \label{prop:PLQGEB} 
    From Definition~\ref{def:PL related}, The following implications are true.
    \begin{itemize}
        \item $\mu$-SC implies $\mu$-P{\L} and $\mu$-RSI.
        \item $\mu$-P{\L} implies $\mu$-QG and $\mu$-EB.
		\item $\mu$-RSI implies $\mu$-EB.
		\item $\mu$-EB and $L$-smoothness together imply $(\mu^2/L)$-P{\L}.
	\end{itemize}
\end{proposition}
\begin{proof}
    Most of the proofs originated from \citet[Theorem~2]{karimi2016linear}.
    
    (SC $\Rightarrow$ P{\L}) Substitute $\vx$ to $\vx_p$ and $\vx'$ to $\vx$, respectively, from Definition~\ref{def:PL related}.\ref{def:SC}.
    
    (P{\L} $\Rightarrow$ QG \& EB) See the proof in \citet[Theorem~2]{karimi2016linear}
    
    (SC $\Rightarrow$ RSI) We know $\mu$-SC $\Rightarrow$ $\mu$-P{\L} $\Rightarrow$ $\mu$-QG. From Definition~\ref{def:PL related}.\ref{def:SC} \& \ref{def:PL related}.\ref{def:QG},
    \begin{align*}
        \inner{\nabla g (\vx), \vx-\vx_p} &\stackrel{\text{SC}}{\ge} g(\vx) - g(\vx_p) + \frac{\mu}{2}\norm{\vx_p-\vx}^2 \\
        &\stackrel{\text{QG}}{\ge} \frac{\mu}{2}\norm{\vx_p-\vx}^2 + \frac{\mu}{2}\norm{\vx_p-\vx}^2 = \mu\norm{\vx_p-\vx}^2.
    \end{align*}
    This implies $\mu$-RSI.
    
    (RSI $\Rightarrow$ EB) See the proof in \citet[Theorem~2]{karimi2016linear}.
    
    (EB \& smooth $\Rightarrow$ P{\L}) We use $\nabla g(\vx_p)=\bm0$. By $L$-smoothness and $\mu$-EB condition,
    \begin{align*}
        g(\vx) - g(\vx_p) &\stackrel{\text{smooth}}{\le}\inner{\nabla g(\vx_p), \vx-\vx_p}+\frac{L}{2}\norm{\vx-\vx_p}^2 = \frac{L}{2}\norm{\vx-\vx_p}^2\\
        &\stackrel{~~\text{EB}~~}{\le} \frac{L}{2\mu^2}\norm{\nabla g(\vx)}^2.
    \end{align*}
    This implies $(\mu^2/L)$-P{\L} condition on $g$.
\end{proof}

\begin{proposition}[Lipschitz continuity-like property of $\vy^*(\vx)$]\label{prop:y*(x) Lipschitz}
    For an $L$-smooth function $g:\gX\times\gY \rightarrow \R$, suppose $-g(\vx; \cdot)$ is $\mu_2$-P{\L}. Let $\kappa_2 = L/\mu_2$.
    
    Consider any $\vx_0, \vx_1\in \gX$. For any $\vy^*_0 \in \gY^*_{\vx_0}=\argmax_{\vy\in \gY} g(\vx_0; \vy)$, there exists a $\vy^*_1 \in \gY_{\vx_1}^*=\argmax_{\vy\in \gY} g(\vx_1; \vy)$ such that $\norm{\vy^*_0 - \vy^*_1}\le \kappa_2\norm{\vx_0-\vx_1}$.
    
    In fact, it is enough to choose $\vy^*_1$ as a projection of $\vy^*_0$ onto the set $\gY_{\vx_1}^*$, namely, $\vy^*_1 \in \Pi_{\gY_{\vx_1}^*} (\vy^*_0).$
\end{proposition}
\begin{proof}
    We borrow the proof from \citet[Lemma~A.3]{nouiehed2019solving}.
    
    Recall $\Phi(\vx):=\max_{\vy'\in\gY} g(\vx; \vy')$. By P{\L} inequality and smoothness of $g$, 
    \begin{align*}
        2\mu_2 \left(\Phi(\vx_1)-g(\vx_1; \vy^*_0)\right) &\le \norm{\nabla_2 g(\vx_1; \vy^*_0)}^2 \\
        &= \norm{\nabla_2 g(\vx_1; \vy^*_0) - \nabla_2 g(\vx_0; \vy^*_0)}^2 \le L^2 \norm{\vx_1-\vx_0}^2.
    \end{align*}
    The second equality applies $\nabla_2\, g(\vx_0; \vy^*_0)=\bm0$, since $\vy^*_0\in \argmax_{\vy} g(\vx_0; \vy)$.
    
    Moreover, note that $-g(\vx_1; \cdot)$ satisfies $\mu_2$-QG condition ($\because$ Proposition~\ref{prop:PLQGEB}). To apply this, we utilize our choice of $\vy^*_1$:
    \[\Phi(\vx_1)-g(\vx_1; \vy^*_0) \ge \frac{\mu_2}{2}\norm{\vy^*_1-\vy^*_0}^2.\]
    As a result, we have $\mu_2^2\norm{\vy^*_0 - \vy^*_1}^2\le L^2\norm{\vx_0-\vx_1}^2$. This completes the proof.
\end{proof}

\begin{proposition}[Smoothness of primal function]\label{prop:Phi smooth}
	Consider the same function $g$ as Proposition~\ref{prop:y*(x) Lipschitz}. Then, the function $\Phi(\vx):=\max_{\vy'\in\gY} g(\vx; \vy')$ is differentiable with 
	\[\nabla \Phi(\vx) = \nabla_1\, g(\vx; \vy^*(\vx)), \quad \text{regardless of the choice of }~ \vy^*(\vx)\in \argmax_{\vy'\in\gY} g(\vx; \vy').\]
	Moreover, $\Phi$ is $L(\kappa_2 +1)$-smooth, where $\kappa_2 = L/\mu_2$.
\end{proposition}

\begin{proof}
    This is already proved in Lemma~A.5 of \cite{nouiehed2019solving}. However, we present a bit different proof without using second-order Taylor expansion. To start, recall $\gY_\vx^* := \argmax_{\vy\in \gY} g(\vx; \vy)$. That is, we could choose any $\vy^*(\vx)\in \gY_\vx^*$.
    
    We first show the differentiability of $\Phi$. Fix a unit vector $\vu\in \gX=\R^{d_x}$: $\norm{\vu}=1$. Let any $h>0$. We first claim that there exists a path $\vp:(-h,h]\rightarrow\gY=\R^{d_y}$ which is continuous at $t=0$ and $\vp(t) \in \gY_{(\vx+t\vu)}^*$. In fact, let $\vp(t)$ be a projection of $\vy^*(\vx)$ (that we chose) onto the set $\gY_{(\vx+t\vu)}^*$. Then, $\vp(0)=\vy^*(\vx)$, and by Proposition~\ref{prop:y*(x) Lipschitz}, we have $\norm{\vp(0)-\vp(t)} \le \kappa_2 \norm{\vx - (\vx+t\vu)} = \kappa_2 t$. This shows the continuity of $\vp(t)$ at $t=0$. Now, note that there exists a $t_1\in (0,h)$ such that,
    \begin{align*}
        &\Phi(\vx + h\vu) - \Phi(\vx)\\
        &= g(\vx+h\vu; \vp(h)) - g(\vx; \vp(0))\\
        &= \big\{g(\vx+h\vu; \vp(h)) -g(\vx+h\vu; \vp(0))\big\} + \big\{g(\vx+h\vu; \vp(0))- g(\vx; \vp(0))\big\} \\
        &\ge 0 + \inner{\nabla_1 g(\vx+t_1\vu; \vp(0)), h\vu},
    \end{align*}
    by mean value theorem (applied to the first argument). We have the inequality at the last line because $g(\vx+h\vu; \vp(h)) \ge g(\vx+h\vu; \vp(0))$, since $\vp(h) \in \gY_{(x+hu)}^*$. With a similar logic, there exists a $t_2\in (0,h)$ such that,
    \begin{align*}
        &\Phi(\vx + h\vu) - \Phi(\vx)\\
        &= g(\vx+h\vu; \vp(h)) - g(\vx; \vp(0))\\
        &= \big\{g(\vx+h\vu; \vp(h)) -g(\vx; \vp(h))\big\} + \big\{g(\vx; \vp(h))- g(\vx; \vp(0))\big\} \\
        &\le \inner{\nabla_1 g(\vx+t_2\vu; \vp(h)), h\vu} + 0.
    \end{align*}
    To combine these two inequalities into a single line,
    \[\inner{\nabla_1 g(\vx+t_1\vu; \vp(0)), \vu}\le \frac{\Phi(\vx + h\vu) - \Phi(\vx)}{h} \le \inner{\nabla_1 g(\vx+t_2\vu; \vp(h)), \vu}.\]
    Using the continuity of $p(\cdot)$ and $\nabla_1 g(\cdot;\cdot)$ ($\because$ $g$ has Lipschitz continuous gradient), we can deduce that the directional derivative of $\Phi$ in a direction $\vu$ (denoted by $D_{\vu} \Phi$) is in fact
    \[D_{\vu} \Phi(\vx) = \inner{\nabla_1 g(\vx; \vy^*(\vx)), \vu},\]
    by taking the limit $h\rightarrow 0+$. Since $\vu$ is arbitrary, we can conclude that $\nabla \Phi(\vx) = \nabla_1\, g(\vx; \vy^*(\vx))$.
    
    The proof of Lipschitz smoothness of $\Phi$ exactly follows the proof by \cite{nouiehed2019solving}. Consider any $\vx_0, \vx_1\in \gX$. As in Proposition~\ref{prop:y*(x) Lipschitz}, choose any $\vy^*_0 \in \gY_{\vx_0}^*$ and $\vy^*_1 \in \Pi_{\gY_{\vx_1}^*} (\vy^*_0)$. Then,
    \begin{align*}
        &\norm{\nabla \Phi(\vx_0) - \nabla \Phi(\vx_1)} \\
        &= \norm{\nabla_1 g(\vx_0; \vy^*_0) - \nabla_1 g(\vx_1; \vy^*_1)} \\
        &\le \norm{\nabla_1 g(\vx_0; \vy^*_0) - \nabla_1 g(\vx_1; \vy^*_0)} + \norm{\nabla_1 g(\vx_1; \vy^*_0) - \nabla_1 g(\vx_1; \vy^*_1)}\\
        &\le L\left\{\norm{\vx_0-\vx_1} + \norm{\vy^*_0-\vy^*_1}\right\} \\
        &\le L(1+\kappa_2)\norm{\vx_0-\vx_1}.
    \end{align*}
    The last inequality holds because of Proposition~\ref{prop:y*(x) Lipschitz}.
\end{proof}

\begin{proposition}[$\vx$-side P{\L} $\Rightarrow$ primal P{\L}] \label{prop:Phi PL}
	Suppose $g:\gX\times\gY \rightarrow \R$ is $L$-smooth and two-sided P{\L} with constants $\mu_1$ and $\mu_2$. Then, $g$ satisfies primal P{\L} condition: the function $\Phi(\vx):=\max_{\vy'\in\gY} g(\vx; \vy')$ is $\mu_1$-P{\L}. As a result, a smooth two-sided P{\L} function is P{\L}($\Phi$)-P{\L}.
\end{proposition}

\begin{proof}
    See Lemma~A.3 of \cite{yang2020global}.
\end{proof}

\begin{definition} \label{def:optimality}
     Consider $g:\gX\times\gY \rightarrow \R$. Then, the point $(\vx^*; \vy^*)\in \gX\times\gY$ is called
    \begin{enumerate}[label=(\roman*)]
	    \item \label{def:stationary} a stationary point of $g$ if ~$\nabla_1\, g(\vx^*; \vy^*) = \nabla_2\, g(\vx^*; \vy^*) = 0.$
	    \item \label{def:saddle} a saddle point of $g$ if ~$g(\vx^*; \vy)\le g(\vx^*; \vy^*)\le g(\vx; \vy^*)$ for all $\vx, \vy$.
	    \item \label{def:global minimax 2} a global minimax point of $g$ if ~$g(\vx^*; \vy)\le g(\vx^*; \vy^*)\le \max_{\vy'} g(\vx; \vy')$ for all $\vx, \vy$.
	    \item \label{def:global maximin} a global maximin point of $g$ if ~$\min_{\vx'} g(\vx'; \vy) \le g(\vx^*; \vy^*)\le g(\vx; \vy^*)$ for all $\vx, \vy$.
	\end{enumerate}
\end{definition}

\begin{proposition}\label{prop:optimality}
    Consider a function $g:\gX\times\gY \rightarrow \R$.
    \begin{enumerate}[label=(\arabic*)]
        \item \label{prop:optimality:saddle } In general, a saddle point of $g$ is a global minimax/maximin point.
        \item Let $\Phi(\vx):=\max_{\vy} g(\vx; \vy)$ and $\Phi^* := \min_x \Phi(\vx)$ be well-defined. Let $\lambda>0$ be a constant. In general, a point $(\vx^*; \vy^*)$ is a global minimax point of $g$ if and only if \[V_\lambda(\vx^*;\vy^*):=\lambda[\Phi(\vx)-\Phi^*] + [\Phi(\vx)-g(\vx;\vy)]=0.\]
        \item If $g$ is smooth nonconvex-P{\L}, then a global minimax point is a stationary point. 
        \item If $g$ is P{\L}($\Phi$)-P{\L}, then there exists a global minimax point  $(\vx^*; \vy^*)$ of $g$. As a result, if $g$ is also smooth, then the point $(\vx^*; \vy^*)$ is a stationary point.
        \item If $g$ is smooth two-sided P{\L}, every stationary point is a saddle point. As a result, there exists a saddle point $(\vx^*; \vy^*)$ of $g$.
    \end{enumerate}
\end{proposition}
In particular, smooth two-sided P{\L} functions enjoy the ``minimax theorem,'' which establishes ``minimax = maximin.''
\begin{proof} 
    (1) (saddle point $\Rightarrow$ global minimax \& global maximin) This is straightforward by the definitions: for any $\vx$ and $\vy$,
    \[\min_{\vx'} g(\vx'; \vy) \le g(\vx^*; \vy) \le g(\vx^*; \vy^*)\le  g(\vx; \vy^*) \le \max_{\vy'} g(\vx; \vy').\]
    
    (2) (global minimax $\iff$ $V_\lambda=0$) The terms $\Phi(\vx)-\Phi^*$ and $\Phi(\vx)-g(\vx;\vy)$ are non-negative. Hence, $V_\lambda (\vx; \vy)$ is non-negative, and $V_\lambda (\vx^*; \vy^*)=0$ if and only if $\Phi^* = \Phi(\vx^*) = g(\vx^*; \vy^*)$, which is equivalent to the global minimax point condition.
    
    (3) (smooth nonconvex-P{\L}: global minimax $\Rightarrow$ stationary) Suppose $(\vx^*; \vy^*)$ is a global minimax point. Since $g(\vx^*; \vy)\le g(\vx^*; \vy^*)$ for any $\vy$, $\Phi(\vx^*) = \max_y g(\vx^*; \vy) = g(\vx^*; \vy^*)$. Thus, $\Phi$ has a minimum $g(\vx^*; \vy^*)$ at $\vx=\vx^*$. By Proposition~\ref{prop:Phi smooth},  $\Phi(\cdot)$ is a differentiable function and we have
    \[\nabla_1\, g(\vx^*; \vy^*) = \nabla \Phi(\vx^*) = 0.\]
    Also, since a differentiable function $g(\vx^*; \vy)$ has a maximum at $\vy=\vy^*$, we also have $\nabla_2\, g(\vx^*; \vy^*) = 0$. Therefore, $(\vx^*; \vy^*)$ is a stationary point.
    
    (4) (P{\L}($\Phi$)-P{\L}: $\exists$ global minimax) Let $\vx^* \in \argmin_\vx \Phi(\vx)$ and $\vy^* \in \argmax_\vy f(\vx^*;\vy)$. Then, $f(\vx^*,\vy^*)=\Phi(\vx^*)=\Phi^*$. as noted in (2), $(\vx^*,\vy^*)$ is a global minimax point. By (3), it is in fact a stationary point, when $g$ is smooth as well.
    
    (5) (smooth two-sided P{\L}: stationary $\Rightarrow$ saddle) Let $(\vx^*; \vy^*)$ be a stationary point. By P{\L} inequalities, for any $\vx$ and $\vy$,
    \begin{align*}
        0 = \norm{\nabla_2\, g(\vx^*; \vy^*)}^2 \ge 2\mu_2(\max_{\vy} g(\vx^*; \vy) - g(\vx^*; \vy^*)) \ge 0, \\
        0 = \norm{\nabla_1\, g(\vx^*; \vy^*)}^2 \ge 2\mu_1(g(\vx^*; \vy^*) - \min_{\vx} g(\vx; \vy^*)) \ge 0.
    \end{align*}
    Since $\mu_1, \mu_2>0$, these imply $\max_{\vy} g(\vx^*; \vy) = g(\vx^*; \vy^*) = \min_{\vx} g(\vx; \vy^*)$. Thus, $(\vx^*; \vy^*)$ is a saddle point. Note that (4) and Proposition~\ref{prop:Phi PL} together proves the existence of a stationary point of $g$. Therefore, there must exists a saddle point, which is also pointed out by \citet[Lemma~8]{guo2020fast}. 
    This concludes the proof.
\end{proof}

We remark that, in the proof above, (3) is false for general (nonconvex-nonconcave) functions. Only \emph{local} minimax point can ensure stationarity \citep{jin2020local}. As remarked by \cite{jin2020local} (Figure~2 of their paper), the function $xy - \cos(y)$ has non-stationary global minimax points $(0, \pm \pi)$.

The following two propositions are for showing that general two-sided P{\L} function may not have a \emph{differential Stackelberg equilibrium} defined as \citet[Definition~3.1]{li2022convergence}.

\begin{proposition} \label{prop: SC compose linear is PL}
    Let $g$ be a $\mu$-strongly convex function on $\R^n$. Consider any matrix $\mM\in\R^{n\times m}$ with a positive rank. Suppose that $\theta$ is the smallest \underline{nonzero} singular value of $\mM$. Then $g(\mM \vy)$ is a $\mu\theta^2$-P{\L} function of $\vy\in \R^m$.
\end{proposition}
\begin{proof}
    See \citet[Appendix~B]{karimi2016linear} for the proof.
\end{proof}

\begin{proposition} \label{prop:2PL wo stackelberg}
	Consider a twice continuously differentiable strongly-convex-strongly-concave function $h:\R^r\times\R^s\rightarrow \R$. That is, for some constants $\mu_1, \mu_2>0$, $h(\vx; \vy)$ is $\mu_1$-strongly-convex in $\vx$ and $-h(\vx; \vy)$ is $\mu_2$-strongly-convex in $\vy$.
	Let $(\vx^*; \vy^*)$ be the unique stationary point of $h$. Of course, it is a \textbf{differential Stackelberg equilibrium} of $h$. That is, if the hessian matrix $\nabla^2 h(\vx^*; \vy^*)$ at that point is written as
	\[\nabla^2 h(\vx^*; \vy^*)=\begin{bmatrix}
	\nabla_{1,1}^2\, h(\vx^*; \vy^*) & \nabla_{1,2}^2\, h(\vx^*; \vy^*) \\
	\nabla_{2,1}^2\, h(\vx^*; \vy^*) & \nabla_{2,2}^2\, h(\vx^*; \vy^*) \\
	\end{bmatrix} = \begin{bmatrix}
	\mC & \mB \\ \mB^\top & -\mA
	\end{bmatrix},\]
	then $\mA$ and $\mC-\mB\mA^{-1}\mB^\top$ are both positive definite matrices.
	Consider a function $g:\R^p\times \R^q \rightarrow\R$ defined by $g(\vx; \vy)=h(\mM \vx; \mN \vy)$ for some matrices $\mM\in \R^{r\times p}$, $\mN\in\R^{s\times q}$. Then, $g$ is two-sided P{\L}. Moreover, each stationary point of $g$ may not be a differential Stackelberg equilibrium in general, for example, when $s<q$.
\end{proposition}
\begin{proof}
    Because of Proposition~\ref{prop: SC compose linear is PL}, $g$ is clearly a two-sided P{\L} function.
    
    If $(\vx; \vy)$ is a stationary point of $g$, then it must be an element of an affine set $\{(\vx; \vy)\in \R^p\times\R^q:\mM \vx = \vx^*; \mN \vy =\vy^*\}$. This is because
    \[\nabla g(\vx; \vy)=\begin{bmatrix}
	\nabla_1\, g(\vx; \vy) \\ \nabla_2\, g(\vx; \vy)
	\end{bmatrix} = \begin{bmatrix}
	\mM^\top \nabla_1\, h(\mM \vx; \mN \vy) \\ \mN^\top \nabla_2\, h(\mM \vx; \mN \vy)
	\end{bmatrix} = \bm{0}\]
	if and only if $\nabla_1\, h(\mM \vx; \mN \vy)=\bm{0}$ and $\nabla_2\, h(\mM \vx; \mN \vy)=\bm{0}$, being equivalent to $\mM \vx = \vx^*$ and $\mN \vy =\vy^*$. Furthermore, the hessian of $g$ at $(\vx; \vy)$ is 
	\begin{align*}
	    \nabla^2 g(\vx; \vy) &=\begin{bmatrix}
    	\mM^\top \nabla_{1,1}^2\, h(\mM \vx; \mN \vy) \mM & \mM^\top \nabla_{1,2}^2\, h(\mM \vx; \mN \vy)\mN \\
    	\mN^\top \nabla_{2,1}^2\, h(\mM \vx; \mN \vy)\mM & \mN^\top \nabla_{2,2}^2\, h(\mM \vx; \mN \vy)\mN \\
    	\end{bmatrix} \\
    	&= \begin{bmatrix}
    	\mM^\top \mC \mM & \mM^\top \mB\mN \\
    	(\mM^\top \mB\mN)^\top & -\mN^\top \mA\mN \\
    	\end{bmatrix}.
	\end{align*}
	If $s<q$, the $q\times q$ matrix $\mN^\top \mA\mN$ cannot have a full rank, thereby it cannot be even invertible. This implies the stationary point $(\vx; \vy)$ cannot be a differential Stackelberg equilibrium.
\end{proof}

\subsection{Without-replacement sampling} \label{sec:RR lemma}

In this subsection, we provide a useful proposition for analysis of mini-batching approach under without-replacement sampling. We consider the case of mutually disjoint mini-batches in a whole epoch, not only applying without-replacement sampling to each individual mini-batch.

Consider a collection of $n$ vectors $\vv_1, \dots, \vv_n \in \R^d$. Suppose we uniformly randomly sample a permutation $\sigma:[n]\rightarrow[n]$; \ie, $\sigma \sim \operatorname{Unif}(\sS_n)$. 
Define
\[\vm=\frac{1}{n}\sum_{i=1}^n \vv_i ~~\text{(sample mean)}\quad \text{and} \quad \tau^2=\frac{1}{n}\sum_{i=1}^n \norm{\vv_i-\vm}^2 ~~\text{(sample variance)}.\]
Fix any $b\in [n]$ and let $n=b(q-1)+s$ for some integers $q\ge 1$ and $s\in [b]$. Now, divide the indices $[n]$ into $q$ batches, with exactly $b$ items per batch (except for the last batch when $s<b$), as follows:
\begin{align*}
    \gW_t &= \left\{\sigma(j) : b(t-1)<j\le bt, j\in [n]\right\} \quad (t\in [q]).
\end{align*}
For each batch $\gW_t$, define 
\[\vw_t=\frac{1}{|\gW_t|}\sum_{i\in \gW_t} \vv_i ~~\text{(batch mean)}.\]
For any $k\in[q-1]$, define
\[\vm_k := \frac{1}{k}\sum_{t=1}^k \vw_t ~~\text{(accumulative average of batch means over $1\le t\le k$)}.\]
Of course, we may simply take $\vm_q = \vm$ (deterministically) for $k=q$.
Thus, because of the randomness of $\sigma$, we can obtain the mean (vector) and the variance (scalar) of $\vm_k$ as follows. 

\begin{proposition}[Without-replacement sampling] \label{prop:variance}
    Given the setup  above, for any $k<q$ and $n>1$,
	\[\E [\vm_k] = \vm ~~\text{ and }~~ \E \left[\norm{\vm_k-\vm}^2\right]= \frac{(n-bk)}{bk(n-1)} \tau^2.\]
	(Of course, if $k=q$ or $n=1=q$, $\E [\norm{\vm_q-\vm}^2 ]=0$ since $\vm_q=\vm$.)
\end{proposition}
{\bf Remark.} As a special case, if $n=bq$ (namely, $b$ divides $n$ and $s=b$), then for any $k\le q$,
\[\E \left[\norm{\vm_k-\vm}^2\right]= \frac{(q-k)}{k(n-1)} \tau^2.\]
If we further assume $b=s=1$ and $q=n$, this proposition recovers Lemma~1 of  \citet{mishchenko2020random}.
\begin{proof}[Proof of Proposition~\ref{prop:variance}]
    Since $\sigma$ is a uniformly randomly sampled permutation, it is easy to obtain that
    \[\E [\vv_{\sigma(i)}] = \E [\vw_t] = \E [\vm_k] = \vm,\]
    for any $i\in[n]$, $t\in [q]$, and $k\in [q]$. 
    
    The covariances between $\vv_{\sigma(i)}$'s can be deduced from the proof by \citet[Lemma~1]{mishchenko2020random} as follows:
    \begin{align*}
        \operatorname{Cov}(\vv_{\sigma(i)}, \vv_{\sigma(j)}) := \E\left[\inner{\vv_{\sigma(i)}-\vm, \vv_{\sigma(j)}-\vm}\right] =\begin{cases}
            -\frac{\tau^2}{n-1}, & ~\text{if}~i\ne j,\\
            \tau^2 & ~\text{if}~i=j.
        \end{cases}
    \end{align*}
    Thus, for each $t\in[q]$, the variance of $\vw_t$ is obtained as
    \begin{align*}
        \E \left[ \norm{\vw_t-\vm}^2 \right] &= \E \left[ \norm{\frac{1}{|\gW_t|} \sum_{i\in \gW_t} \left( \vv_i-\vm \right)}^2 \right]\\
        &= \frac{1}{|\gW_t|^2} \left\{ \sum_{i\in \gW_t} \E \left[\norm{\vv_i-\vm}^2\right] + \sum_{\substack{i,j\in \gW_t \\ i\ne j}} \operatorname{Cov}(\vv_{i}, \vv_{j}) \right\}\\
        &= \frac{1}{|\gW_t|^2} \left\{ |\gW_t| \tau^2 + |\gW_t| (|\gW_t|-1) \left(-\frac{\tau^2}{n-1}\right) \right\} = \frac{n-|\gW_t|}{|\gW_t|(n-1)}\tau^2,
    \end{align*}
    which can also be directly deduced by Lemma~1 of  \citet{mishchenko2020random}. We notice that this does not depends on the size of the batch $\gW_t$.
    
    Next, we look at the covariances between distinct $\vw_t$'s. For a pair of distinct integers $t, u\in [q]$, by the bi-linearity of covariance,
    \begin{align*}
        \operatorname{Cov}(\vw_t, \vw_u) &= \frac{1}{|\gW_t|\cdot |\gW_u|} \sum_{(i,j)\in \gW_t\times \gW_u} \operatorname{Cov}(\vv_i, \vv_j) \\
        &= \frac{1}{|\gW_t|\cdot |\gW_u|} \sum_{(i,j)\in \gW_t\times \gW_u} \left(-\frac{\tau^2}{n-1}\right) = -\frac{\tau^2}{n-1}.
    \end{align*}
    The second equality holds because $\gW_t$ and $\gW_u$ are a disjoint set of integers whenever $t\ne u$.
    
    Now, fix any $k\in [q-1]$. Note that, by our mini-batching strategy, $|\gW_t| = b$ for every $t<q$. Therefore, by definition of $\vm_k$, 
    \begin{align*}
        \E \left[ \norm{\vm_k-\vm}^2 \right] &= \E \left[ \norm{\frac{1}{k} \sum_{t=1}^{k} \left( \vw_t-\vm \right)}^2 \right]\\
        &= \frac{1}{k^2} \left\{ \sum_{t=1}^{k} \E \left[\norm{\vw_t-\vm}^2\right] + \sum_{\substack{t,u \in [k] \\ t\ne u}} \operatorname{Cov}(\vw_{t}, \vw_{u}) \right\}\\
        &= \frac{1}{k^2} \left\{ k \cdot \left(\frac{n-b}{b(n-1)}\tau^2\right) + k(k-1)\cdot \left(-\frac{\tau^2}{n-1}\right) \right\} = \frac{n-bk}{bk(n-1)}\tau^2.
    \end{align*}
\end{proof}

\subsection{Basic recurrence inequality}

In this subsection, we present a basic result of a recurrence inequality. It serves as a stepping-stone of our convergence bound, particularly at the end of the proof (Appendix \ref{sec:proofs simSGDA RR 2PL}).

\begin{proposition}\label{prop:recurrence inequality}
	Let $\{a_k\}_{k=1}^{\infty}$ be a sequence of non-negative numbers satisfying the following recurrence inequality:
	\[a_{k+1} \le (1-b\eta)a_k + c\eta^{m+1} ,\]
	where $b, c,$ and $\eta$ are non-negative real numbers such that $b\eta\in (0,1)$, and $m$ is a non-negative integer. Then, for any integer $K\ge 1$, we have
	\[a_{K+1}\le (1-b\eta)^{K} a_1 + c\eta^m/b.\]
\end{proposition}
\begin{proof}
	We proceed with induction on $K=0, 1, 2,\cdots$. Note that
	\begin{align*}
		a_{1} \le (1-b\eta)^0a_1 +c \eta^m/b.
	\end{align*}
	This shows the case when $K=0$. On the other hand, if $K\ge 1$, by an inductive assumption,
	\begin{align*}
		a_{K+1} &\le (1-b\eta)a_{K} + c\eta^{m+1} \\
		&\le (1-b\eta)\cdot \left((1-b\eta)^{K-1} a_1+ c\eta^m/b\right) + c\eta^{m+1} \\
		&= (1-b\eta)^Ka_1 + c\eta^{m}/b.
	\end{align*}
\end{proof}

\section{Proofs for (mini-batch) simultaneous SGDA-RR}
\label{sec:proofs simSGDA RR}
In this appendix, we provide a convergence analysis for the mini-batch \textbf{sim}SGDA-RR (\Algref{alg:SGDA wo replacement minibatch}) on both general nonconvex-P{\L} problems and primal-P{\L}-P{\L} problems. The two cases mostly share the same proof strategies; they only diverge at the end of the proofs. 
The proof is long; we first provide the sketch of proof in subsection~\ref{subsec:sketch of proof}; then, we provide the full proof by dividing it into 4 follow-up subsections of this appendix.
The proof for the alternating counterpart (minibatch altSGDA-RR) can be done with some modifications illustrated in Appendix~\ref{sec:proofs altSGDA RR}.
All technical propositions required for the proofs can be found in Appendix~\ref{sec:propositions}.

\subsection{Warm-up: proof sketch for $b=1$}\label{subsec:sketch of proof}
Here we simply consider the proofs of Theorem~\ref{thm:NCPL}~and~\ref{thm:2PL} for {\bf sim}SGDA-RR, which is a fully stochastic case (mini-batches of size $b=1$). The proofs for altSGDA-RR can be done with slight modifications.

We start the proof by aggregating all updates throughout an epoch to obtain an ``epoch-wise'' update:
\begin{align*}
    \vx_0^{k+1} &= \vx_0^k - n\alpha\vg^k, \quad \vg^k = \textstyle\frac{1}{n}\sum_{i=1}^n \nabla_1 f_{\sigma_k(i)}(\vz_{i-1}^k),\\
    \vy_0^{k+1} &= \vy_0^k + n\beta\vh^k, \quad \vh^k = \textstyle\frac{1}{n}\sum_{i=1}^n \nabla_2 f_{\sigma_k(i)}(\vz_{i-1}^k).
\end{align*}
The reason is that the sampled components in each epoch are dependent to each other so that it is much harder to deal with each iteration individually.
The strategy of update-aggregation is quite general for analysis of optimization algorithms involving without-replacement sampling \citep{ahn2020sgd, mishchenko2020random, nguyen2021unified,das2022sampling}.
We assume that the intermediate iterates $\vz_1^k,\ldots,\vz_n^k$ stay close to the starting iterate $\vz_0^k$ of an epoch $k$, which can be ensured by small step sizes. Then, we can approximate the aggregated epoch of SGDA-RR as a step of simGDA applied to $f=\frac{1}{n}\sum_{i=1}^n f_i$, with approximations of $\vg^k \approx \nabla_1 f(\vz_0^k)$ and $\vh^k \approx \nabla_2 f(\vz_0^k)$.

With Assumptions~\ref{ass:smooth}~and~\ref{ass:NCPL}, note that the primal function $\Phi(\cdot)$ is $(L+L^2/\mu_2)$-smooth (Proposition~\ref{prop:Phi smooth}). Applying this and $L$-smoothness of $-f$, we can have the following inequality (Lemma~\ref{lem:NaiveReccurence}):
\begin{align*}
	V_\lambda(\vz_0^{k+1}) - V_\lambda(\vz_0^{k}) &\le -\left((\lambda+1)/2\right) n\alpha \norm{\nabla \Phi (\vx_0^k)}^2 + (\lambda+1)n\alpha \norm{\nabla \Phi (\vx_0^k)-\nabla_1\, f(\vz_0^k)}^2 \\
	&\quad + \left(n\alpha/2\right) \norm{\nabla_1\, f(\vz_0^k)}^2 - \left(n\beta/2\right) \norm{\nabla_2\, f(\vz_0^k)}^2\\
    &\quad + \left(\lambda + 1/2\right)n\alpha\norm{\vg^k-\nabla_1\, f(\vz_0^k)}^2 + \left(n\beta/2\right)\norm{\vh^k-\nabla_2\, f(\vz_0^k)}^2.
\end{align*}
Hence, to guarantee the fast decrease of $V_\lambda(\vz_0^k)$, it is important to control the ``noise'' terms for GDA approximations, $\norm{\vg^k-\nabla_1\, f(\vz_0^k)}^2$ and $\norm{\vh^k-\nabla_2\, f(\vz_0^k)}^2$, in the last line of inequality above. By applying the tools for without-replacement sampling (Proposition~\ref{prop:variance}), we can actually upper-bound the conditional expectations of both noise terms %
by
\[2L^2n(n+A)\left(\alpha^2\norm{\nabla_1\, f(\vz_0^k)}^2 +\beta^2\norm{\nabla_2\, f(\vz_0^k)}^2 \right) + 2L^2 n (\alpha^2+\beta^2)B. ~~ \text{(Lemma~\ref{lem:DefnGk}\,\&\,\ref{lem:boundingEGk})}\]
Then, by taking advantage of several properties of smooth nonconvex-P{\L} functions (\eg, Propositions~\ref{prop:PLQGEB},~\ref{prop:y*(x) Lipschitz},~and~\ref{prop:Phi smooth}) and some small-step-size assumptions (\eg, $\beta=\gO(1/nL)$, $\beta/\alpha = r \gtrsim \kappa_2^2$), we eventually have
\[\E\left[V_\lambda(\vz_0^{k+1})\right] - \E \left[V_\lambda(\vz_0^{k})\right] \le -n\alpha \E\left[\norm{\nabla \Phi (\vx_0^k)}^2\right] - (L\kappa_2n\alpha/2) \E\left[\Phi(\vx_0^k)-f(\vz_0^k)\right] + C\alpha^3, \]
where $C\ge0$ is a constant (with respect to $k$) depending on $L$, $n$, $B$, and $r=\beta/\alpha$. (Lemma~\ref{lem:ReccurenceGeneral}).
We note that the step size ratio $r \gtrsim \kappa_2^2$ is crucial for showing that the coefficient in front of the term  $\E\left[\Phi(\vx_0^k)-f(\vz_0^k)\right]$ is non-positive: even if it is possible with $r \lesssim \kappa_2^2$, we must assume that $\kappa_2$ upper-bounded by a positive numerical constant, which is not desirable for showing convergence bounds. Thus, we expect that a different proof strategy should be applied to avoid the requirement $r \gtrsim \kappa_2^2$ on the step size ratio.

The proofs of Theorems~\ref{thm:NCPL}~and~\ref{thm:2PL} diverge from here. The rest of the proof is mostly about choosing appropriate step sizes and solving the recurrence inequalities. %

The full proof of Theorems~\ref{thm:NCPL 2}~and~\ref{thm:2PL 2} starts from the following subsection.

\subsection{Epoch-wise representations and bounding noise terms}

Before starting the proof, we again remark that we assume that the mini-batch size $b$ divides the number of components $n$ (namely, $q:=n/b$ is a positive integer) for simplicity: thus, readers who want to read proofs for fully stochastic case (\ie, $b=1$) can substitute $n$ to every $q$. Also,
there is no problem in treating any fraction with a positive numerator and a zero denominator as $+\infty$. Moreover, we simply regard $(q-1)/(n-1)=1$ when $n=1$.

We start the proof by aggregating all updates throughout an epoch to obtain an ``epoch-wise'' update equation. The reason is that the sampled components in each epoch depend on each other, so it is much harder to deal with each iteration individually. At iteration $t\in [n/b]=[q]$ of epoch $k\in [K]$, we use a mini-batch
\[\gB^k_t := \{\sigma_k(j): b(t-1)<j\le bt, j\in[n]\}.\]
To ease the analysis of Algorithm~\ref{alg:SGDA wo replacement minibatch}, define the following sums associated with (partial) gradient oracles at a point $\vz=(\vx;\vy)$ over the mini-batch:
\[\vg_t^k (\vz) := \frac{1}{b} \sum_{i\in \gB_t^k} \nabla_1 f_{i}(\vz), \quad \vh_t^k (\vz) := \frac{1}{b} \sum_{i\in \gB_t^k} \nabla_2 f_{i}(\vz).\]
By Assumption~\ref{ass:smooth}, $\vg_t^k$ and $\vh_t^k$ are $L$-Lipschitz continuous.
Computing the average of them over a whole epoch ($\vz_0^k, \cdots, \vz_{q-1}^k$), we define
\[\vg^k := \frac{1}{q} \sum_{t=1}^q \vg_t^k (\vz_{t-1}^k), \quad \vh^k  := \frac{1}{q} \sum_{t=1}^q \vh_t^k (\vz_{t-1}^k).\]
Then, by summing up the updates in the epoch $k$, we can summarize the epoch as follows.
\[\vx_{0}^{k+1} = \vx_{0}^k - q\alpha \vg^k, \quad \vy_{0}^{k+1} = \vy_{0}^k + q\beta \vh^k. \tag{simSGDA-RR}\label{eq:epochwise update}\]

We may assume that the intermediate iterates $\vz_1^k,\ldots,\vz_q^k$ stay close to the starting iterate $\vz_0^k$ of an epoch $k$, which results from, \eg, small step sizes. Then, we can approximate the aggregated epoch of SGDA-RR as a step of simGDA applied to $f=\frac{1}{n}\sum_{i=1}^n f_i$: $\vg^k \approx \nabla_1 f(\vz_0^k), \quad \vh^k \approx \nabla_2 f(\vz_0^k)$. In other words,
\[\vx_{0}^{k+1} \approx \vx_{0}^k - q\alpha \nabla_1 f(\vz_0^k), \quad \vy_{0}^{k+1} \approx \vy_{0}^k + q\beta \nabla_2 f(\vz_0^k),\quad \tag{$\approx$simGDA}\label{eq:approximately GDA}\]

With Assumptions~\ref{ass:smooth}, \ref{ass:primal,dual} and \ref{ass:NCPL}, we can yield a naive (but complicated) upper bound of the gap $V_\lambda(\vz_0^{k+1}) - V_\lambda(\vz_0^{k})$, only applying the smoothness of $\Phi$ and $-f$, without any assumptions on step sizes. 
\begin{lemma} \label{lem:NaiveReccurence}
	Suppose that Assumptions~\ref{ass:smooth}, \ref{ass:primal,dual} and \ref{ass:NCPL} hold. Let $\kappa_2 = L/\mu_2$, where $\mu_2$ is P{\L} constant of $-f(\vx; \cdot)$. Then, the mini-batch simSGDA-RR satisfies that
	\begin{align}
		&V_\lambda(\vz_0^{k+1}) - V_\lambda(\vz_0^{k})\nonumber\\
        &\le -\left(\frac{\lambda + 1}{2}\right) q\alpha \norm{\nabla \Phi (\vx_0^k)}^2 + (\lambda+1)q\alpha \norm{\nabla \Phi (\vx_0^k)-\nabla_1 f(\vz_0^k)}^2 \nonumber\\
		&\quad + \frac{q\alpha}{2} \norm{\nabla_1 f(\vz_0^k)}^2 - \frac{q\beta}{2} \norm{\nabla_2 f(\vz_0^k)}^2 \nonumber\\
		&\quad + \left(\lambda + \frac{1}{2}\right)q\alpha\norm{\vg^k-\nabla_1 f(\vz_0^k)}^2 + \frac{q\beta}{2}\norm{\vh^k-\nabla_2 f(\vz_0^k)}^2 \nonumber\\
		&\quad - \big[\lambda - \left\{(\lambda+1)(\kappa_2+1)+1\right\}Lq\alpha \big]\frac{q\alpha}{2}\norm{\vg^k}^2 - (1-Lq\beta)\frac{q\beta}{2} \norm{\vh^k}^2. \label{eq:NaiveReccurence}
	\end{align}
\end{lemma}
\begin{proof}
	By definition of $V_\lambda$, the following equation holds:
	\begin{equation}
		V_\lambda(\vz_0^{k+1}) - V_\lambda(\vz_0^{k}) = (\lambda+1)\left[\Phi(\vx_0^{k+1})-\Phi(\vx_0^k)\right] + \left[f(\vz_0^k) - f(\vz_0^{k+1}) \right]. \label{eq:potential gap}
	\end{equation}
	
	First, we seek for an upper bound of $\Phi(\vx_0^{k+1})-\Phi(\vx_0^k)$. By Proposition~\ref{prop:Phi smooth}, $\Phi$ is $L(\kappa_2 + 1)$-smooth. Hence, we have
	\begin{align}
		&\Phi(\vx_0^{k+1})-\Phi(\vx_0^k)\nonumber\\
        &\le \inner{\nabla \Phi (\vx_0^k), \vx_0^{k+1}- \vx_0^k} + \frac{L(\kappa_2+1)}{2}\norm{\vx_0^{k+1}-\vx_0^k}^2 \nonumber\\
		&= -q\alpha \inner{\nabla \Phi (\vx_0^k), \vg^k}+\frac{L(\kappa_2+1)}{2}q^2 \alpha^2 \norm{\vg^k}^2 \nonumber\\
		&= -\frac{q\alpha}{2}\left\{\norm{\nabla \Phi (\vx_0^k)}^2 + \norm{\vg^k}^2 - \norm{\nabla \Phi (\vx_0^k)-\vg^k}^2\right\} + \frac{L(\kappa_2+1)}{2}q^2 \alpha^2 \norm{\vg^k}^2 \nonumber\\
		&= -\frac{q\alpha}{2}\norm{\nabla \Phi (\vx_0^k)}^2 +\frac{q\alpha}{2} \norm{\nabla \Phi (\vx_0^k)-\vg^k}^2 -\frac{q\alpha}{2}(1-L(\kappa_2+1)q\alpha)\norm{\vg^k}^2 \nonumber\\
		&\le -\frac{q\alpha}{2}\norm{\nabla \Phi (\vx_0^k)}^2 +q\alpha \norm{\nabla \Phi (\vx_0^k)-\nabla_1 f(\vz_0^k)}^2 +q\alpha \norm{\vg^k-\nabla_1 f(\vz_0^k)}^2\nonumber\\
		&\quad -\frac{q\alpha}{2}(1-L(\kappa_2+1)q\alpha)\norm{\vg^k}^2. \label{eq:Phi ineq}
	\end{align}	
	The third line is due to polarization equality\footnote{For any $\va, \vb\in \R^d$, $2\langle{\va,\vb}\rangle = \norm{\va}^2 + \norm{\vb}^2 - \norm{\va-\vb}^2$.} and the last inequality applies Young's inequality.\footnote{For any $\va, \vb\in \R^d$, $\norm{\va+\vb}^2 \le 2\norm{\va}^2 + 2\norm{\vb}^2$.}

	Next, applying Assumption~\ref{ass:smooth}, $L$-smoothness of $-f(\cdot;\cdot)$ yields an upper bound of $f(\vz_0^k) - f(\vz_0^{k+1})$.
	\begin{align}
		&f(\vz_0^k) - f(\vz_0^{k+1})\nonumber\\
        &\le -\inner{\nabla f (\vz_0^k), \vz_0^{k+1}- \vz_0^k} + \frac{L}{2}\norm{\vz_0^{k+1}-\vz_0^k}^2 \nonumber\\
		&= -\inner{\nabla_1 f (\vz_0^k), \vx_0^{k+1}- \vx_0^k} -\inner{\nabla_2 f (\vz_0^k), \vy_0^{k+1}- \vy_0^k}\nonumber + \frac{L}{2}\norm{\vx_0^{k+1}-\vx_0^k}^2 + \frac{L}{2}\norm{\vy_0^{k+1}-\vy_0^k}^2 \nonumber\\
		&= q\alpha \inner{\nabla_1 f (\vz_0^k), \vg^k}-q\beta \inner{\nabla_2 f (\vz_0^k), \vh^k} + \frac{L}{2}q^2 \alpha^2 \norm{\vg^k}^2 +\frac{L}{2}q^2 \beta^2 \norm{\vh^k}^2 \nonumber \\
		&= \frac{q\alpha}{2} \norm{\nabla_1 f (\vz_0^k)}^2 - \frac{q\alpha}{2}\norm{\vg^k-\nabla_1 f (\vz_0^k)}^2 + \frac{q\alpha}{2}(1+Lq\alpha)\norm{\vg^k}^2 \nonumber\\
		&\quad - \frac{q\beta}{2} \norm{\nabla_2 f (\vz_0^k)}^2 + \frac{q\beta}{2}\norm{\vh^k-\nabla_2 f (\vz_0^k)}^2 - \frac{q\beta}{2}(1-Lq\beta)\norm{\vh^k}^2. \label{eq:f ineq}
	\end{align}
	
	The last equality is due to polarization equality. Lastly, substituting \eqref{eq:Phi ineq} and \eqref{eq:f ineq} to \eqref{eq:potential gap} finishes the proof.
\end{proof}

We remark that the last two terms of the inequality~\eqref{eq:NaiveReccurence} can be simply ignored by applying small enough step sizes. However, the terms in the third line of \eqref{eq:NaiveReccurence} are non-negatives terms related to the ``noise'' of approximation $\vg^k \approx \nabla_1 f(\vz_0^k),~\vh^k \approx \nabla_2 f(\vz_0^k)$. Hence, it is important to control the noise terms $\norm{\vg^k-\nabla_1 f (\vz_0^k)}^2$ and $\norm{\vh^k-\nabla_2 f (\vz_0^k)}^2$ to guarantee a fast decrease of $V_\lambda(\vz_0^k)$.

\begin{lemma} \label{lem:DefnGk}
	For mini-batch simSGDA-RR, define
	\begin{equation}
		G_k := \frac{1}{q}\sum_{t=1}^{q} \norm{\vz_{t-1}^k - \vz_0^k}^2. \label{eq:noise}
	\end{equation}
    With Assumption~\ref{ass:smooth}, then
    \[\norm{\vg^k-\nabla_1 f (\vz_0^k)}^2 \le L^2G_k \quad\text{and}\quad \norm{\vh^k-\nabla_2 f (\vz_0^k)}^2 \le L^2G_k.\]
\end{lemma}
As a side remark, $G_k=0$ when $q=1$ and, in particular, $n=1$.
\begin{proof}
	Recall that $\frac{1}{q} \sum_{t=1}^q \vg_t^k (\vz) = \nabla_1 f(\vz)$ and $\frac{1}{q} \sum_{t=1}^q \vh_t^k (\vz) = \nabla_2 f(\vz)$. By Lipschitz continuity and Jensen's inequality,\footnote{For any $n$ vectors $a_1,\cdots,a_n$, $\norm{\frac{1}{n}\sum_{j=1}^n a_j}^2 \le \frac{1}{n} \sum_{j=1}^n \norm{a_j}^2$.}
	\begin{align*}
		\norm{\vg^k-\nabla_1 f (\vz_0^k)}^2 &= \norm{\frac{1}{q}\sum_{t=1}^{q} \left[\vg_t^k (\vz_{t-1}^k) - \vg_t^k (\vz_{0}^k)\right]}^2 \nonumber\\
		&\le \frac{1}{q}\sum_{t=1}^{q}\norm{\vg_t^k (\vz_{t-1}^k) - \vg_t^k (\vz_{0}^k)}^2 \le \frac{L^2}{q}\sum_{t=1}^{q}\norm{\vz_{t-1}^k - \vz_0^k}^2. 
	\end{align*}
	Similarly,
	\begin{equation*}
		\norm{\vh^k-\nabla_2 f (\vz_0^k)}^2 \le \frac{L^2}{q}\sum_{t=1}^{q}\norm{\vz_{t-1}^k - \vz_0^k}^2.
	\end{equation*}
	This concludes the proof.
\end{proof}
Thanks to the lemma, it suffices to bound the term $G_k$. One can notice that it also represents how far the intermediate iterates $\vz_{t}^k$ are from the pivot $\vz_0^k$ in average. Before moving on, we define an algorithm-specific symbol denoting a conditional expectation.
\begin{definition} \label{def:conditional E}
	We denote a conditional expectation of a random variable $X$ given all iterates of the first $k-1$ epochs by $\E_k [X]= \E[X|\vz_0^1, \vz_1^1, \ldots, \vz_n^{k-1}]$.
	In particular, if $k=1$, it boils down to a conditional expectation given only the initial iterate $\vz_0^1$. 
\end{definition}

We get an upper bound of a (conditional) expectation $\E_k [G_k]$ in the following lemma, which extends a lemma of \citet[Lemma~6]{nguyen2021unified} to our minimax problems.
\begin{lemma}\label{lem:boundingEGk}
	Suppose that Assumptions~\ref{ass:smooth} and \ref{ass:bddvar} hold. 
	Assume that the permutation $\sigma_k$ is sampled uniformly at random from $\sS_n$. Then, for any step sizes $\alpha,\beta$ satisfying $\alpha^2 + \beta^2 \le \frac{1}{3q(q-1) L^2}$, 
	the iterates $\{\vz_t^k\}_{t=0}^{q-1}$ of the $k$-th epoch of mini-batch simSGDA-RR satisfies (for $n>1$)
	\begin{equation*}
		\E_k G_k \le  2\left( q^2+\frac{q(q-1)}{n-1} A\right) \left(\alpha^2 \norm{\nabla_1 f(\vz_0^k)}^2 + \beta^2 \norm{\nabla_2 f(\vz_0^k)}^2\right) + \frac{2q(q-1)}{n-1}(\alpha^2+\beta^2)B.
	\end{equation*}
\end{lemma}
\begin{proof}
    Note that $G_k=0$ when $q=1$ by its definition. From now, we may assume $q>1$ and $n>1$ in this proof.
	By summing the first $t \in [q-1]$ updates of the $k$-th epoch of mini-batch simSGDA-RR, we have 
	\[\vx_{t}^{k} = \vx_{0}^k - t\alpha \left(\frac{1}{t} \sum_{j=1}^t \vg_j^k(\vz_{j-1}^k)\right), \quad \vy_{t}^{k} = \vy_{0}^k + t\beta \left(\frac{1}{t} \sum_{j=1}^t \vh_j^k(\vz_{j-1}^k)\right).\]
	Then we can bound the following squared distance.
	\begin{align}
		&\norm{\vx_t^k - \vx_0^k}^2 = \alpha^2 t^2 \norm{\frac{1}{t} \sum_{j=1}^t \vg_j^k(\vz_{j-1}^k)}^2 \nonumber\\
		&\le 3\alpha^2 t^2 \left[\norm{\frac{1}{t}\sum_{j=1}^{t} \left[\vg_j^k  (\vz_{j-1}^k)-\vg_j^k (\vz_0^k)\right]}^2 + \norm{\frac{1}{t}\sum_{j=1}^{t} \vg_j^k (\vz_0^k)-\nabla_1 f (\vz_0^k)}^2 + \norm{\nabla_1 f(\vz_0^k)}^2\right]\nonumber\\
		&\le 3\alpha^2 t\sum_{j=1}^{t} \norm{\vg_j^k  (\vz_{j-1}^k)-\vg_j^k (\vz_0^k)}^2 + 3\alpha^2 t^2 \left[\norm{\frac{1}{t}\sum_{j=1}^{t} \vg_j^k (\vz_0^k)-\nabla_1 f (\vz_0^k)}^2 + \norm{\nabla_1 f(\vz_0^k)}^2\right]\nonumber\\
		&\le 3\alpha^2 L^2 t \cdot \sum_{j=1}^{t} \norm{\vz_{j-1}^k-\vz_0^k}^2 + 3\alpha^2 t^2 \left[\norm{\frac{1}{t}\sum_{j=1}^{t} \vg_j^k (\vz_0^k)-\nabla_1 f (\vz_0^k)}^2 + \norm{\nabla_1 f(\vz_0^k)}^2\right]\nonumber\\
		&\le 3\alpha^2 L^2 t \cdot qG_k + 3\alpha^2 t^2 \left[\norm{\frac{1}{t}\sum_{j=1}^{t} \vg_j^k (\vz_0^k)-\nabla_1 f (\vz_0^k)}^2 + \norm{\nabla_1 f(\vz_0^k)}^2\right]. \label{eq:xgap}
	\end{align}
	The second and third lines are due to Jensen's inequality. The fourth line is due to $L$-Lipschitz continuity of $\vg_j^k$. Likewise, 
	\begin{equation}
		\norm{\vy_t^k - \vy_0^k}^2 \le 3\beta^2 L^2 t \cdot qG_k + 3\beta^2 t^2 \left[\norm{\frac{1}{t}\sum_{j=1}^{t} \vh_j^k (\vz_0^k)-\nabla_2 f (\vz_0^k)}^2 + \norm{\nabla_2 f(\vz_0^k)}^2\right].\label{eq:ygap}
	\end{equation}
	Summing up (\ref{eq:xgap}) and (\ref{eq:ygap}),
	\begin{align}
		&\norm{\vz_t^k - \vz_0^k}^2 = \norm{\vx_t^k - \vx_0^k}^2 + \norm{\vy_t^k - \vy_0^k}^2 \nonumber\\
		&\le 3(\alpha^2 +\beta^2) L^2 t qG_k + 3\alpha^2 t^2 \left[\norm{\frac{1}{t}\sum_{j=1}^{t} \vg_j^k (\vz_0^k)-\nabla_1 f (\vz_0^k)}^2 + \norm{\nabla_1 f(\vz_0^k)}^2\right] \nonumber\\
		&\quad + 3\beta^2 t^2 \left[\norm{\frac{1}{t}\sum_{j=1}^{t} \vh_j^k (\vz_0^k)-\nabla_2 f (\vz_0^k)}^2 + \norm{\nabla_2 f(\vz_0^k)}^2\right]. \label{eq:zgap}
	\end{align}

	Taking (conditional) expectation $\E_k$ (given $\vz_0^k$) to inequality~\eqref{eq:zgap},
	\begin{align}
		&\E_k \norm{\vz_t^k - \vz_0^k}^2 \nonumber\\
		&\stackrel{\eqref{eq:zgap}}{\le} 3(\alpha^2 +\beta^2) L^2 t q \cdot (\E_k[G_k]) + 3\alpha^2 t^2 \norm{\nabla_1 f(\vz_0^k)}^2   + 3\beta^2 t^2\norm{\nabla_2 f(\vz_0^k)}^2  \nonumber\\
		&\quad + 3\alpha^2 t^2 \E_k \norm{\frac{1}{t}\sum_{j=1}^{t} \vg_j^k (\vz_0^k)-\nabla_1 f (\vz_0^k)}^2 + 3\beta^2 t^2 \E_k\norm{\frac{1}{t}\sum_{j=1}^{t} \vh_j^k (\vz_0^k)-\nabla_2 f (\vz_0^k)}^2. \label{eq:expected zgap}
	\end{align}

	Here we take advantage of the without-replacement sampling. Putting $\nabla_s f_i (\vz_0^k) \mapsto \vv_i$ ($s\in \{1,2\}$), one can realize a correspondence between the quantities that arise from our algorithm and the symbols in Appendix~\ref{sec:RR lemma}: for $s=1$ ($\nabla_1 f_i (\vz_0^k) \mapsto \vv_i$),
	\[\vm = \nabla_1 f (\vz_0^k), \quad \tau^2 \le A\norm{\nabla_1 f (\vz_0^k)}^2+B,\quad \vw_t = \vg_t^k(\vz_0^k), \quad \vm_t = \frac{1}{t}\sum_{j=1}^{t} \vg_j^k (\vz_0^k),\]
	and for $s=2$ ($\nabla_2 f_i (\vz_0^k) \mapsto \vv_i$), 
	\[\vm = \nabla_2 f (\vz_0^k), \quad \tau^2 \le A\norm{\nabla_2 f (\vz_0^k)}^2+B,\quad \vw_t = \vh_t^k(\vz_0^k), \quad \vm_t = \frac{1}{t}\sum_{j=1}^{t} \vh_j^k (\vz_0^k).\]
	The upper bounds of $\tau^2$'s come from Assumption~\ref{ass:bddvar}. Then by Proposition \ref{prop:variance}, for any $t\le q$, 
	\[t^2\E_k \norm{\frac{1}{t}\sum_{j=1}^{t} \vg_j^k (\vz_0^k)-\nabla_1 f (\vz_0^k)}^2 \le \frac{t(q-t)}{n-1} \left(A\norm{\nabla_1 f (\vz_0^k)}^2 + B\right),\]
	\[t^2\E_k\norm{\frac{1}{t}\sum_{j=1}^{t} \vh_j^k (\vz_0^k)-\nabla_2 f (\vz_0^k)}^2 \le \frac{t(q-t)}{n-1} \left(A\norm{\nabla_2 f (\vz_0^k)}^2 + B\right).\]

	Putting these to the inequality~\eqref{eq:expected zgap},
	\begin{align*}
		\E_k \norm{\vz_t^k - \vz_0^k}^2 &\le 3(\alpha^2+\beta^2)\left[L^2 t q \E_k [G_k] + \frac{t(q-t)}{n-1}B\right] \\
		&\quad +  3\left(\alpha^2 \norm{\nabla_1 f(\vz_0^k)}^2 + \beta^2 \norm{\nabla_2 f(\vz_0^k)}^2\right)\left[t^2+\frac{t(q-t)}{n-1}A\right].
	\end{align*}
	Taking an average of the inequality above over $0\le t\le q-1$, 
	\begin{align}
		\E_k G_k &= \frac{1}{q} \sum_{t=0}^{q-1} \E_k \norm{\vz_{t}^k - \vz_0^k}^2 \notag\\
		&\le \frac{3q(q-1)}{2} (\alpha^2 +\beta^2)L^2 \E_k G_k + (\alpha^2+\beta^2)\frac{q^2-1}{2(n-1)}B\notag\\
		&\quad + \left(\alpha^2 \norm{\nabla_1 f(\vz_0^k)}^2 + \beta^2 \norm{\nabla_2 f(\vz_0^k)}^2\right)\left(\frac{(q-1)(2q-1)}{2}+\frac{q^2-1}{2(n-1)}A\right),\label{eq:sumbound}
	\end{align}
	where we used the facts 
    $$\textstyle\sum_{t=0}^{q-1} t = \frac{q(q-1)}{2},\quad\frac{1}{q}\sum_{t=0}^{q-1} t^2 = \frac{(q-1)(2q-1)}{6},\quad\text{and}\quad\frac{1}{q}\sum_{t=0}^{q-1} \frac{t(q-t)}{n-1} = \frac{q^2-1}{6(n-1)}.$$
    Since we assumed $\alpha^2 + \beta^2 \leq \frac{1}{3q(q-1)L^2}$, we have $1 \leq 2 \left( 1-\frac{3q(q-1)L^2}{2}(\alpha^2 +\beta^2)\right)$. Using this,
	\begin{align*}
		&\E_k G_k \le 2\left(1-\frac{3q(q-1)L^2}{2}(\alpha^2 +\beta^2)\right)\E_k G_k\\
		&\stackrel{\text{\eqref{eq:sumbound}}}{\le} \left((q-1)(2q-1)+ \frac{q^2-1}{(n-1)}A\right)\left(\alpha^2 \norm{\nabla_1 f(\vz_0^k)}^2 + \beta^2 \norm{\nabla_2 f(\vz_0^k)}^2\right) + \frac{q^2-1}{n-1}(\alpha^2 + \beta^2)B \\
		&\le 2\left( q^2+\frac{q(q-1)}{n-1} A\right) \left(\alpha^2 \norm{\nabla_1 f(\vz_0^k)}^2 + \beta^2 \norm{\nabla_2 f(\vz_0^k)}^2\right) + \frac{2q(q-1)}{n-1}(\alpha^2+\beta^2)B,
	\end{align*}
	where the last inequality used $(q-1)(2q-1) \leq 2q^2$  and $q+1 \le 2q$ for $q \geq 1$.
\end{proof}

\subsection{Recurrence inequalities for general smooth nonconvex-P{\L} objective}

Subsequently, we obtain recurrence inequalities about (expected) potential function $\E_k[V_\lambda(\vz_0^k)]$ for nonconvex-P{\L} problem. Since primal-P{\L}-P{\L} problem is a subclass of nonconvex-P{\L} problem, the recurrence relations can serve as stepping-stones of our convergence rates.

We introduce some assumptions on \emph{small} step sizes which enable us to get rid of a few troublesome terms from our bound. On top of that, combining the P{\L} condition (Assumption~\ref{ass:NCPL}) with Lemmas~\ref{lem:NaiveReccurence},~\ref{lem:DefnGk},~and~\ref{lem:boundingEGk}, we eventually obtain a much more concise bound on the expected per-epoch change of $V_\lambda$. This simple recurrence inequality becomes the key to proving our convergence bounds.
\begin{lemma}\label{lem:ReccurenceWithSmallStepSizes}
	Suppose that Assumptions~\ref{ass:smooth}, \ref{ass:bddvar}, \ref{ass:primal,dual}, and \ref{ass:NCPL} hold. Assume that the step sizes $\alpha$ and $\beta$ satisfy
	\begin{equation}
		\alpha\le \frac{\lambda}{\{(\lambda+1)(\kappa_2+1) +1\}qL}, \quad \beta\le \frac{1}{qL}, \quad \alpha^2+\beta^2 \le \frac{1}{3q(q-1)L^2}, \label{eq:step sizes condition 1}
	\end{equation}
	and the condition
	\[C_0 := q\beta- 2L^2q\left(q^2+\frac{q(q-1)}{n-1}A\right) \left((2\lambda+1)\alpha +\beta \right)\beta^2 \ge 0\]
	as well. Then, the iterates of mini-batch simSGDA-RR satisfy
	\[\E_k[V_\lambda(\vz_0^{k+1})] - V_\lambda(\vz_0^{k}) \le -C_1 \norm{\nabla \Phi (\vx_0^k)}^2 - C_2 \left[\Phi(\vx_0^k)-f(\vz_0^k)\right] + C_3\]
	where
	\begin{align*}
		C_1 &= \left(\frac{\lambda-1}{2}\right)q\alpha- 2L^2q\left(q^2+\frac{q(q-1)}{n-1}A\right)\big( (2\lambda+1)\alpha +\beta \big)\alpha^2, \\
		C_2 &= \mu_2 C_0 - 2(\lambda+2)L\kappa_2 q\alpha - 4L^3\kappa_2 q\left(q^2+\frac{q(q-1)}{n-1}A\right)\big( (2\lambda+1)\alpha +\beta \big)\alpha^2\\
		&= \mu_2 q \beta - 2(\lambda+2)L\kappa_2 q\alpha - 2L^2\mu_2 q\left(q^2+\frac{q(q-1)}{n-1}A\right)\big( (2\lambda+1)\alpha +\beta \big)\left(2\kappa_2^2\alpha^2 + \beta^2\right),\\
		C_3 &= \left(\frac{L^2q^2(q-1)}{n-1}\right)\left((2\lambda+1)\alpha +\beta\right)(\alpha^2+\beta^2)B.
	\end{align*}
\end{lemma}
\begin{proof}
	The first two inequalities of \eqref{eq:step sizes condition 1} eliminate the last two terms on the right-hand side of the inequality in Lemma~\ref{lem:NaiveReccurence}. In addition, applying Lemma~\ref{lem:DefnGk} to Lemma~\ref{lem:NaiveReccurence} as well, we have
	\begin{align}
		V_\lambda(\vz_0^{k+1}) - V_\lambda(\vz_0^{k}) &{\le} -\left(\frac{\lambda + 1}{2}\right) q\alpha \norm{\nabla \Phi (\vx_0^k)}^2 + (\lambda+1)q\alpha \norm{\nabla \Phi (\vx_0^k)-\nabla_1 f(\vz_0^k)}^2 \nonumber\\
		&\quad + \frac{q\alpha}{2} \norm{\nabla_1 f(\vz_0^k)}^2 - \frac{q\beta}{2} \norm{\nabla_2 f(\vz_0^k)}^2 + \frac{(2\lambda+1)\alpha +\beta}{2}qL^2 G_k. \label{eq:before expectation}
	\end{align}
	If we take the conditional expectation $\E_k$ and apply Lemma~\ref{lem:boundingEGk} (which requires the third inequality of \eqref{eq:step sizes condition 1} to hold) to \eqref{eq:before expectation}
	\begin{align}
		&\E_k[V_\lambda(\vz_0^{k+1})] - V_\lambda(\vz_0^{k}) \nonumber\\
		&{\le} -\left(\frac{\lambda + 1}{2}\right) q\alpha \norm{\nabla \Phi (\vx_0^k)}^2 + (\lambda+1)q\alpha \norm{\nabla \Phi (\vx_0^k)-\nabla_1 f(\vz_0^k)}^2 \nonumber\\
		&\quad + \frac{1}{2}\left[ q\alpha + 2L^2q\left(q^2+\frac{q(q-1)}{n-1}A\right) \left((2\lambda+1)\alpha +\beta \right)\alpha^2\right]\norm{\nabla_1 f(\vz_0^k)}^2 \nonumber\\
		&\quad - \frac{1}{2}\underbrace{\left[q\beta- 2L^2q\left(q^2+\frac{q(q-1)}{n-1}A\right) \left((2\lambda+1)\alpha +\beta \right)\beta^2\right]}_{C_0}\norm{\nabla_2 f(\vz_0^k)}^2 \nonumber\\
		&\quad+ \underbrace{\left(\frac{L^2q^2(q-1)}{n-1}\right)\left((2\lambda+1)\alpha +\beta\right)(\alpha^2+\beta^2)B}_{C_3}. \label{eq:epochwise intermed}
	\end{align}
    It is now left to bound terms in \eqref{eq:epochwise intermed} using the tools developed so far. First, recall that $\Phi(\vx):= \max_{\vy'\in \gY} f(\vx; \vy')$. Since $-f(\vx; \vy)$ is $\mu_2$-P{\L} in $\vy$, we have
	\begin{equation}
		-\norm{\nabla_2 f(\vz_0^k)}^2\le -2\mu_2(\Phi(\vx_0^k)-f(\vz_0^k)). \label{eq:-f is PL}
	\end{equation}
	Given any $\vx$, $\nabla \Phi(\vx) = \nabla_1 f(\vx; \vy^*(\vx))$ for any $\vy^*(\vx) \in \argmax_{\vy'\in \gY} f(\vx; \vy')$ by Proposition~\ref{prop:Phi smooth}. Besides, $-f(\vx; \cdot)$ satisfies QG condition with constant $\mu_2$ by Proposition~\ref{prop:PLQGEB}. Thus, by choosing $\vy^*(\vx_0^k)$ to be the projection of $\vy_0^k$ onto $\argmax_{\vy'\in \gY} f(\vx_0^k; \vy')$,
	\begin{align}
		\norm{\nabla \Phi (\vx_0^k)-\nabla_1 f(\vz_0^k)}^2 \le L^2 \norm{\vy^*(\vx_0^k)-\vy_0^k}^2 \le 2L\kappa_2 \left[\Phi(\vx_0^k)-f(\vz_0^k)\right]. \label{eq:apply QG}
	\end{align}
	Here, the first inequality applies $L$-Lipschitz continuity of $\nabla_1 f(\vx_0^k; \cdot)$, implied by Assumption~\ref{ass:smooth}.
	On top of that, applying the Young's inequality to the term $\norm{\nabla_1 f(\vz_0^k)}^2$, 
	\begin{align}
		\norm{\nabla_1 f(\vz_0^k)}^2 &\le 2 \norm{\nabla \Phi(\vx_0^k)}^2 + 2 \norm{\nabla\Phi(\vx_0^k)-\nabla_1 f(\vz_0^k)}^2 \nonumber\\
		&\stackrel{\text{\eqref{eq:apply QG}}}{\le} 2 \norm{\nabla \Phi(\vx_0^k)}^2 + 4L\kappa_2 \left[\Phi(\vx_0^k)-f(\vz_0^k)\right] \label{eq:apply QG and Young}
	\end{align}

	By applying inequalities~\eqref{eq:-f is PL}, \eqref{eq:apply QG}, and \eqref{eq:apply QG and Young} to the bound~\eqref{eq:epochwise intermed}, we conclude the proof.
\end{proof}

In Lemma~\ref{lem:ReccurenceWithSmallStepSizes}, we saw that if step sizes are chosen to satisfy certain conditions, then we can simplify the per-epoch progress a great deal. It is now left to choose appropriate step sizes and parameters (\eg, $\lambda$) so as to make sure not only that $\alpha$ and $\beta$ meet the \emph{small} step size conditions \eqref{eq:step sizes condition 1} but also that the constants $C_0$, $C_1$, $C_2$, and $C_3$ are positive. 

\begin{lemma}\label{lem:ReccurenceGeneral}
	Suppose that Assumptions~\ref{ass:smooth}, \ref{ass:bddvar}, \ref{ass:primal,dual} and \ref{ass:NCPL} hold. Let $\lambda=4$ and assume that
	\[0<\beta\le \frac{1}{6L \sqrt{q^2 + \frac{q(q-1)}{n-1}A}}, \quad \alpha = \frac{\beta}{r}, \quad \text{where}~~r\ge14\kappa_2^2.\]
	Then these satisfy all the inequalities~\eqref{eq:step sizes condition 1} and the terms defined in Lemma~\ref{lem:ReccurenceWithSmallStepSizes} satisfy
	\[C_0>0, \quad C_1>q\alpha, \quad C_2>L\kappa_2 q \alpha/2, \quad C_3\ge0.\]
	Consequently, due to the recurrence inequality in Lemma~\ref{lem:ReccurenceWithSmallStepSizes}, mini-batch simSGDA-RR satisfies, for some numerical constant $c>0$,
	\begin{align}
		&\E_k[V_\lambda(\vz_0^{k+1})] - V_\lambda(\vz_0^{k}) \nonumber\\
        &\le -q\alpha \norm{\nabla \Phi (\vx_0^k)}^2 - (L\kappa_2 q \alpha/2) \left[\Phi(\vx_0^k)-f(\vz_0^k)\right] + (cr)^3 L^2 \left(\frac{q^2(q-1)}{n-1}\right)B\alpha^3.\tag{$\star$}\label{eq:recurrence inequality general}
	\end{align}
\end{lemma}
Please note that we mark the recurrence inequality above with a special symbol \eqref{eq:recurrence inequality general} because this inequality is the exact point where the proofs of Theorems~\ref{thm:NCPL 2} and \ref{thm:2PL 2} start to deviate.
\begin{proof}
	Regardless of $A\ge0$, we have
	\begin{equation}
	    \beta \le \frac{1}{6Lq} \quad \text{and}\quad \alpha \le \frac{1}{6Lqr}\le \frac{1}{84L\kappa_2^2q}. \label{eq:alpha bound}
	\end{equation}
	This is enough to guarantee that the inequalities~\eqref{eq:step sizes condition 1} hold with $\lambda=4$. Since $C_0 > C_2/\mu_2$, it is enough to show $C_2 > 0$ to prove that $C_0 > 0$. 
	Applying $\lambda=4$, $\kappa_2\ge 1$, and $\beta/\alpha =r\ge 14\kappa_2^2$,
	\begin{align*}
		\frac{C_1}{q\alpha} &= \frac{3}{2} - 2L^2\left(q^2+\frac{q(q-1)}{n-1}A\right) \left(9+ r\right)\alpha^2 \\
		&\ge \frac{3}{2} - \frac{2}{6^2}\cdot\frac{9+r}{r^2}\ge \frac{3}{2} - \frac{2\cdot 23}{6^2\cdot14^2} > 1, \\
		\frac{C_2}{\mu_2 q\beta} &= 1- \frac{12\kappa_2^2}{r} -  2L^2\left(q^2+\frac{q(q-1)}{n-1}A\right) \left(\frac{9}{r}+ 1\right)
		\left(\frac{2\kappa_2^2}{r^2} + 1\right)\beta^2 \\
		&\ge 1-\frac{12}{14} - \frac{2}{6^2}\left(\frac{9}{14\kappa_2^2}+ 1\right)
		\left(\frac{2}{14^2\kappa_2^2} + 1\right)\ge \frac{2}{14} - \frac{2\cdot 23\cdot 198}{6^2\cdot 14^3}>\frac{1}{2\cdot 14}.
	\end{align*}
	Thus, $C_1 > q\alpha$ and
	\[C_2 > \frac{\mu_2 q \beta}{2\cdot 14} = \frac{\mu_2qr\alpha}{2\cdot14}\ge L\kappa_2 q\alpha/2.\]
	Then we conclude the proof by bounding the term $C_3$. We can already check from the definition that $C_3 \geq 0$. We can upper-bound $C_3$ by
	\[C_3 = \left(\frac{L^2q^2(q-1)}{n-1}\right)\left(9 + r\right)(1+r^2)B\alpha^3 \le (cr)^3 L^2 \left(\frac{q^2(q-1)}{n-1}\right)B\alpha^3,\]
	for some numerical constant $c>0$.
\end{proof}

\subsection{Convergence rates for smooth nonconvex-P{\L} problem}

In this subsection, we show the convergence bound of general smooth nonconvex-P{\L} problems in terms of $\min_{k\in [K]} \E\left[\norm{\nabla \Phi (\vx_0^k)}^2\right]$. From the inequality~\eqref{eq:recurrence inequality general} in Lemma~\ref{lem:ReccurenceGeneral}, we can simply ignore the second term
\[- (L\kappa_2 q \alpha/2) \left[\Phi(\vx_0^k)-f(\vz_0^k)\right] \le 0\]
of the right-hand side because $\Phi(\vx) \ge f(\vx;\vy)$ for any $(\vx;\vy)$. In other words, we may deal with the inequality 
\[\E_k[V_\lambda(\vz_0^{k+1})] - V_\lambda(\vz_0^{k}) \le -q\alpha \norm{\nabla \Phi (\vx_0^k)}^2 + (cr)^3 L^2 \left(\frac{q^2(q-1)}{n-1}\right)B\alpha^3. \tag{nc-P{\L}}\label{eq:recurrence inequality NCPL}\]
Plugging $q=n/b$, we eventually show the convergence rate (Theorem~\ref{thm:NCPL 2}). (Recall that $b$ is the size of mini-batches.) 

\begin{theorem}[Equivalent to Theorem~\ref{thm:NCPL 2}, for simSGDA-RR] \label{thm:NCPL 3}
	Suppose that $f$ satisfies Assumptions~\ref{ass:smooth}, \ref{ass:bddvar}, \ref{ass:primal,dual}, and \ref{ass:NCPL} are satisfied. 
	Let $\lambda=4$. Choose the step sizes $\alpha$ and $\beta$ by $\alpha=\beta/r$ for some $r\ge 14\kappa_2^2$ and
	\[\beta = \min\left\{\frac{1}{6L \sqrt{q^2 + \frac{q(q-1)}{n-1}A}},\,\, \frac{1}{c}\left( \frac{V_\lambda(\vz_0^1)}{L^2q^2(\frac{q-1}{n-1})BK} \right)^{\frac{1}{3}}\right\},\]
	for some numerical constant $c>0$. Then, mini-batch simSGDA-RR satisfies
	\begin{align*}
		\frac{1}{K}\sum_{k=1}^K  \E\left[\norm{\nabla \Phi (\vx_0^k)}^2\right] \le \frac{6r L V_\lambda(\vz_0^1)}{K}\sqrt{1+\left(\frac{q-1}{n-1}\right)\frac{A}{q}} + 2cr \left(\frac{L^2 B\, V_\lambda(\vz_0^1)^2}{qK^2}\cdot \frac{q-1}{n-1}\right)^{1/3}.
	\end{align*}
\end{theorem}

\begin{proof}
	To replace the conditional expectations with unconditional expectations, we take expectation to both sides of the inequality~\eqref{eq:recurrence inequality NCPL}:
	\[\E[V_\lambda(\vz_0^{k+1}) - V_\lambda(\vz_0^{k})] \le -q\alpha  \E\left[\norm{\nabla \Phi (\vx_0^k)}^2\right] + (cr)^3 L^2 \left(\frac{q^2(q-1)}{n-1}\right)B\alpha^3. \]
	Rearranging the terms and taking a sum from $k=1$ to $k=K$, we have
	\[q\alpha\sum_{k=1}^K  \E\left[\norm{\nabla \Phi (\vx_0^k)}^2\right] \le \E[V_\lambda(\vz_0^{1})-V_\lambda(\vz_0^{K+1})]+(cr)^3 L^2 \left(\frac{q^2(q-1)}{n-1}\right)B\alpha^3K.\]
	Dividing both sides by $qK\alpha$, we get the following. Note that $V_\lambda$ is non-negative.
	\begin{align*}
		\frac{1}{K}\sum_{k=1}^K  \E\left[\norm{\nabla \Phi (\vx_0^k)}^2\right] &\le \frac{V_\lambda(\vz_0^{1})}{qK\alpha}+(cr)^3 L^2 \left(\frac{q(q-1)}{n-1}\right)B\alpha^2
	\end{align*}
	Since our choice of step sizes implies
	\[\alpha = \min\left\{\frac{1}{6rL\sqrt{q^2 + \frac{q(q-1)}{n-1}A}},\,\, \frac{1}{cr} \left( \frac{V_\lambda(\vz_0^1)}{L^2Bq^2(\frac{q-1}{n-1})K} \right)^{\frac{1}{3}}\right\},\]
	we eventually prove the theorem by using the inequality $\max\{a,b\}\le a+b$ (for $a,b\ge0$).
\end{proof}

\subsection{Convergence rates for smooth primal-P{\L}-P{\L} problem} \label{sec:proofs simSGDA RR 2PL}

In this subsection, we prove the convergence bound of primal-P{\L}-P{\L} (or, P{\L}($\Phi$)-P{\L}) problems in terms of $\E\left[V_\lambda(\vz_0^{K+1})\right]$.

Unlike the previous subsection, we additionally utilize Assumption~\ref{ass:primalPL} stating that $f(\vx; \vy)$ satisfies primal P{\L} condition, namely, the primal function $\Phi(\vx) = \max_{\vy'} f(\vx;\vy')$ is a $\mu_1$-P{\L} function. With this assumption, we yield another recurrence inequality from the inequality~\eqref{eq:recurrence inequality general}. We note that it uses the $\mu_1$-P{\L} condition for $\Phi$ ($\because$ Proposition~\ref{prop:Phi PL}) but not necessarily for $f(\cdot; \vy)$.

\begin{lemma}\label{lem:Reccurence2PL}
	Suppose that $f$ satisfies Assumptions~\ref{ass:smooth}, \ref{ass:bddvar}, \ref{ass:primal,dual}, \ref{ass:NCPL}, and \ref{ass:primalPL}. Then, with the same choice of $\lambda=4$ and the same condition of the step sizes $\alpha$ and $\beta$ as in Lemma~\ref{lem:ReccurenceGeneral}, the mini-batch simSGDA-RR satisfies that, for some numerical constant $c>0$,
	\begin{align}
		\E_k[V_\lambda(\vz_0^{k+1})] &\le (1-\mu_1 q\alpha/2) V_\lambda(\vz_0^k) + (cr)^3 L^2\left(\frac{q^2(q-1)}{n-1}\right)B\alpha^3. \tag{P{\L}($\Phi$)-P{\L}}\label{eq:recurrence inequality 2PL}
	\end{align}
\end{lemma}
\begin{proof}
	Since the primal function $\Phi$ is a $\mu_1$-P{\L} function,
	\[-\norm{\nabla \Phi(\vx_0^k)}^2 \le -2\mu_1\left[\Phi(\vx_0^k)-\Phi^*\right].\]

	Also, since $\mu_1 \le L$ and $\kappa_2\ge 1$, we know that $-L\kappa_2 \le -\mu_1$. Applying these to the inequality~\eqref{eq:recurrence inequality general}, we have
	\begin{align*}
		&\E_k\left[V_\lambda(\vz_0^{k+1})\right] - V_\lambda(\vz_0^{k}) \\
		&\le -(2\mu_1 q\alpha/\lambda)\cdot \lambda \left[\Phi(\vx_0^k)-\Phi^*\right] - (\mu_1 q\alpha/2) \left[\Phi(\vx_0^k)-f(\vz_0^k)\right] + (cr)^3 L^2 \left(\frac{q^2(q-1)}{n-1}\right)B\alpha^3\\
		&= -(\mu_1 q\alpha/2)\cdot V_\lambda(\vz_0^k) + (cr)^3 L^2 \left(\frac{q^2(q-1)}{n-1}\right)B\alpha^3,
	\end{align*}
	since $\lambda=4$. By re-arranging the terms, we conclude the proof.
\end{proof}
Of course, the multiplier $1-\mu_1 q\alpha/2$ has a value between 0 and 1. To see why, note that from Equation~\eqref{eq:alpha bound},
\[ 0<\mu_1 q\alpha/2\le  \mu_1 q\cdot \frac{1}{2\cdot84L\kappa_2^2 q} =\frac{1}{168\kappa_1\kappa_2^2}<1.\]

\begin{theorem}[Equivalent to Theorem~\ref{thm:2PL 2}, for simSGDA-RR] \label{thm:2PL 3}
	Assume that $f$ satisfies Assumptions~\ref{ass:smooth}, \ref{ass:bddvar}, \ref{ass:primal,dual}, \ref{ass:NCPL}, and \ref{ass:primalPL}. Let $\lambda=4$. Choose the step sizes by $\alpha=\beta/r$ for some $r\ge 14\kappa_2^2$ and
	\[\beta = \min \left\{ \frac{1}{6L \sqrt{q^2 + \frac{q(q-1)}{n-1}A}},\,\, \frac{2r}{\mu_1 q K} \max\left\{1, \,\,\log \left(\frac{V_\lambda(\vz_0^1) \mu_1 q K^2}{8(cr)^3\kappa_1^2 \left( \frac{q-1}{n-1}\right)B}\right)\right\}\right\},\]
	for some numerical constant $c>0$. Then, mini-batch simSGDA-RR satisfies
	\[\E[V_\lambda(\vz_n^{K})] \le \gO\left(V_\lambda(\vz_0^1)\cdot\exp\left(-\frac{K}{12\kappa_1r\sqrt{1 + \left(\frac{q-1}{n-1}\right)\frac{A}{q}}}\right)\right)+ \tilde{\gO}\left(\frac{\kappa_1^2r^3 B}{\mu_1 q K^2}\right) \cdot \frac{q-1}{n-1}.\]
\end{theorem}
\begin{proof}
	To replace the conditional expectations with unconditional expectations, we take expectation to both sides of the inequality~\eqref{eq:recurrence inequality 2PL}:
	\begin{equation*}
	    \E\left[V_\lambda(\vz_0^{k+1})\right] \le
		(1-\mu_1 q\alpha/2) \E\left[V_\lambda(\vz_0^k)\right] +  (cr)^3 L^2 \left(\frac{q^2(q-1)}{n-1}\right)B\alpha^3.
	\end{equation*}
	Unrolling the recurrence inequality (Proposition~\ref{prop:recurrence inequality}) and using the facts $\beta = 14\kappa_2^2 \alpha$, we have
	\begin{align}
		\E[V_\lambda(\vz_n^{K})] &\le (1-\mu_1 q\alpha/2)^K V_\lambda(\vz_0^1)+\frac{2\cdot (cr)^3 L^2}{\mu_1q\alpha} \left(\frac{q^2(q-1)}{n-1}\right)B\alpha^3 \nonumber\\
		&\le \exp(-\mu_1qK\alpha/2) V_\lambda(\vz_0^1)+2(cr)^3 \mu_1\kappa_1^2\left(\frac{q(q-1)}{n-1}\right)B\alpha^2. \label{eq:2PL total bound}
	\end{align}
	Note that, in the inequality above, the second term of the right hand side becomes zero when $q=1$. In that case, we can prove exponential decay of $\E[V_\lambda(\vz^{k}_0)]$. Thus, we simply assume $q>1$ hereafter.
	
	\emph{Case 1:} If $K$ is as large as
    \[K> \frac{\kappa_1r^{3/2}}{\sqrt{\mu_1}}\cdot\sqrt{\frac{8c^3e B}{V_\lambda(\vz_0^1)\, q}\left( \frac{q-1}{n-1}\right)},  \quad (e = \exp(1))\]
    we have a step size $\alpha$ as
	\begin{equation*}
		\alpha = \min \left\{ \frac{1}{6Lr\sqrt{q^2 + \frac{q(q-1)}{n-1}A}},~ \frac{2}{\mu_1 q K} \log \left(\clubsuit \right)\right\}, \quad \text{where}~~\clubsuit = \frac{V_\lambda(\vz_0^1) \mu_1 q K^2}{8(cr)^3 \kappa_1^2\kappa_2^6 \left( \frac{q-1}{n-1}\right)B}.
	\end{equation*}
	Due to the lower bound of epoch size $K$, the fraction $\clubsuit$ inside the log factor is indeed greater than $e>1$, which guarantees the step size is positive.
	Putting this to the inequality~\eqref{eq:2PL total bound} and using the fact that $\max\{a,b\}\le a+b$ (for $a,b\ge 0$), we eventually have
	\begin{align*}
		&\E\left[V_\lambda(\vz_n^{K})\right] \\
		&\le V_\lambda(\vz_0^1)\cdot\exp\left(-\frac{K}{12\kappa_1r\sqrt{1 + \left(\frac{q-1}{n-1}\right)\frac{A}{q}}}\right)+ \frac{2\cdot8(cr)^3\kappa_1^2B}{\mu_1 q K^2} \left( \frac{q-1}{n-1}\right)\left[1+\log^2 \left(\clubsuit\right) \right]\\
    	&= V_\lambda(\vz_0^1)\cdot\exp\left(-\frac{K}{12\kappa_1r\sqrt{1 + \left(\frac{q-1}{n-1}\right)\frac{A}{q}}}\right)+ \tilde{\gO}\left(\frac{\kappa_1^2r^3B}{\mu_1 q K^2} \right)\cdot \frac{q-1}{n-1}.
	\end{align*}
	
	\emph{Case 2:} Otherwise, the log factor might have a negative value when $K$ is too small. However, in this case, we have
	\[V_\lambda(\vz_0^1) \le \frac{8(cr)^3e \kappa_1^2 B}{\mu_1 q K^2}\cdot\frac{q-1}{n-1}\,;\quad \alpha = \min \left\{ \frac{1}{84L\kappa_2^2\sqrt{q^2 + \frac{q(q-1)}{n-1}A}},~ \frac{2}{\mu_1 q K}\right\}.\]
	Putting these to the inequality~\eqref{eq:2PL total bound}, we have
	\begin{align*}
	    \E\left[V_\lambda(\vz_n^{K})\right] &\le \frac{8(cr)^3e \kappa_1^2 B}{\mu_1 q K^2} \left( \frac{q-1}{n-1}\right)\left[\exp(-\mu_1qK\alpha/2) + \frac{1}{e}\cdot\left(\mu_1qK\alpha/2\right)^2\right]\\
	    &\le \frac{8(cr)^3e \kappa_1^2 B}{\mu_1 q K^2} \left( \frac{q-1}{n-1}\right) = \gO\left(\frac{\kappa_1^2r^3 B}{\mu_1 q K^2}\right) \cdot \frac{q-1}{n-1}.
	\end{align*}
	The inequality in the last line is due to the fact that $e^{-t} + t^2/e \le 1$ for each $t\in (0,1]$, and that $\mu_1qK\alpha/2 \in (0,1]$.
	
	Combining both \emph{Case 1} and \emph{Case 2}, we conclude the proof of the theorem.
\end{proof}

\section{Proofs for (mini-batch) alternating SGDA-RR: focusing on changes in the proof}
\label{sec:proofs altSGDA RR}
In this appendix, we prove the same convergence rates for altSGDA-RR as the simultaneous update counterpart. Since most of the steps in the proof are similar to those in Appendix~\ref{sec:proofs simSGDA RR}, we only describe which steps change in the proof.

\subsection{Epoch-wise representations and bounding noise terms}

To analyze altSGDA-RR, we modify the notation for epoch-wise updates. The only change is that an update $\vy_t^k\mapsto\vy_{t+1}^k$ uses $\vx_{t+1}^k$ instead of $\vx_{t}^k$. Hence, the definition of $\vh^k$ should be modified. Recall that
\[\vg_t^k (\vz) := \frac{1}{b} \sum_{i\in \gB_t^k} \nabla_1 f_{i}(\vz), \quad \vh_t^k (\vz) := \frac{1}{b} \sum_{i\in \gB_t^k} \nabla_2 f_{i}(\vz),\]
where $\gB_t^k$ is a mini-batch of size $b$ formed at iteration $t$ of epoch $k$. Then, at epoch $k$, by re-definition of $\vh^k$,
\begin{align*}
	\vg^k := \frac{1}{q} \sum_{t=1}^q \vg_t^k (\vx^k_{t-1};\vy_{t-1}^k), \quad &\vh^k := \frac{1}{q} \sum_{t=1}^q \vh_t^k (\vx^k_{\color{red}t};\vy_{t-1}^k).\\
	\vx_{0}^{k+1} = \vx_{0}^k - q\alpha \vg^k, \quad &\vy_{0}^{k+1} = \vy_{0}^k + q\beta \vh^k. \tag{altSGDA-RR}\label{eq:epochwise update alt}
\end{align*}

We still approximate this epoch-wise update rule to a full-batch simultaneous GDA update \eqref{eq:approximately GDA} with step sizes $q\alpha$ and $q\beta$.
Again, we control the ``noise'' terms $\norm{\vg^k-\nabla_1 f (\vz_0^k)}^2$ and $\norm{\vh^k-\nabla_2 f (\vz_0^k)}^2$ not to be large. 
Because of the modification of $\vh^k$, we have a different result for $\norm{\vh^k-\nabla_2 f (\vz_0^k)}^2$ as follows. 

\begin{lemma} \label{lem:noise, alt}
	For mini-batch altSGDA-RR, recall that
	\begin{equation*}
		G_k := \frac{1}{q}\sum_{t=1}^{q} \norm{\vz_{t-1}^k - \vz_0^k}^2.
	\end{equation*}
	If we have Assumption~\ref{ass:smooth}, then we have 
	\begin{equation}
		\norm{\vh^k-\nabla_2 f (\vz_0^k)}^2 \le L^2 G_k + L^2q\alpha^2\norm{\vg^k}^2,\quad \text{ whereas }~ \norm{\vg^k-\nabla_1 f (\vz_0^k)}^2 \le L^2 G_k. \label{eq:xy gradgap alt}
	\end{equation}
\end{lemma}
\begin{proof} Because of $L$-Lipschitz continuity of $\vh^k_t(\cdot;\cdot)$,
	\begin{align*}
		\norm{\vh^k-\nabla_2 f (\vz_0^k)}^2 &= \norm{\frac{1}{q}\sum_{t=1}^{q} \left[\vh_t^k (\vx_{t}^k; \vy_{t-1}^k) - \vh_t^k (\vx_{0}^k; \vy_{0}^k)\right]}^2 \nonumber\\
		&\le \frac{1}{q}\sum_{t=1}^{q}\norm{\vh_t^k (\vx_{t}^k; \vy_{t-1}^k) - \vh_t^k (\vx_{0}^k; \vy_{0}^k)}^2 \nonumber \\
		&\le \frac{L^2}{q}\sum_{t=1}^{q}\norm{\vz_{t-1}^k - \vz_0^k}^2 + \frac{L^2}{q}\norm{\vx_q^k-\vx_0^k}^2 = L^2 G_k + L^2q\alpha^2\norm{\vg^k}^2.
	\end{align*}
	The last ineqaulity holds because $\vx^k_q = \vx^{k+1}_0$.
\end{proof}

\subsection{Bounding noise terms: a bit different proof of Lemma~\ref{lem:boundingEGk}}

We notice that the same result as Lemma~\ref{lem:boundingEGk} holds not only for simultaneous updates but also alternating updates, even though it is not very straightforward. We need to reflect the changes from the previous subsection. That is, we have to be careful when we expand the term $\norm{\vy_t^k-\vy_0^k}^2$ ($0\le t\le q-1$). Unlike the inequality~\eqref{eq:xgap} (in the original proof), we have
\begingroup
\allowdisplaybreaks
\begin{align*}
	&\norm{\vy_t^k - \vy_0^k}^2 = \beta^2 t^2 \norm{\frac{1}{t}\sum_{j=1}^{t} \vh_j^k (\vx_{j}^k; \vy_{j-1}^k) }^2 \nonumber\\
	&\le 3\beta^2 t^2\! \left[\norm{\frac{1}{t}\sum_{j=1}^{t} \left[\vh_j^k  (\vx_{j}^k; \vy_{j-1}^k)-\vh_j^k (\vz_0^k)\right]}^2\! + \norm{\frac{1}{t}\sum_{j=1}^{t} \vh_j^k (\vz_0^k)-\nabla_2 f (\vz_0^k)}^2\!+ \norm{\nabla_2 f(\vz_0^k)}^2\right]\nonumber\\
	&\le 3\beta^2 t^2\!\left[\frac{1}{t}\sum_{j=1}^{t} \norm{\vh_j^k  (\vx_j^k; \vy_{j-1}^k)-\vh_j^k (\vz_0^k)}^2\! + \norm{\frac{1}{t}\sum_{j=1}^{t} \vh_j^k (\vz_0^k)-\nabla_2 f (\vz_0^k)}^2\!+ \norm{\nabla_2 f(\vz_0^k)}^2\right]\nonumber\\
	&\le 3\beta^2 t^2\! \left[\!\frac{L^2}{t}\!\left(\norm{\vx_t^k\!-\!\vx_0^k}^2\!+\! \sum_{j=1}^{t} \norm{\vz_{j\!-\!1}^k\!-\!\vz_0^k}^2\!\right)\!+\!\norm{\frac{1}{t}\sum_{j=1}^{t} \vh_j^k (\vz_0^k)-\nabla_2 f (\vz_0^k)}^2\!+ \norm{\nabla_2 f(\vz_0^k)}^2\right]\nonumber\\
	&\le 3\beta^2 L^2 t\! \sum_{j=1}^{t} \norm{\vz_j^k-\vz_0^k}^2\! + 3\beta^2 t^2 \!\left[\norm{\frac{1}{t}\sum_{j=1}^{t} \vh_j^k (\vz_0^k)-\nabla_2 f (\vz_0^k)}^2\!+ \norm{\nabla_2 f(\vz_0^k)}^2\right]\nonumber\\
	&\le 3\beta^2 L^2 t \cdot qG_k + 3\beta^2 t^2 \left[\norm{\frac{1}{t}\sum_{j=1}^{t} \vh_j^k (\vz_0^k)-\nabla_2 f (\vz_0^k)}^2\!+ \norm{\nabla_2 f(\vz_0^k)}^2\right].
\end{align*}
\endgroup
The second and third inequality holds by Jensen's inequality, and the last inequality holds because $t\le q-1$. The resulting upper bound is identical to the inequality~\eqref{eq:ygap}. Proving this inequality above suffices to show that the conclusion of Lemma~\ref{lem:boundingEGk} also holds for altSGDA-RR, because we eventually take an average along $0\le t \le q-1$ and the other steps in the proof do not utilize the ``order'' (either simultaneous or alternating) of updates.

\subsection{Recurrence inequalities for general smooth nonconvex-P{\L} objective}

In the proof for simSGDA-RR, we applied Lemma~\ref{lem:NaiveReccurence}, Lemma~\ref{lem:boundingEGk}, and the ``small-step-size'' assumptions (three inequalities in~\eqref{eq:step sizes condition 1}) to deduce Lemma~\ref{lem:ReccurenceWithSmallStepSizes}. However, due to Lemma~\ref{lem:noise, alt} that we obtained for altSGDA-RR, we need slightly different assumptions on step sizes rather than \eqref{eq:step sizes condition 1}. 

Fortunately, we notice that the Lemma~\ref{lem:NaiveReccurence} also holds for altSGDA-RR, with a modified version of $\vh^k$. This is because the proof of the lemma does not utilize step-wise updates, while the discrepancy between simultaneous and alternating updates only appears in the step-wise updates. Thus, we have the same result as Lemma~\ref{lem:ReccurenceWithSmallStepSizes}.

\begin{lemma}\label{lem:ReccurenceWithSmallStepSizes, alt}
	Suppose that Assumptions~\ref{ass:smooth}, \ref{ass:bddvar}, \ref{ass:primal,dual}, and \ref{ass:NCPL} hold. Modify the inequalities~\eqref{eq:step sizes condition 1} (from Lemma~\ref{lem:ReccurenceWithSmallStepSizes}) by
	\begin{equation}
		\lambda - \left\{(\lambda+1)(\kappa_2+1)+1\right\}Lq\alpha -L^2q\alpha\beta \ge 0, \quad \beta\le \frac{1}{qL}, \quad \alpha^2+\beta^2 \le \frac{1}{3q(q-1)L^2}. \label{eq:step sizes condition 2}
	\end{equation}
	(In fact, only the first one is different.) Then, the result of Lemma~\ref{lem:ReccurenceWithSmallStepSizes} still holds for mini-batch altSGDA-RR.
\end{lemma}
\begin{proof}
	We first apply Lemma~\ref{lem:noise, alt} to the general bound resulted from Lemma~\ref{lem:NaiveReccurence}:
	\begin{align}
		&V_\lambda(\vz_0^{k+1}) - V_\lambda(\vz_0^{k}) \nonumber\\
		&{\le} -\left(\frac{\lambda + 1}{2}\right) q\alpha \norm{\nabla \Phi (\vx_0^k)}^2 + (\lambda+1)q\alpha \norm{\nabla \Phi (\vx_0^k)-\nabla_1 f(\vz_0^k)}^2 \nonumber\\
		&\quad + \frac{q\alpha}{2} \norm{\nabla_1 f(\vz_0^k)}^2 - \frac{q\beta}{2} \norm{\nabla_2 f(\vz_0^k)}^2 + \frac{(2\lambda+1)\alpha +\beta}{2}qL^2 G_k \nonumber\\
		&\quad - \big[\lambda - \left\{(\lambda+1)(\kappa_2+1)+1\right\}Lq\alpha -L^2q\alpha\beta\big]\frac{q\alpha}{2}\norm{\vg^k}^2 - (1-Lq\beta)\frac{q\beta}{2} \norm{\vh^k}^2. \label{eq:general bound}
	\end{align}

	Hence, the first two inequalities of \eqref{eq:step sizes condition 2} eliminate the last two terms on the right side of the inequality~\eqref{eq:general bound} above:
	\begin{align*}
		V_\lambda(\vz_0^{k+1}) - V_\lambda(\vz_0^{k}) &{\le} -\left(\frac{\lambda + 1}{2}\right) q\alpha \norm{\nabla \Phi (\vx_0^k)}^2 + (\lambda+1)q\alpha \norm{\nabla \Phi (\vx_0^k)-\nabla_1 f(\vz_0^k)}^2 \\
		&\quad + \frac{q\alpha}{2} \norm{\nabla_1 f(\vz_0^k)}^2 - \frac{q\beta}{2} \norm{\nabla_2 f(\vz_0^k)}^2 + \frac{(2\lambda+1)\alpha +\beta}{2}qL^2 G_k.
	\end{align*}
	This is identical to the inequality~\eqref{eq:before expectation} in the proof of Lemma \ref{lem:ReccurenceWithSmallStepSizes}. From this point on, the rest of the proof is exactly identical to Lemma~\ref{lem:ReccurenceWithSmallStepSizes}.
\end{proof}

Lemma~\ref{lem:ReccurenceWithSmallStepSizes, alt} establishes that altSGDA-RR also satisfies a concise bound on the expected per-epoch change of $V_\lambda$, albeit under a slightly different set of assumptions \eqref{eq:step sizes condition 2} on step sizes. 
Using this result, we can prove the convergence rates for altSGDA-RR that are exactly the same as simSGDA-RR. 

\subsection{Small step size assumptions}

It is left to show an altSGDA-RR counterpart for Lemma~\ref{lem:ReccurenceGeneral} which establishes the general recurrence inequality~\eqref{eq:recurrence inequality general}. In fact, the same choice of step sizes as simSGDA-RR, namely
\[0<\beta\le \frac{1}{6L \sqrt{q^2 + \frac{q(q-1)}{n-1}A}}\quad \text{and} \quad \alpha = \frac{\beta}{r} \quad \text{where} ~~r\ge 14\kappa_2^2,\]
actually meets the newly introduced conditions~\eqref{eq:step sizes condition 2}. Among the three inequalities, the only one that needs to be checked is
\[\lambda - \left\{(\lambda+1)(\kappa_2+1)+1\right\}Lq\alpha - L^2q\alpha\beta>0.\]

Note that, regardless of $A\ge0$,
\[\beta \le \frac{1}{6Lq} \quad \text{and}\quad \alpha \le \frac{1}{6Lqr}\le \frac{1}{84L\kappa_2^2q}\]
In this case, 
\begin{align*}
    &\lambda - \left\{(\lambda+1)(\kappa_2+1)+1\right\}Lq\alpha - L^2q\alpha\beta\\
    &\ge 4 - (11\kappa_2 + L\beta)Lq\alpha \ge 4 - \left (11\kappa_2 + \frac{1}{6} \right)\cdot \frac{1}{84\kappa_2^2}>0.
\end{align*}
Therefore, there is no need to modify our choices of $\lambda$ and the step sizes $\alpha, \beta$ for the analysis of altSGDA-RR, and the rest of the proof for simSGDA-RR goes through.

\section{Proofs for lower bound of deterministic full-batch simGDA} 
\label{sec:proofs LB GDA}
In this appendix, we illustrate a comprehensive lower bound for full-batch GDA, which is specific to the choice of step size ratio (Theorem~\ref{thm:LB}). Before we start the proof, we define a class of smooth strongly-convex-strongly concave functions. 
\begin{definition}
Let $\gF(L, \mu_1, \mu_2)$ be the class of functions $f(\vx;\vy)$ with two arguments $\vx$ and $\vy$ of any dimension, which is $L$-smooth, $\mu_1$-strongly-convex in $\vx$, and $\mu_2$-strongly-concave in $\vy$. Let $\kappa_1 = L/\mu_1\ge1$ and $\kappa_2 = L/\mu_2\ge1$ be condition numbers of the function class. Denote the (unique) saddle (or, global minimax) point by $\vz^*=(\vx^*;\vy^*)$.
\end{definition}

We restate and prove the Theorem~\ref{thm:LB} for reader's convenience.
\begin{theorem}[Restatement of Theorem~\ref{thm:LB}]\label{thm:LB 2}
    Suppose $\kappa_1\ge c$ and $\kappa_2\ge c$ for some constant $c>1$. Then, for each step size ratio $r>0$, there exists a function $f\in\gF(L, \mu_1, \mu_2)$ for which simGDA with any step sizes $\alpha$ and $\beta$ of ratio $r=\beta/\alpha$ requires
    \[K=\left\{\begin{array}{ll}
    \Omega\left(\kappa_1 r\log(1/\eps)\right), &\text{\rm if}~ r\ge \kappa_2/c,\\
    \Omega\left(\kappa_1\kappa_2\log(1/\eps)\right), & \text{\rm if}~ c/\kappa_1\le r\le \kappa_2/c,\\
    \Omega((\kappa_2/r)\log(1/\eps)), & \text{\rm if}~ 0<r\le c/\kappa_1
    \end{array}\right.\]
    iterations to achieve either $\norm{\vz_k-\vz^*}^2 \le \eps^2$ or $V_\lambda(\vz_K)\le \eps^2$.
\end{theorem}
\begin{proof}
    The proof is done in case by case, constructing a worst-case function for each of 4 different regimes of step size ratio $r$: {\it (1)} $\mu_1/\mu_2 \le r \le\kappa_2/c$, {\it (2)} $c/\kappa_1 \le r \le \mu_1/\mu_2$, {\it (3)} $r\ge \kappa_2/c$, and {\it (4)} $0<r\le c/\kappa_1$.
    Readers might notice the similarities of the proofs for {\it (1)}$\leftrightarrow${\it (2)} and {\it (3)}$\leftrightarrow${\it (4)}. 
    
    \underline{Case 1.} ($\mu_1/\mu_2 \le r \le\kappa_2/c$). Consider
    \[f^{(1)}(v,x;y):=\frac{\mu_1}{2}v^2+\frac{r\mu_2}{2}x^2-\frac{\mu_2}{2}y^2+\ell xy,\]
    where $\ell^2 = L^2 - r\mu_2^2-L\mu_2|r-1|\ge 0$. Applying Proposition~\ref{prop:quadratic LB}, it can be shown that $f^{(1)}\in \gF(L,\mu_1, \mu_2)$. Also, $\vz^*=(0,0;0)$ is its unique saddle point. Note that, the GDA on $f^{(1)}$ can be written as
    \begin{align*}
        v_{t+1}=\left(1-\frac{\beta\mu_1}{r}\right)v_t, \quad 
        \begin{bmatrix}
            x_{t+1} \\ y_{t+1}
        \end{bmatrix}=\underbrace{\begin{bmatrix}
            1-\beta\mu_2 & -{\beta \ell}/{r} \\ \beta \ell & 1-\beta\mu_2
        \end{bmatrix}}_{\mA} \begin{bmatrix}
            x_{t} \\ y_{t}
        \end{bmatrix} = \mA \begin{bmatrix}
            x_{t} \\ y_{t}
        \end{bmatrix}.
    \end{align*}
    Also, the eigenvalues $\tau$ of $\mA$ is
    \begin{align*}
        \tau &= 1-\beta\mu_2 \pm\sqrt{(1-\beta\mu_2)^2 - \left((1-\beta\mu_2)^2 + \beta^2\ell^2/r\right)} \\
        &= 1-\beta\mu_2 \pm \frac{\beta \ell}{\sqrt{r}}\sqrt{-1}.
    \end{align*}
    The spectral radius (\ie, maximum absolute eigenvalue) is
    \[\rho(\mA) = \sqrt{(1-\beta\mu_2)^2 + \beta^2\ell^2/r}.\]
    Since the eigenvalues are complex conjugates of each other (the magnitudes are the same), both eigenvalues have magnitude $\rho(\mA)$. Then, by Proposition~\ref{prop:spectralradius}, $\rho(\mA)<1$ is necessary for convergence. To this end, we need $\beta>0$ satisfying $\beta <{2\mu_2r}/(r\mu_2^2 + \ell^2)$.
    
    To guarantee $\norm{(v_k,x_k;y_k)-(0,0;0)}^2\le \eps^2$, we need a large enough $k$ to have $v_k^2\le \gO(\eps^2)$. Such a $k$ is required to be at least $\Omega\left(\frac{r}{\beta\mu_1}\log(1/\eps)\right)$. Now note that, since $\mu_1/\mu_2 \le r\le \kappa_2/c$ and $\kappa_2\ge c$, 
    \[\frac{1}{\beta}> \frac{r\mu_2^2 + \ell^2}{2\mu_2 r}=\frac{L^2 - L\mu_2|r-1|}{2\mu_2 r} = \frac{L^2}{2\mu_2r}\left(1-\frac{|r-1|}{\kappa_2}\right) \ge \frac{L^2}{2\mu_2 r}\left(1-\frac{1}{c}\right).\]
    The last inequality is true by minimizing  $\left(1-\frac{|r-1|}{\kappa_2}\right)$ for $r\in [\mu_1/\mu_2, \kappa_2/c]$. If $r\ge1$, it has smaller value when $r$ is larger: by taking $r=\kappa_2/c$, we have $1-\frac{\kappa_2/c-1}{\kappa_2}\ge 1-\frac{1}{c}$. Otherwise ($r<1$), which is possible only when $\mu_1<\mu_2$, the term has smaller value when $r$ is smaller: by taking $r=\mu_1/\mu_2$, we have $1+\frac{\mu_1/\mu_2 - 1}{\kappa_2} = 1+\frac{\mu_1-\mu_2}{L}\ge 1-\frac{1}{\kappa_2}\ge 1-\frac{1}{c}$.
    Thus, we eventually need $\Omega\left(\frac{L^2}{\mu_1\mu_2}\log(1/\eps)\right)$ iterations.
    
    \underline{Case 2.} ($c/\kappa_1 \le r \le \mu_1/\mu_2$). Consider
    \[f^{(2)}(x;y,w):=\frac{\mu_1}{2}x^2-\frac{\mu_1}{2r}y^2+\tilde{\ell} xy-\frac{\mu_2}{2}w^2,\]
    where $\tilde{\ell}^2 = L^2 - \mu_1^2/r-L\mu_1|1-1/r|\ge 0$. Applying Proposition~\ref{prop:quadratic LB}, it can be shown that $f^{(2)}\in \gF(L,\mu_1, \mu_2)$, and $\vz^*=(0;0,0)$ is its unique saddle point.
     Note that, the GDA on $f^{(2)}$ can be written as
    \begin{align*} 
        \begin{bmatrix}
            x_{t+1} \\ y_{t+1}
        \end{bmatrix}=\underbrace{\begin{bmatrix}
            1-\beta\mu_1/r & -{\beta \ell}/{r} \\ \beta \ell & 1-\beta\mu_1/r
        \end{bmatrix}}_{\mB} \begin{bmatrix}
            x_{t} \\ y_{t}
        \end{bmatrix} = \mB \begin{bmatrix}
            x_{t} \\ y_{t}
        \end{bmatrix}, \quad 
        w_{t+1}=\left(1-\beta\mu_2\right)w_t.
    \end{align*}
    Also, the eigenvalues $\tau$ of $\mB$ is
    \begin{align*}
        \tau &= 1-\beta\mu_1/r \pm\sqrt{(1-\beta\mu_1/r)^2 - \left((1-\beta\mu_1/r)^2 + \beta^2\ell^2/r\right)} \\
        &= 1-\frac{\beta\mu_1}{r} \pm \frac{\beta \ell}{\sqrt{r}}\sqrt{-1}.
    \end{align*}
    The spectral radius is
    \[\rho(\mB) = \sqrt{(1-\beta\mu_1/r)^2 + \beta^2\ell^2/r}.\]
    Since the eigenvalues are complex conjugates of each other (the magnitudes are the same), both eigenvalues have magnitude $\rho(\mB)$. Then, by Proposition~\ref{prop:spectralradius}, $\rho(\mB)<1$ is necessary for convergence. To this end, we need $\beta>0$ satisfying $\beta <{2\mu_1}/({\mu_1^2/r + \ell^2})$.
    
    To guarantee $\norm{(x_k;y_k,w_k)-(0;0,0)}^2\le \eps^2$, we need a large enough $k$ to have $w_k^2\le \gO(\eps^2)$. Such a $k$ is required to be at least $\Omega\left(\frac{1}{\beta\mu_2}\log(1/\eps)\right)$. Now note that, since $c/\kappa_1 \le r \le \mu_1/\mu_2$ and $\kappa_1\ge c$, 
    \[\frac{1}{\beta}> \frac{\mu_1^2/r + \ell^2}{2\mu_1 }=\frac{L^2 - L\mu_1|1-1/r|}{2\mu_1} = \frac{L^2}{2\mu_1}\left(1-\frac{|1-1/r|}{\kappa_1}\right) \ge \frac{L^2}{2\mu_1 }\left(1-\frac{1}{c}\right).\]
    The last inequality is true by minimizing  $\left(1-\frac{|1-1/r|}{\kappa_1}\right)$ for $r\in [c/\kappa_1, \mu_1/\mu_2]$.
    If $1>1/r$, which is possible only when $\mu_1>\mu_2$, it has smaller value when $r$ is larger: by taking $r=\mu_1/\mu_2$, we have $1-\frac{1-\mu_2/\mu_1}{\kappa_1}=1-\frac{\mu_1-\mu_2}{L}\ge1-\frac{1}{\kappa_1}\ge1-\frac{1}{c}$.
    Otherwise ($1<1/r$), the term has smaller value when $r$ is smaller: by taking $r=c/\kappa_1$, we have $1+\frac{1-\kappa_1/c}{\kappa_1} \ge 1-\frac{1}{c}$.
    Thus, we eventually need $\Omega\left(\frac{L^2}{\mu_1\mu_2}\log(1/\eps)\right)$ iterations.
    
    \underline{Case 3.} ($r\ge \kappa_2/c$). Consider $f^{(3)}(x;y)=\frac{\mu_1}{2}x^2-\frac{L}{2}y^2$. Clearly, $f^{(3)}\in \gF(L,\mu_1,L)\subset \gF(L,\mu_1, \mu_2)$ and $\vz^*=(0,0)$ is its unique saddle point. The GDA on $f^{(3)}$ can be written as
    \begin{align*}
        x_{k+1} = \left(1-\frac{\beta\mu_1}{r}\right)x_k, \quad y_{k+1} = \left(1-\beta L\right)y_k.
    \end{align*}
    To guarantee $\norm{(x_k;y_k)-(0,0)}^2\le \eps^2$, we need a large enough $k$ to have $x_k^2\le \gO(\eps^2)$. Such a $k$ is required to be at least $\Omega\left(\frac{r}{\beta\mu_1}\log(1/\eps)\right)$. Also, we need $\beta<2/L$ to guarantee $y_k\rightarrow0$ (\ie, otherwise, it diverges). Combining these facts, we eventually need $\Omega\left(\frac{Lr}{\mu_1}\log(1/\eps)\right)$ iterations.
    
    \underline{Case 4.} ($0<r\le c/\kappa_1$). Consider $f^{(4)}(x;y)=\frac{L}{2}x^2-\frac{\mu_2}{2}y^2$. Clearly, $f^{(4)}\in \gF(L,L,\mu_2)\subset \gF(L,\mu_1, \mu_2)$ and $\vz^*=(0,0)$ is its unique saddle point. The GDA on $f^{(4)}$ can be written as
    \begin{align*}
        x_{k+1} = \left(1-\frac{\beta L}{r}\right)x_k, \quad y_{k+1} = \left(1-\beta \mu_2\right)y_k.
    \end{align*}
    To guarantee $\norm{(x_k;y_k)-(0,0)}^2\le \eps^2$, we need a large enough $k$ to have $y_k^2\le \gO(\eps^2)$. Such a $k$ is required to be at least $\Omega\left(\frac{1}{\beta\mu_2}\log(1/\eps)\right)$. Also, we need $\beta<2r/L$ to guarantee $x_k\rightarrow0$ (\ie, otherwise, it diverges). Combining these facts, we eventually need $\Omega\left(\frac{L}{r\mu_2}\log(1/\eps)\right)$ iterations.
    
    Lastly, we note that the lower iteration complexity bound in terms of the potential function $V_\lambda$ is equivalent to the complexity in terms of squared distance norm from the (unique) saddle point $\vz^*$, up to constant factors.
    This is proved in Lemma~\ref{lem:CriteriaEquivalence} that we defer its proof.
\end{proof}

Here are the postponed/omitted proofs from the proof above.
\begin{proposition}\label{prop:spectralradius}
    For a square matrix $\mA\in \R^{m\times m}$ and a sequence of $m$-dimensional vectors $(\vv_k)$, the matrix iteration $\vv_{k+1}  = \mA \vv_k$ converges to $\vv_k\rightarrow \bm0$ if and only if the spectral radius (\ie, maximum absolute eigenvalue) of $\rho(\mA)$ of $\mA$ is less than 1. Furthermore, its convergence speed is characterized by $\gO((\rho(\mA)+\eps)^k)$ for any (arbitrarily small) $\eps>0$.
\end{proposition}
\begin{proof}
    See \citet[Theorem~5.6.10-12]{Horn2012-du}.
\end{proof}

\begin{proposition} \label{prop:quadratic LB}
    Let $\mu_1$, $\mu_2$, and $L$ be positive numbers such that $L\ge \max\{\mu_1, \mu_2\}$. Consider a quadratic function $f$ on $\R\times\R$ defined by 
    \[f(x;y) = \frac{\mu_1}{2} x^2 - \frac{\mu_2}{2} y^2 + \ell xy, \quad \text{where}~~\ell^2 \le L^2 - \mu_1\mu_2 - L|\mu_1-\mu_2|.\]
    Then, $f\in \gF(L, \mu_1, \mu_2)$, and its unique saddle point is $\vz^*=(0,0)$.
\end{proposition}
For example, if $\mu_1 \ge \mu_2$, $\ell^2 = (L-\mu_1)(L+\mu_2)$ is enough to guarantee $L$-smoothness.
\begin{proof}
    The strong-convex-strong-concavity is trivially true. Note that the gradient and hessian of $f$ is
    \[\nabla f(x;y) = \mH [x~~y]^\top, \quad \mH = \begin{bmatrix}\mu_1 & \ell \\ \ell & -\mu_2\end{bmatrix}.\]
    Since $\mH$ is a non-singular matrix, $f$ has a unique stationary point at origin ($x=0,y=0$). By Proposition~\ref{prop:optimality}, it is also a unique saddle \& global minimax point.
    
    For any two distinct points $\vz_1 = (x_1; y_1)$ and $\vz_2 = (x_2;y_2)$ in $\R\times\R$,
    \[\frac{\norm{\nabla f(\vz_1) - \nabla f(\vz_2)}}{\norm{\vz_1-\vz_2}} = \frac{\norm{\mH(\vz_1-\vz_2)}}{\norm{\vz_1-\vz_2}}\le \norm{\mH}_2,\]
    where $\norm{\mH}_2$ is spectral norm (\ie, maximum singular value) of $\mH$. We would like to show that $\norm{\mH}_2\le L$. To this end, it is enough to verify the following two inequalities:
    \begin{align*}
        \det(L^2\mI - \mH\mH^\top) &= L^4 -(\mu_1^2+\mu_2^2+2\ell^2)L^2 + (\mu_1\mu_2+\ell^2)^2 \ge 0,\\
        \operatorname{trace}(\mH\mH^\top)/2 &= (\mu_1^2+\mu_2^2)/2+\ell^2 \le L^2.
    \end{align*}
    This is because the characteristic polynomial of $\mH\mH^\top$, or $\det(\omega\mI - \mH\mH^\top)$, is a quadratic polynomial of $\omega$, and its maximum root should not be greater than $L^2$.
    Let $\ell^2 =  L^2 - \mu_1\mu_2 + a$ for some $a\in\R$. Plugging this $\ell^2$ into both inequalities above, we get
    \begin{align*}
        a^2-(\mu_1-\mu_2)^2L^2 \geq 0 \quad \text{and}\quad
        a\le-(\mu_1-\mu_2)^2/2,
    \end{align*}
    respectively. One can check that $a = -L|\mu_1 - \mu_2|$ is the largest possible $a$ satisfying both inequalities above.
    This proves the proposition.
\end{proof}

Subsequently, we show that if our convergence rate is exponential, then the iteration complexity in terms of $\norm{\vz-\vz^*}^2$ is equivalent to that in terms of $V_\lambda(\vz) = \lambda[\Phi(\vx)-\Phi^*] + [\Phi(\vx)-f(\vz)]$ for P{\L}($\Phi$)-P{\L} problem, up to constant factors. This also applies to the function class $\gF(L, \mu_1, \mu_2)$ since it is a subclass of smooth P{\L}($\Phi$)-P{\L} functions ($\because$ Propositions~\ref{prop:PLQGEB}~and~\ref{prop:Phi PL}).
\begin{lemma}\label{lem:CriteriaEquivalence}
    Suppose $f(\vx;\vy)$ is an $L$-smooth function satisfying $\vy$-side $\mu_2$-P{\L} condition and primal $\mu_1$-P{\L} condition (\ie, P{\L}($\Phi$)-P{\L}). Suppose $\vz^*=(\vx^*;\vy^*)$ is a global minimax point of $f$. Then, it satisfies
    \[\frac{\lambda\mu_1\mu_2^2}{2(\lambda\mu_1\mu_2+2L^2)} \norm{\vz-\vz^*}^2 \le V_\lambda(\vz) \le \frac{(\lambda+1)L^3}{\mu_2^2}\norm{\vz-\vz^*}^2.\]
\end{lemma}
We remark that the second inequality also holds for general smooth nonconvex-P{\L} problems. 
\begin{proof}
    Let $\kappa_1 = L/\mu_1$ and $\kappa_2 = L/\mu_2$ be condition numbers.
    By the conditions of $f$ (smoothness and P{\L} conditions), for any $\vx$ and $\vy$,
    \begin{align*}
        \frac{\mu_1}{2}\norm{\vx-\vx^*}^2 &
        \stackrel{\text{Prop.~\ref{prop:Phi PL}}}{\le} \Phi(\vx)-\Phi^* \stackrel{\text{Prop.~\ref{prop:Phi smooth}}}{\le} \frac{L(\kappa_2+1)}{2}\norm{\vx-\vx^*}^2,\\
        \frac{\mu_2}{2}\norm{\vy-\vy^*(\vx)}^2 &\stackrel{\text{Ass.~\ref{ass:NCPL}}}{\le} \Phi(\vx)-f(\vx;\vy) \stackrel{\text{Ass.~\ref{ass:smooth}}}{\le} \frac{L}{2}\norm{\vy-\vy^*(\vx)}^2,
    \end{align*}
    where $\vy^*(\vx)$ is a projection of $\vy$ to $\argmax_{\vy'} f(\vx;\vy')$. In particular, $\vy^*(\vx^*) = \vy^*$. Since $\vy^*(\vx)$ is a function of $\vx$ and can differ from $\vy^*$, we need to bound the term $\norm{\vy-\vy^*(\vx)}^2$ using $\norm{\vx-\vx^*}^2$ and $\norm{\vy-\vy^*}^2$. To upper-bound the term $\norm{\vy-\vy^*(\vx)}^2$, note that,
    \begin{align*}
        \norm{\vy-\vy^*(\vx)}^2 &\le \left(\norm{\vy-\vy^*(\vx^*)} + \norm{\vy^*(\vx)-\vy^*(\vx^*)}\right)^2 \\
        &\le \left(\norm{\vy-\vy^*} + \kappa_2\norm{\vx-\vx^*}\right)^2\\
        &\le \left(1+\kappa_2^2\right)\left(\norm{\vy-\vy^*}^2 + \norm{\vx-\vx^*}^2\right).
    \end{align*}
    The first inequality holds by triangle inequality, the second inequality holds by Proposition~\ref{prop:y*(x) Lipschitz}, and the last inequality holds by Cauchy-Schwarz inequality.\footnote{$(ax+by)^2\le (a^2+b^2)(x^2+y^2)$ for real numbers $a,b,x,y$.}
    To lower-bound in a similar way, note that for any constant $a>0$,
    \begin{align*}
        \norm{\vy-\vy^*}^2 &\le \left(\norm{\vy-\vy^*(\vx)} + \norm{\vy^*(\vx)-\vy^*(\vx^*)}\right)^2 \\
        &\le \left(\norm{\vy-\vy^*(\vx)} + \frac{\kappa_2}{\sqrt{a}}\cdot\sqrt{a}\norm{\vx-\vx^*}\right)^2\\
        &\le \left(1+\frac{\kappa_2^2}{a}\right)\left(\norm{\vy-\vy^*(\vx)}^2 + a\norm{\vx-\vx^*}^2\right).\\
        \therefore \norm{\vy-\vy^*(\vx)}^2 &\ge \frac{1}{1+\kappa_2^2/a}\norm{\vy-\vy^*}^2 - a\norm{\vx-\vx^*}^2.
    \end{align*}
    Now we can prove the inequalities in the lemma. We first show the second one.  Applying $\kappa_2\ge1$ multiple times,
    \begin{align*}
        V_\lambda (\vx;\vy) &=\lambda[\Phi(\vx)-\Phi^*] + [\Phi(\vx)-f(\vz)]\\
        &\le \frac{\lambda L(\kappa_2+1)}{2}\norm{\vx-\vx^*}^2+\frac{L}{2}\norm{\vy-\vy^*(\vx)}^2\\
        &\le \left(\frac{\lambda L(\kappa_2+1)}{2} + \frac{L(1+\kappa_2^2)}{2}\right)\norm{\vx-\vx^*}^2 + \frac{L(1+\kappa_2^2)}{2}\norm{\vy-\vy^*}^2\\
        &\le (\lambda+1)L\kappa_2^2 \left(\norm{\vx-\vx^*}^2+\norm{\vy-\vy^*}^2\right) = \frac{(\lambda+1)L^3}{\mu_2^2}\norm{\vz-\vz^*}^2.
    \end{align*}
    To show the first inequality of the lemma, let $a=\frac{\lambda \mu_1}{2\mu_2}$.
    \begin{align*}
        V_\lambda (\vx;\vy) &\ge \frac{\lambda \mu_1}{2}\norm{\vx-\vx^*}^2+\frac{\mu_2}{2}\norm{\vy-\vy^*(\vx)}^2\\
        &\ge \left(\frac{\lambda \mu_1}{2} - \frac{\mu_2 a}{2}\right)\norm{\vx-\vx^*}^2 + \frac{\mu_2}{2(1+\kappa_2^2/a)}\norm{\vy-\vy^*}^2\\
        &\ge \frac{\lambda \mu_1}{4} \norm{\vx-\vx^*}^2 + \frac{\lambda\mu_1}{4(a+\kappa_2^2)}\norm{\vy-\vy^*}^2\\
        &\ge \frac{\lambda\mu_1}{4(a+\kappa_2^2)} \left(\norm{\vx-\vx^*}^2+\norm{\vy-\vy^*}^2\right) = \frac{\lambda\mu_1\mu_2^2}{2(\lambda\mu_1\mu_2+2L^2)}\norm{\vz-\vz^*}^2.
    \end{align*}
    This concludes the proof.
\end{proof}
The equivalence of iteration complexities for achieving $\norm{\vz_K-\vz^*}^2 \le \eps^2$ or $V_\lambda(\vz_K)\le \eps^2$ is quite straightforward from this lemma, as long as the convergence speed is exponential. For example, suppose we have a upper convergence bound $\norm{\vz_K-\vz^*}^2 \le a \exp(-K/r)$ for some constants $a,r>0$. This implies a upper iteration complexity bound $K = \gO(r\log(1/\eps))$ sufficient to achieve $\norm{\vz_K-\vz^*}^2 \le \eps^2$. Then by Lemma~\ref{lem:CriteriaEquivalence}, we also have $V_\lambda(\vz_K)^2 \le a' \exp(-K/r)$ where $a'={a(\lambda+1)L^3}/{\mu_2^2}$ is also a constant. This implies a lower iteration complexity bound $K = \gO(r\log(1/\eps))$ as well, sufficient to achieve $V_\lambda(\vz_K)^2 \le \eps^2$. The other way of complexity translation operates with a similar logic.

\section{Remark on smoothness assumptions and Lower bound of with-replacement SGD(A)}
\label{sec:LB with-replacement}
During the discussion phase of the conference, a reviewer raised a question about whether or not the \emph{component smoothness} (Assumption~\ref{ass:smooth}) is more crucial than the without-replacement component sampling for faster convergence.
However, we would like to claim that the component smoothness alone is not sufficient for improving the convergence rate for with-replacement SGD(A). To this end, we provide some formal results on lower convergence bounds. For simplicity, we use mini-batches of size 1 throughout this appendix. 

Firstly, the theorem below provides a lower bound on with-replacement SGD for minimization problems. Readers can also verify that an analogous lower bound holds for SGD with unbiased and independently sampled gradient oracle for more general stochastic minimization problems. The proof will appear later in this appendix.
\begin{theorem} \label{thm:LB with-replacement}
	For any step size $\eta>0$, there exists a real-valued strongly-convex function $f(\vx)$ defined on $\R^d$ with $f^* := \min_{\vx} f(\vx)$, satisfying:
	\begin{enumerate}
		\item $f$ consists of $n>1$ smooth component functions $f_i$: $f(\vx)= \frac{1}{n}\sum_{i=1}^n f_i (\vx)$, where each component $f_i$ is smooth;
		\item After running $T> 1$ iterations of with-replacement SGD (with mini-batch size 1) starting from $x_0\in\R^d$, the last iterate $x_T$ satisfies $\mathbb{E}[f(\vx_T)-f^*] \ge \Omega(1/T)$, where the expectation is taken with respect to the randomness of i.i.d. index choice at each iteration.
	\end{enumerate}
\end{theorem}
Next, we show this theorem naturally induces a convergence lower bound for the minimax counterpart: \textit{with-replacement SGDA}.
Consider a (finite-sum) minimax problem $\min_x \max_y g(\vx,\vy) := f(\vx)-f(\vy)$, where $f=\frac{1}{n}\sum_{i=1}^n f_i$ is a worst-case function in the proof of Theorem~\ref{thm:LB with-replacement}. Here, the minimax problem on $g$ can be solved by minimizing $f$. Moreover, since the primal function $\Phi(\vx):= \max_y g(\vx,\vy)$ associated with $g$ is in fact the same as $f(\vx)-f^\ast$, the potential function $V_\lambda (\vx,\vy) := \lambda[\Phi(\vx)-(\min_x \Phi(\vx))] + [\Phi(\vx)-g(\vx,\vy)]$ becomes the same as $\lambda(f(\vx)-f^\ast) + (f(\vy)-f^\ast)$ for a constant $\lambda>0$. Combining these facts, we can immediately obtain the following lower convergence bound of with-replacement SGDA.
\begin{corollary} \label{coro:LB with-replacement}
	There exists a strongly-convex-strongly-concave function $g(\vx,\vy):=\frac{1}{n}\sum_{i=1}^n g_i(\vx,\vy)$ consisting of $n$ smooth component functions $g_i$, where the last iterate $(\vx_T,\vy_T)$ of with-replacement SGDA satisfies $\mathbb{E}[V(\vx_T,\vy_T)]\ge \Omega(1/T)$.
\end{corollary}
Corollary~\ref{coro:LB with-replacement} formally proves that with-replacement SGDA on strongly-convex-strongly-concave minimax problems with smooth components has a worst-case convergence rate $\Omega(1/T)$. This in fact matches the $\mathcal{O}(1/T)$ upper bound obtained for primal-P{\L}-P{\L} problems by \citet{yang2020global}. Considering that strongly-convex-strongly-concave functions form a strict subset of primal-P{\L}-P{\L} functions, Corollary~\ref{coro:LB with-replacement} establishes that adding component smoothness assumption does not provide further speed up for with-replacement SGDA.

In contrast, our theoretical result in Theorem~\ref{thm:2PL} shows that SGDA-RR has a much faster convergence rate $\mathbb{E} [V_\lambda] \le \tilde{\mathcal{O}}(\frac{1}{nK^2})$ for primal-P{\L}-P{\L} minimax problems, where $K$ is the number of epochs. One can check that our $\tilde{\mathcal{O}}(\frac{1}{nK^2})$ bound is faster than the tight convergence rate $\Theta(1/T)$ of with-replacement SGDA by simply plugging in $T = nK$. In light of Corollary~\ref{coro:LB with-replacement} we proved, we can now claim that the improvement can be solely attributed to RR.

Although we do not provide a lower bound for more general nonconvex-P{\L} problems here, we believe the more challenging case of nonconvex-P{\L} lower bound is a topic for another separate paper. Nonetheless, we conjecture that the speed up by SGDA-RR in nonconvex-P{\L} settings is also due to the effect of RR, not component smoothness.

From now on, we provide the postponed proof of Theorem~\ref{thm:LB with-replacement}.

\begin{proof}[Proof of Theorem~\ref{thm:LB with-replacement}]
We construct worst-case functions with quadratic functions on $\mathbb{R}$, which are clearly $ L$-smooth for a fixed constant $ L>0$. Then, it is easy to extend the logic to the functions with domains of higher dimensions. Let $x_0\in \mathbb{R}$ be the initial iterate.
	
\textbf{Case 1 ($\frac{1}{ L T}\le \eta\le\left( \frac{2}{ L}-\frac{1}{ L T} \right)$).} Note that the condition on the step size, $\frac{1}{ L T}\le \eta\le\left( \frac{2}{ L}-\frac{1}{ L T} \right)$, is equivalent to an inequality $(1-\eta L)^2 \le (1-1/T)^2$.
    
We first assume $n$ is an even number. We will encounter the case with an odd $n>1$ a bit later. Consider $f(x)=\frac{L}{2} x^2$ consisting of even number of components $f_i$'s defined by
\[f_i(x) =
\begin{cases}
    \frac{ L}{2} x^2 + \nu x, & (i\le \frac{n}{2}),\\
    \frac{ L}{2} x^2 - \nu x, & (i\ge \frac{n}{2}+1),
\end{cases}\]
for some number $\nu\in\mathbb{\R}$.
At each iteration $t\ge 1$, we choose a component index $i(t)\stackrel{\mathrm{i.i.d.}}{\sim} \mathrm{Unif}([n])$ (with-replacement sampling). Then we can write the chosen component function at iteration $t$ as $f_i(t) = \frac{ L}{2}x^2 - s_t \nu x$ for some i.i.d. random variable $s_t \sim \mathrm{Unif}(\{\pm 1\})$. Accordingly, an SGD step can be written as
\[x_t = x_{t-1} - \eta \nabla f_{i(t)} (x_{t-1}) = (1-\eta L) x_{t-1} + \eta s_t \nu.\]
By applying telescopic sum, we have
\[x_T = (1-\eta L)^T x_0 + \eta \nu \sum_{t=1}^T (1-\eta L)^{(T-t)} \cdot s_{t}.\]
Taking squares and expectations (with respect to the random variables $s_1, \ldots, s_T$) to both sides, we have
\[\mathbb{E}[x_T^2] = (1-\eta L)^{2T}x_0^2 + \eta^2\nu^2 \sum_{t=1}^T (1-\eta L)^{2(T-t)},\]
by applying the fact that $s_t$'s are zero-mean independent random variables with absolute values 1:
\[\mathbb{E}[s_t\cdot s_{t'}] =
\begin{cases}
    0, & t\ne t'\quad (\because \text{independent}),\\
    1, & t=t' \quad (\because s_t^2=1).
\end{cases}\] 
We calculate the sum above as follows: since $(1-\eta L)^2 \le (1-1/T)^2$ and $(1-1/T)^T\le e^{-1}$,
\[\sum_{t=1}^T (1-\eta L)^{2(T-t)} = \frac{1-(1-\eta L)^{2T}}{1-(1-\eta L)^2}\ge \frac{1-(1-\frac{1}{T})^{2T}}{2\eta L(1-\frac{\eta L}{2})}\ge \frac{1-e^{-2}}{2\eta L}.\]
With this inequality, and since $(1-\eta L)^{2T}x_0^2\ge 0$,  we can lower-bound the expectation $\mathbb{E}[x_T^2]$:
\[\mathbb{E}[x_T^2] \ge \eta^2 \nu^2 \cdot \frac{1-e^{-2}}{2\eta L} = \frac{(1-e^{-2})\nu^2}{2 L}\eta\ge \frac{(1-e^{-2})\nu^2}{2 L^2 T}.\]
Since $f$ has a minimum $f^*=0$ at $x=0$, we eventually have
\[\mathbb{E}[f(x_T)-f^*]=\frac{ L}{2}\mathbb{E}[x_T^2]\ge \frac{(1-e^{-2})\nu^2}{4 L T}=\Omega \left( \frac{\nu^2}{ L T} \right).\]

Now we consider the case when the number of components $n>1$ is odd. Consider $f_n(x)\equiv 0$ and let the remaining $n-1$ components be the same as the case above (with an even number of components). Note that the zero-component $f_n$ does not affect the trajectory of SGD ({\it i.e.}, the points visited by SGD) and the optimality of $f$ ($f^*=0$ at $x=0$), while the whole objective function becomes $f(x)=\frac{n-1}{n}\cdot \frac{L}{2}x^2$. Thus, it can be easily shown that the $\Omega \left( \frac{\nu^2}{ L T} \right)$ lower bound also holds. 

\textbf{Case 2 ($0<\eta<\frac{1}{ L T}$ or $\eta>\left( \frac{2}{ L}-\frac{1}{ L T} \right)$).} From the condition on the step size, we have $(1-\eta L)^2 > (1-1/T)^2.$ Consider $f_i(x) = \frac{ L}{2} x^2$ for every $i\in [n]$: every components are the same. In this case, we show that the last iterate of SGD is bounded below by a constant with respect to $T>1$.

At each iteration $t\ge 1$, we obtain $x_t = (1-\eta L)x_{t-1}$ by a step of SGD. Then, applying $T\ge 2$,
\[x_T^2 = \left(1-\eta L\right)^{2T} \cdot x_0^2 > \left( 1-\frac{1}{T} \right)^{2T} x_0^2 \ge \left( 1-\frac{1}{2} \right)^4 x_0^2 = \frac{x_0^2}{16}.\]
Since $f(x)=\frac{1}{n}\sum_{i=1}^n f_i(\vx)=\frac{ L}{2} x^2$ has a minimum $f^*=0$ at $x=0$, we have $$f(x_T)-f^*>\frac{ Lx_0^2}{32} = \Omega( 1)\cdot L x_0^2.$$

\end{proof}

\section{Experiments: quadratic games}
\label{sec:experimental details:quadratic}
In this appendix, we provide a more detailed illustration of our numerical evaluations on quadratic games introduced in Section~\ref{sec:experiments}. Recall that the objective function $f$ and its component functions $f_i$ are given in Equation~\eqref{eq:quadraticgame} as
\begin{align*}
    f(\vx;\vy) &= \frac{1}{2}\vx^\top \mA \vx + \vx^\top \mB \vy - \frac{1}{2}\vy^\top\mC\vy, \\
    f_i(\vx;\vy) &= \frac{1}{2}\vx^\top \mA_i \vx + \vx^\top \mB_i \vy - \frac{1}{2}\vy^\top\mC_i\vy + \vu_i^\top \vx- \vv_i^\top \vy. %
\end{align*}

We choose the same dimensions for the variables $\vx\in\R^{d_x}$ and $\vy\in\R^{d_y}$: we set $d_x=d_y=d$.

\subsection{Parameter choices}

To sample the matrix $\mC=\frac{1}{n}\sum_{i=1}^n\mC_i\in\R^{d}$ satisfying that $\mu_C \mI_d \preceq \mC$ and $\|\mC_i\|_2\le L_C$, we first randomly generate an orthogonal matrix $\mQ_C\in \R^{d\times d}$ (\ie, $\mQ_C\mQ_C^\top = \mI_d$), by taking advantage of the QR-decomposition of a random matrix. Then, we generate the eigenvalues of $\mC_i$'s as follows. We sample the entries of $n$ vectors $\bm\lambda^C_i\in \R^d$ ($i\in[n]$) uniformly from the interval $[\mu_C, L_C]$. We add some level of perturbations to some entries of each $\bm\lambda^C_i$; we replace some entries to the numbers in an interval $[-L_C, \mu_C]$, keeping the entries of the vector $\frac{1}{n}\sum_{i=1}^{n}\bm\lambda^C_i$ in the interval $[\mu_C, L_C]$. Finally, we define $\mC_i = \mQ_C\bm\Lambda^C_i\mQ_C^\top$ where $\bm\Lambda^C_i= \operatorname{diag}(\bm\lambda^C_i)$. Because of the perturbation step, some $\mC_i$'s are not positive definite, thereby some components $f_i$'s become non-(strongly-)concave in $\vy$. 

Next, we sample the matrix $\mB_i$'s. There are no requirements for $\mB$ but $\norm{\mB_i}_2\le L_B$; $\mB_i$'s are even not necessarily symmetric when $d_x\ne d_y$. Thus, we first generate the orthogonal matrices $\mU^B_i$ and  $\mV^B_i$ by taking advantage of the singular value decomposition of random matrices. Then, we generate the singular values of $\mB_i$'s by sampling the entries of $n$ vectors $\bm\sigma^B_i$ uniformly from the interval $[0, L_C]$. After that, we define $\mB_i=\mU^B_i\bm\Sigma^B_i\mV^B_i$ where $\bm\Sigma^C_i= \operatorname{diag}(\bm\sigma^C_i)$. We typically want to take a larger $L_B$ than $L_C$ to strengthen the interaction term $\vx^\top\mB\vy$.

Recall that the primal function $\Phi$ associated with $f$ is explicitly written as
\begin{align}
    \Phi(\vx) = \max_{\vy\in \R^d} f(\vx;\vy) =\frac{1}{2}\vx^\top \left(\mA + \mB\mC^{-1}\mB^\top\right)\vx := \frac{1}{2}\vx^\top \mM\vx. \label{eq:quadratic primal}
\end{align}

Note that the inverse of $\mC$ can be efficiently computed as $\mC^{-1}=\mQ_C(\bm\Lambda^C)^{-1}\mQ_C^\top$.

Before generating the matrices $\mA_i$'s, we first generate $\mM_i$'s satisfying that  $\frac{1}{n}\sum_{i=1}^n \mM_i =\mM$ and the nonzero eigenvalues of positive \emph{semi}definite $\mM$ are in the interval $[\mu_M, L_M]$. The process of sampling $\mM_i$'s is almost identical to how to sample $\mC_i$'s. One notable difference is, $\mM_i$'s and $\mM$ are forced to have $r(<d)$ zero eigenvalues: this makes $\mM$ a positive semidefinite (but not strictly positive definite) matrix of rank $d-r$. Moreover, we get the $\mu_M$-P{\L}($\Phi$) condition in $\vx$ as follows:

\begin{proposition}
    Consider a positive semidefinite matrix $\mM\in \R^d$. If the smallest nonzero eigenvalue of $\mM$ is $\mu$, then $\Phi(\vx):=\frac{1}{2}\vx^\top\mM\vx$ is $\mu$-P{\L} in $\vx$. Also, $\Phi^* =\min_\vx \Phi(\vx)=0$.
\end{proposition}

\begin{proof}
    Apply the eigendecomposition of $\mM$: $\mM=\mQ\bm\Lambda\mQ^\top$. Let $\overline{\mM}=\bm\Lambda^{1/2}\mQ^\top$. Then, we have $\Phi(\vx)=\frac{1}{2}\norm{\overline{\mM}\vx}^2$, which implies that $\Phi(\vx)\ge 0$ ($\forall \vx$) and in fact $\Phi^*=0$.  Note that $\frac{1}{2}\norm{\vx}^2$ is 1-strongly convex. Also, the minimum nonzero singular value of $\overline{\mM}$ is $\sqrt{\mu}$ ({\small $\because \mM=\overline{\mM}^\top\overline{\mM}$}). Therefore, by the proof of Proposition~\ref{prop: SC compose linear is PL}, $\Phi(\vx)$ is a $\mu$-P{\L} function of $\vx$. Lastly, we note that $\Phi(\vx)$ is not strongly convex in general, especially when $\mM$ is a rank-deficient matrix.
\end{proof}
Typically, the spectral norm $\|\mM\|_2$ is known to be bounded above by $\|\mA\|_2 + L_B^2/\mu_C$ in \emph{worst-case} \citep{nouiehed2019solving, li2022convergence}. However, since we sample $\mM$ without knowing the exact form of $\mA_i$'s while we want to control the spectral norm $\|\mA_i\|_2$ not too large (for smoothness of $f_i$), we (empirically) decide to choose rather smaller $L_M$: simply, we choose $L_M=L_B$.

Now we let $\mA_i=\mM_i - \mB \mC^{-1}\mB^\top$ and $\mA = \frac{1}{n}\sum_{i=1}^n \mA_i$ to satisfy Equation~\eqref{eq:quadratic primal}. We emphasize that $\mA$ may have negative eigenvalues; the objective is nonconvex in $\vx$ in general. We have checked this is true across the experimental settings. 
Also, we let $L:=\max\{\|\mA\|_2, L_B, L_C\}$ for further parameter selection. (In fact, because of our choice of parameter values, $L$ was always equal to $L_B$ in our experiments.)

Furthermore, we generate the vectors $\vu_i$'s and $\vv_i$'s satisfying $\sum_{i=1}^n \vu_i = \vzero = \sum_{i=1}^n \vv_i$. The entries of these vectors are uniformly sampled from an interval $[-\Delta, \Delta]$, thereby the average of entries is centered to zero. In addition, to verify our theory, we choose the step-sizes of the form $\beta = c_1\cdot\nicefrac{b}{nL}$ and $\alpha=c_0\cdot\nicefrac{\beta}{\kappa_2^2}$ for some constants $c_0$ and $c_1$ and batch size $b$.

Lastly, we specify the values of parameters described above: $n=100$, $d=25$, $\mu_M=\mu_C$, and $L_C=1<L_M=L_B$. The constants $c_0$ and $c_1$ are tuned among $10^{\{-2,\,-1.5,\,\pm1,\,\pm0.5,\,0\}}$.
In the following subsections, we investigate the effects of the change of 
\begin{itemize}
    \item[(i)] $\Delta\in \{10, {\bf 20}, 40\}$, determining the discrepancy between components,
    \item[(ii)] condition number $\kappa_2\in \{5, {\bf 10}, 20\}$, determined by $L_B$ and $\mu_C$, and
    \item[(iii)] batch size $b\in \{{\bf 1}, 25, 50, 100\}$,
\end{itemize}
from the plots of the values potential function $V_\lambda(\vx;\vy) = (1+\lambda)\Phi(\vx) -f(\vx;\vy)$ over \emph{epochs}.\footnote{During and after the discussion phase, we performed some more experiments. As we tried to plot all the results over \emph{iterations}, the size of the figures in \texttt{pdf} format became too large. Consequently, in this appendix, we only plot the results over epochs to reduce the file size of the figures.} (Numbers in bold font above are the default values of parameters.)

\begin{figure}[t]
    \centering
    \begin{subfigure}[b]{0.3\textwidth}
        \centering
        \includegraphics[width=\textwidth]{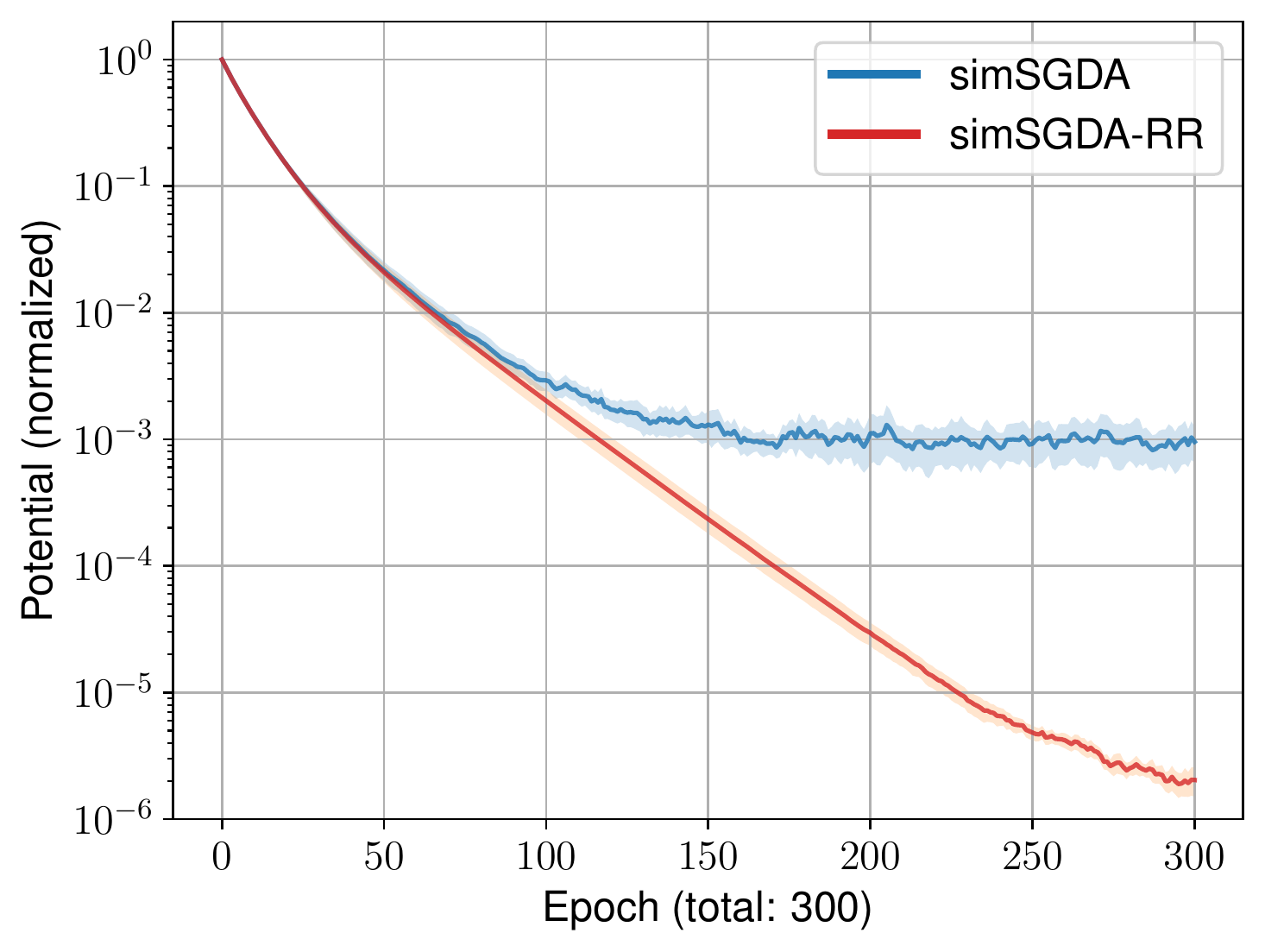}
        \caption{$\Delta=10$, {\color{tab:blue}simSGDA}{\color{tab:red}(-RR)}}
        \label{fig:Experiment:Delta10simSGDA}
    \end{subfigure}
    ~
    \begin{subfigure}[b]{0.3\textwidth}
        \centering
        \includegraphics[width=\textwidth]{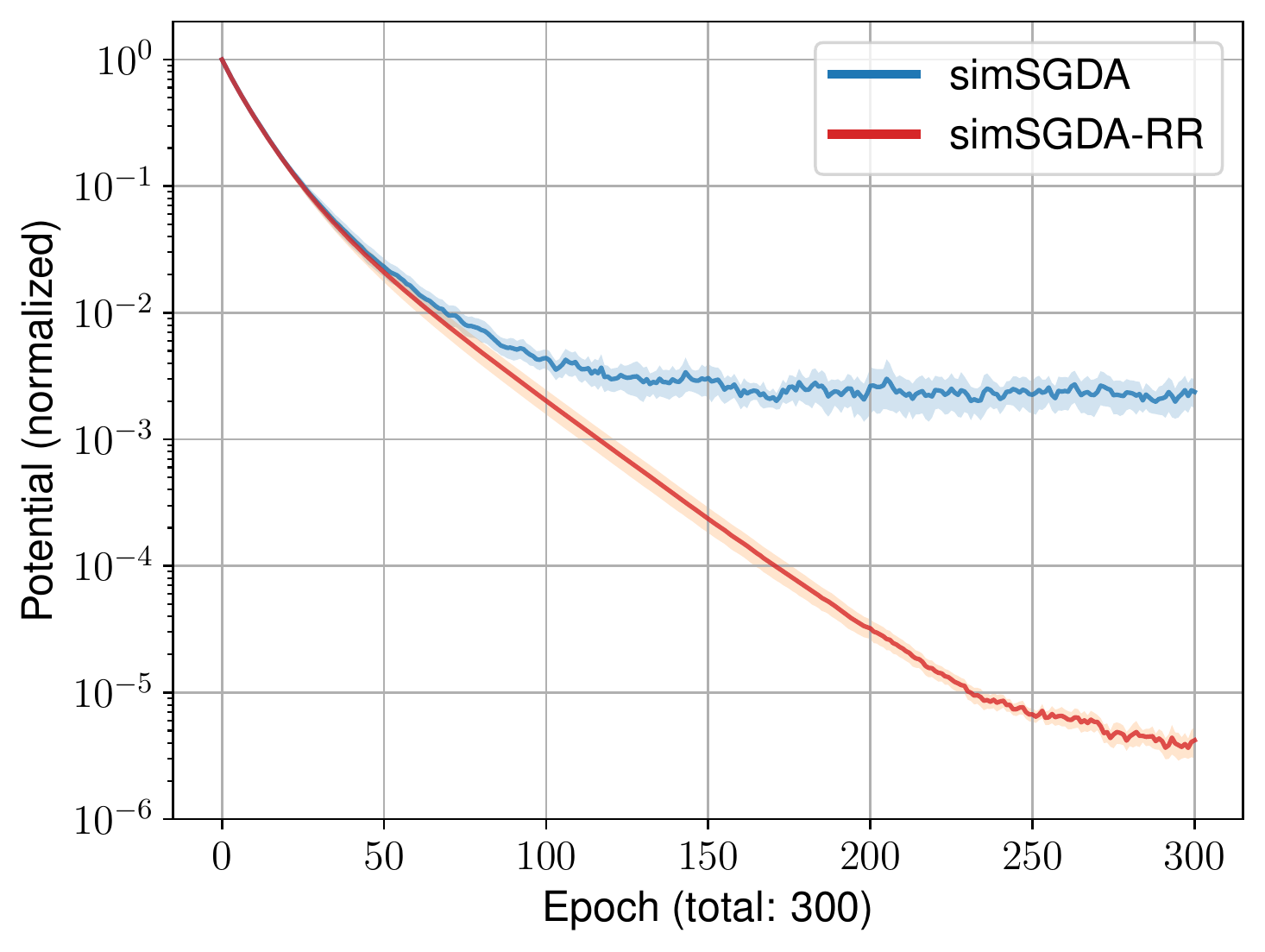}
        \caption{$\Delta=20$, {\color{tab:blue}simSGDA}{\color{tab:red}(-RR)}}
        \label{fig:Experiment:Delta20simSGDA}
    \end{subfigure}
    ~
    \begin{subfigure}[b]{0.3\textwidth}
        \centering
        \includegraphics[width=\textwidth]{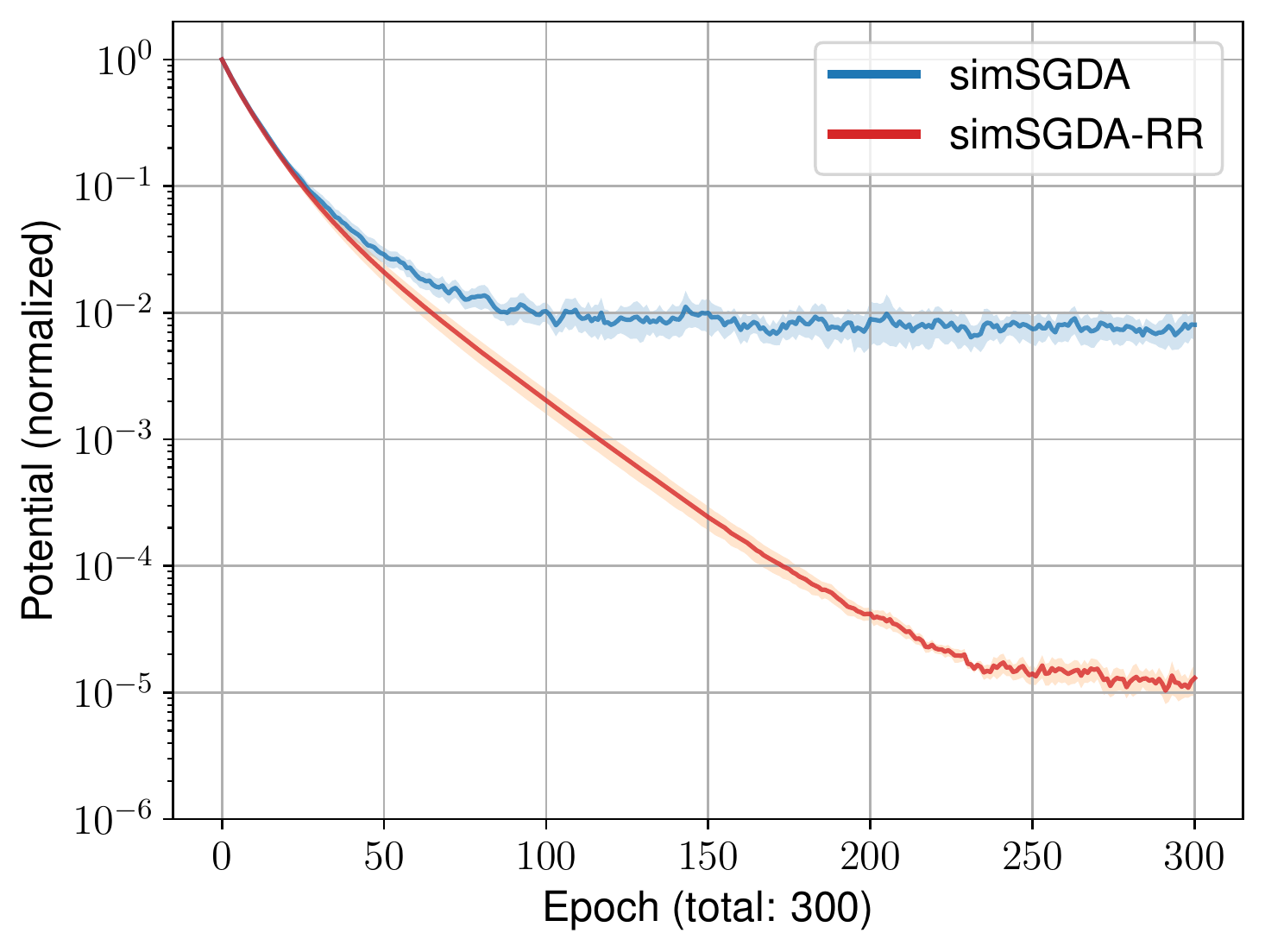}
        \caption{$\Delta=40$, {\color{tab:blue}simSGDA}{\color{tab:red}(-RR)}}
        \label{fig:Experiment:Delta40simSGDA}
    \end{subfigure}
    \\
    \begin{subfigure}[b]{0.3\textwidth}
        \centering
        \includegraphics[width=\textwidth]{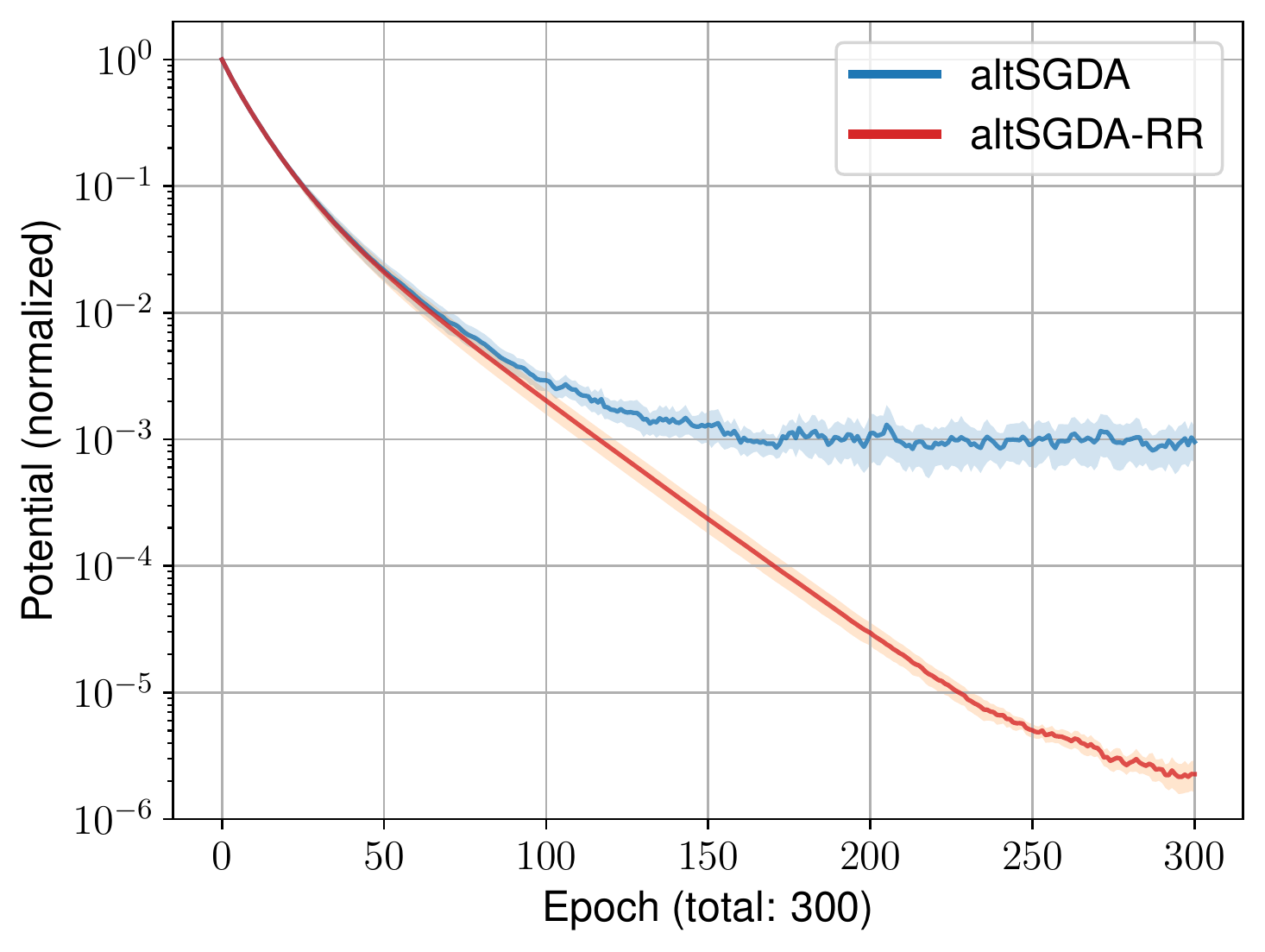}
        \caption{$\Delta=10$, {\color{tab:blue}altSGDA}{\color{tab:red}(-RR)}}
        \label{fig:Experiment:Delta10altSGDA}
    \end{subfigure}
    ~
    \begin{subfigure}[b]{0.3\textwidth}
        \centering
        \includegraphics[width=\textwidth]{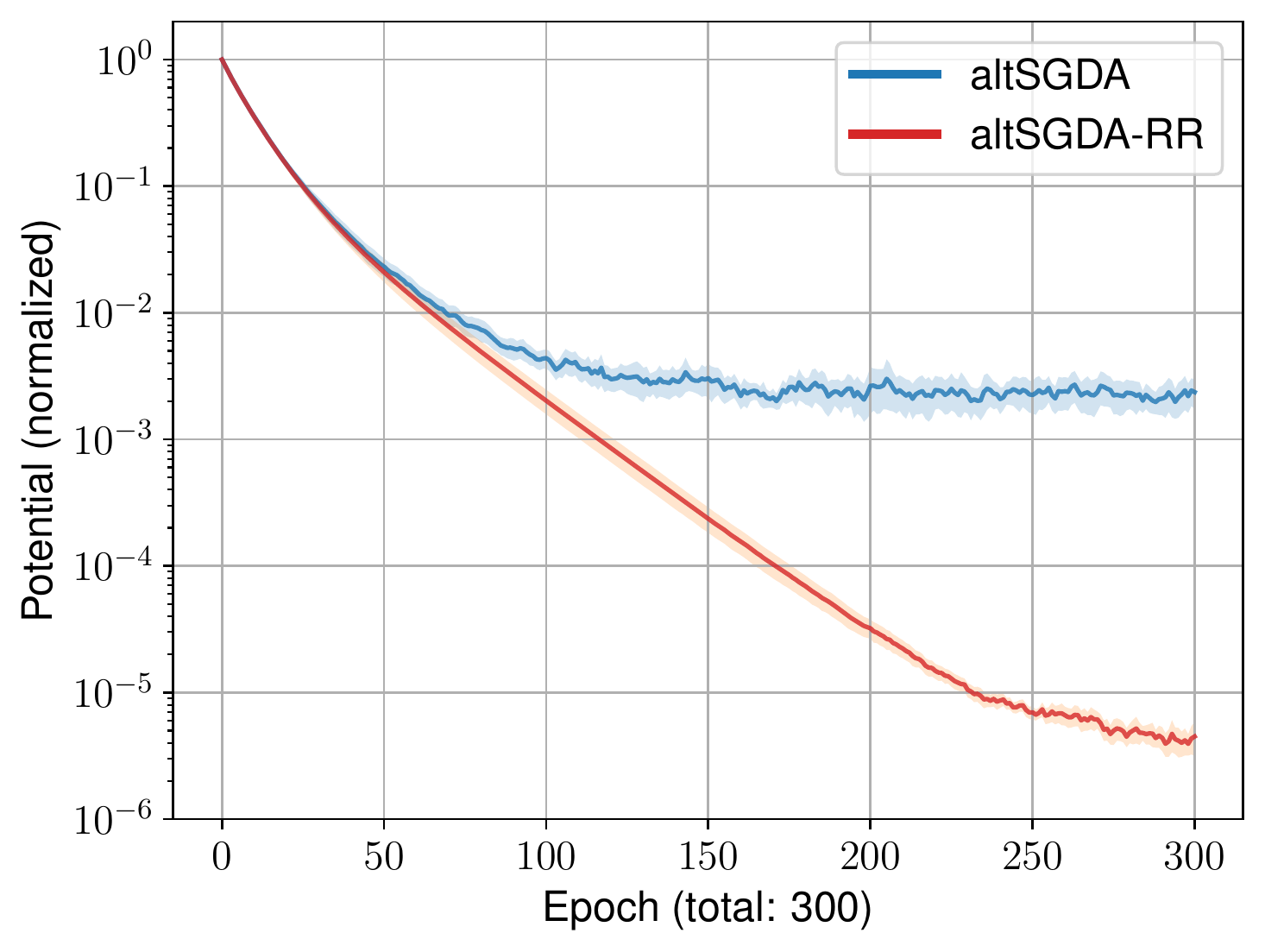}
        \caption{$\Delta=20$, {\color{tab:blue}altSGDA}{\color{tab:red}(-RR)}}
        \label{fig:Experiment:Delta20altSGDA}
    \end{subfigure}
    ~
    \begin{subfigure}[b]{0.3\textwidth}
        \centering
        \includegraphics[width=\textwidth]{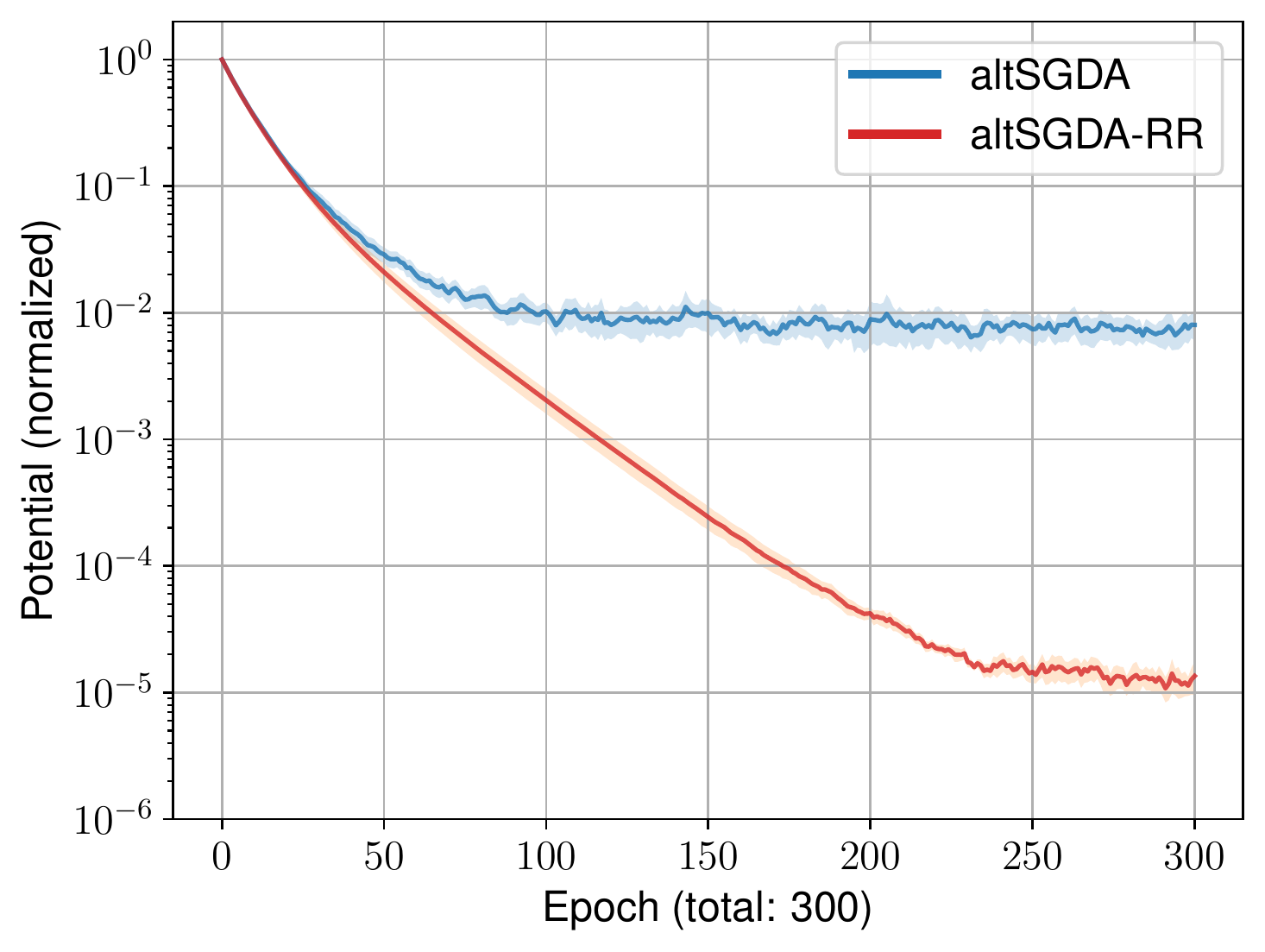}
        \caption{$\Delta=40$, {\color{tab:blue}altSGDA}{\color{tab:red}(-RR)}}
        \label{fig:Experiment:Delta40altSGDA}
    \end{subfigure}
    \\
    \begin{subfigure}[b]{0.3\textwidth}
        \centering
        \includegraphics[width=\textwidth]{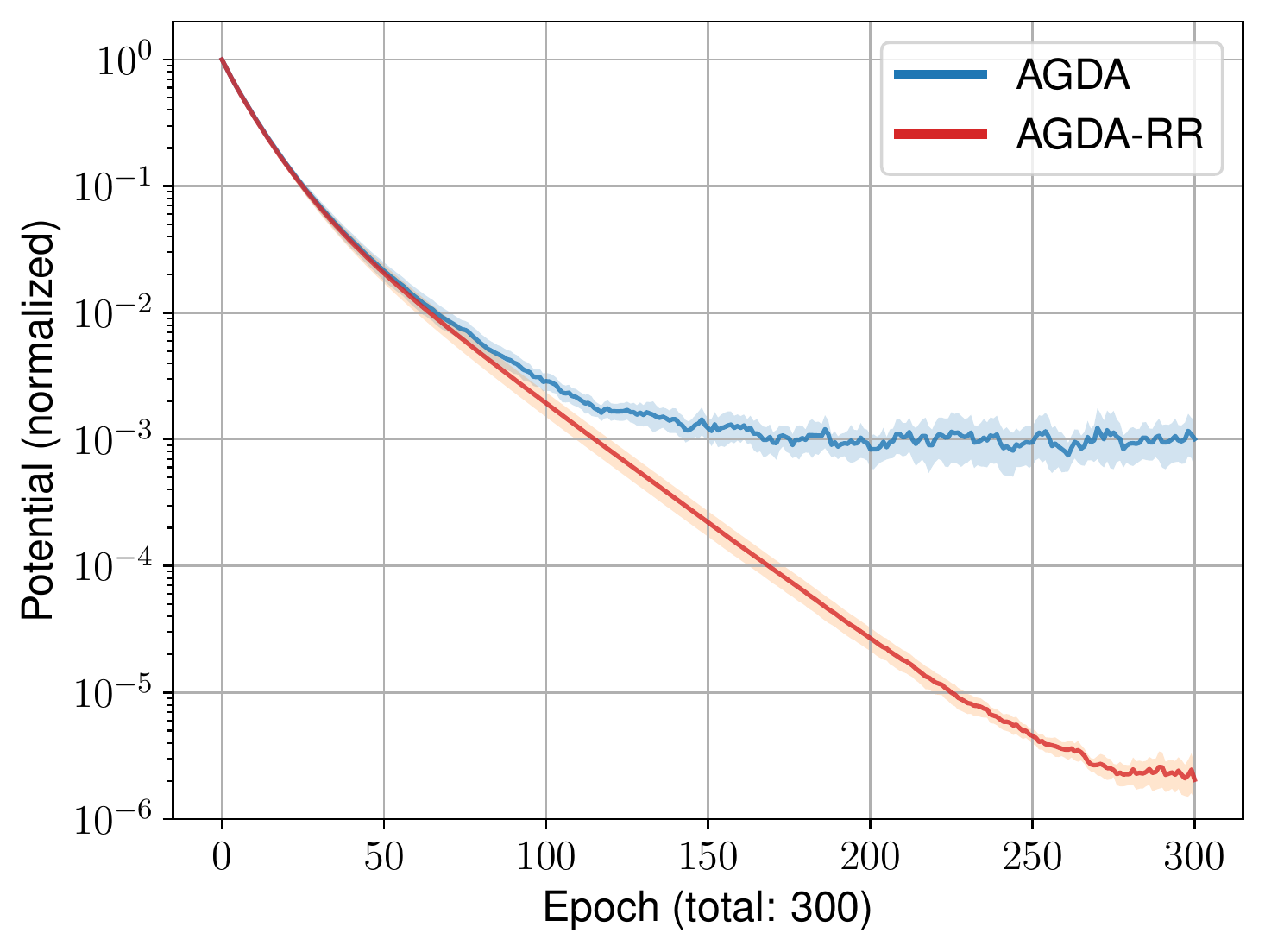}
        \caption{$\Delta=10$, {\color{tab:blue}AGDA}{\color{tab:red}(-RR)}}
        \label{fig:Experiment:Delta10AGDA}
    \end{subfigure}
    ~
    \begin{subfigure}[b]{0.3\textwidth}
        \centering
        \includegraphics[width=\textwidth]{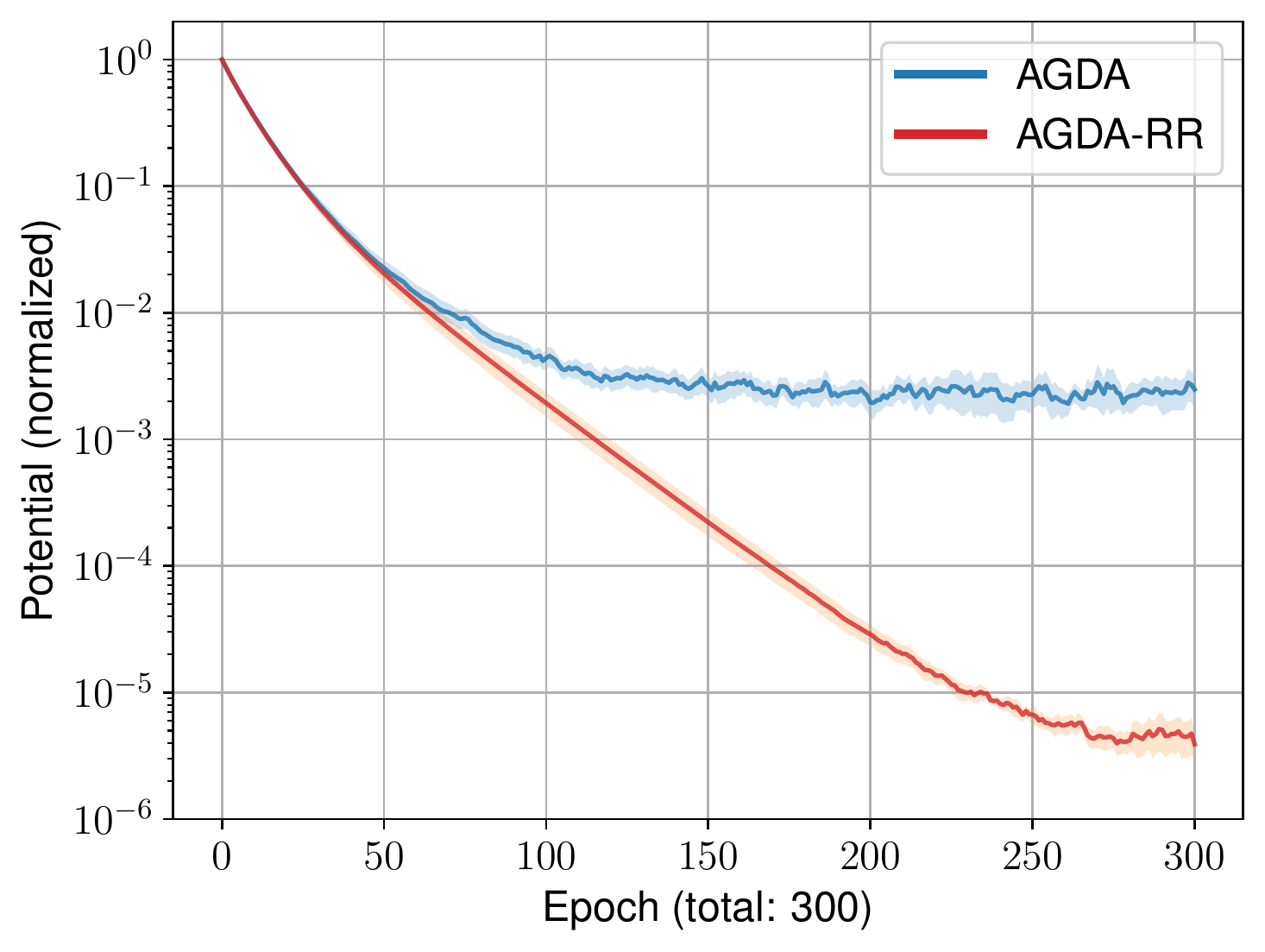}
        \caption{$\Delta=20$, {\color{tab:blue}AGDA}{\color{tab:red}(-RR)}}
        \label{fig:Experiment:Delta20AGDA}
    \end{subfigure}
    ~
    \begin{subfigure}[b]{0.3\textwidth}
        \centering
        \includegraphics[width=\textwidth]{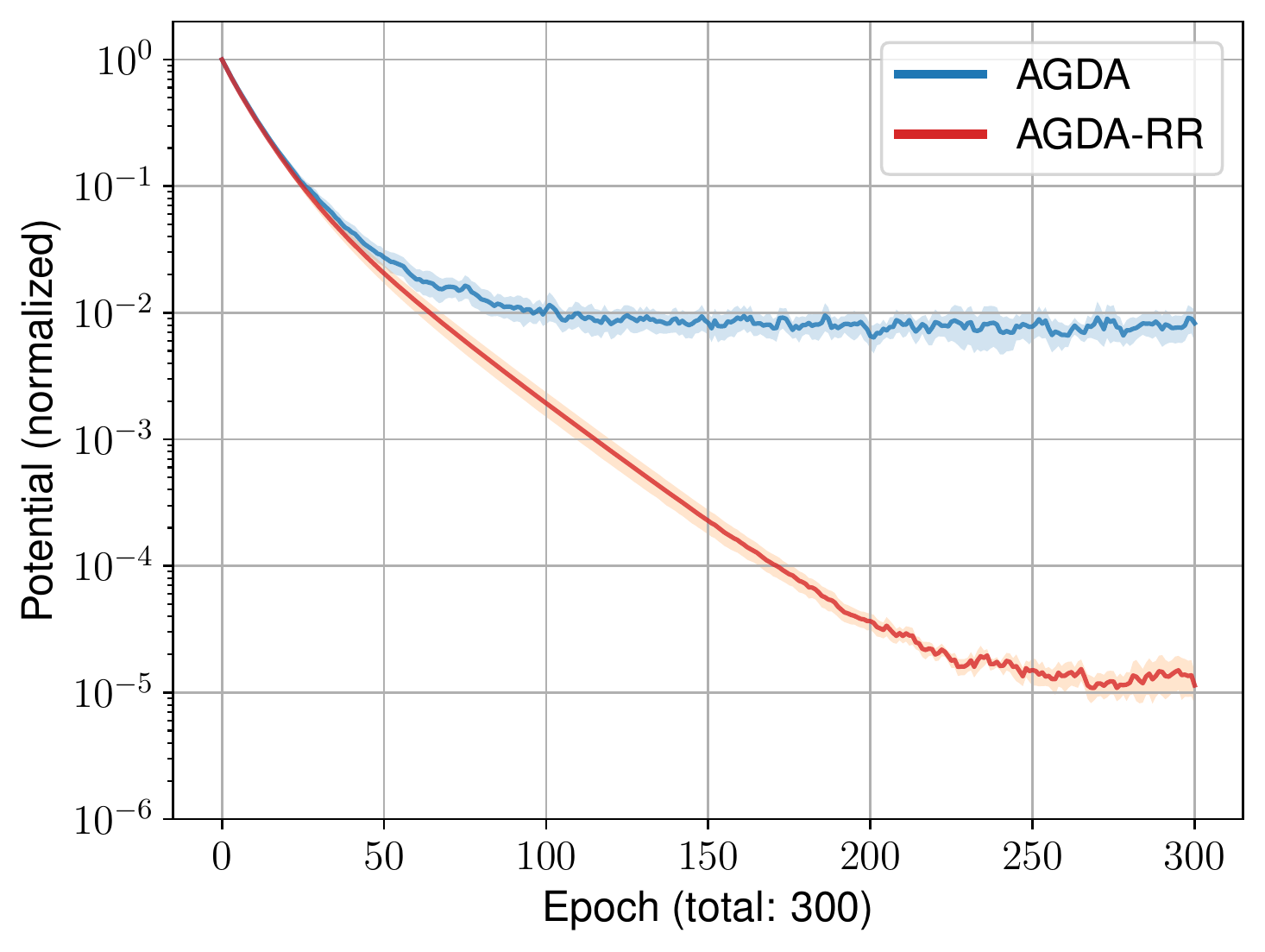}
        \caption{$\Delta=40$, {\color{tab:blue}AGDA}{\color{tab:red}(-RR)}}
        \label{fig:Experiment:Delta40AGDA}
    \end{subfigure}
    \caption{Comparisons by changing the value of $\Delta\in\{10,20,40\}$. Solid lines: average across 10 different runs. Shaded regions: 95\% confidence intervals ($\pm 1.96$ std). The vertical axes are on a \emph{logarithmic scale}.}
    \label{fig:ExperimentDeltaComparison}
\end{figure}

\subsection{Comparison: the effect of component discrepancy}

Notice that the discrepancy between component functions gets larger as $\Delta$ grows. Technically, one can check that the gradient variance (that we controlled in Assumption~\ref{ass:bddvar}) is proportional to the norms of the vectors $\vu_i$ and $\vv_i$. Moreover, we have already discussed that the gap between convergence speeds of SGDA and SGDA-RR becomes larger especially when the gradient variance is large. 

Now, we present the results of numerical experiments by varying the values of $\Delta$ to 10, 20, and 40, while fixing $L_B=4$, $\mu_C=0.4$, $b=1$, and other experiment parameters. As shown in \Figref{fig:ExperimentDeltaComparison}, we can observe that the difference between the random-reshuffling algorithm and the uniform-sampling algorithm gets larger as $\Delta$ increases.

\begin{figure}[t]
    \centering
    \begin{subfigure}[b]{0.3\textwidth}
        \centering
        \includegraphics[width=\textwidth]{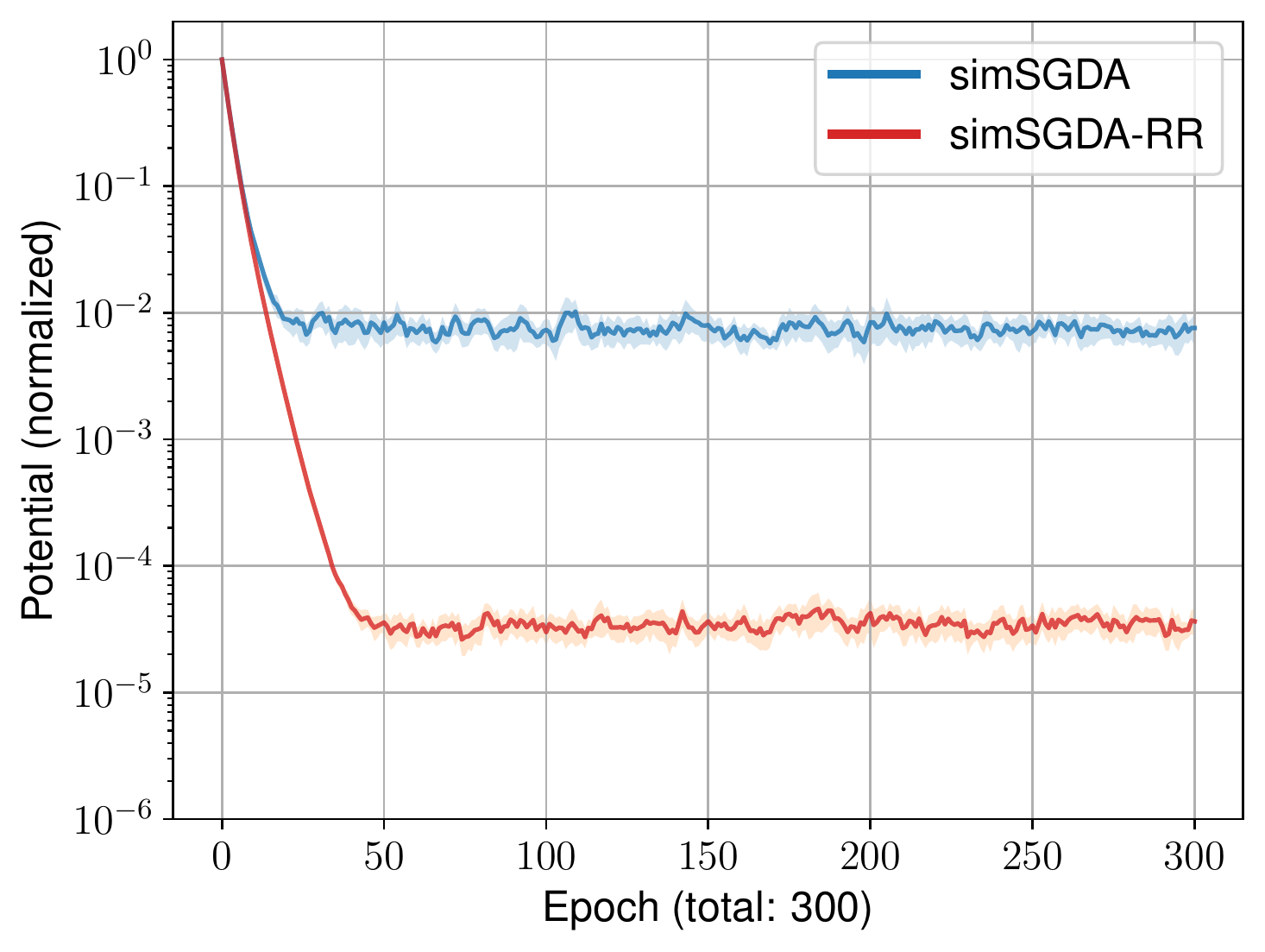}
        \caption{$\kappa_2=5$, {\color{tab:blue}simSGDA}{\color{tab:red}(-RR)}}
        \label{fig:Experiment:kappa5simSGDA}
    \end{subfigure}
    ~
    \begin{subfigure}[b]{0.3\textwidth}
        \centering
        \includegraphics[width=\textwidth]{Lyapunov_simSGDA_QuadraticPrimalPLSC_base.pdf}
        \caption{$\kappa_2=10$, {\color{tab:blue}simSGDA}{\color{tab:red}(-RR)}}
        \label{fig:Experiment:kappa10simSGDA}
    \end{subfigure}
    ~
    \begin{subfigure}[b]{0.3\textwidth}
        \centering
        \includegraphics[width=\textwidth]{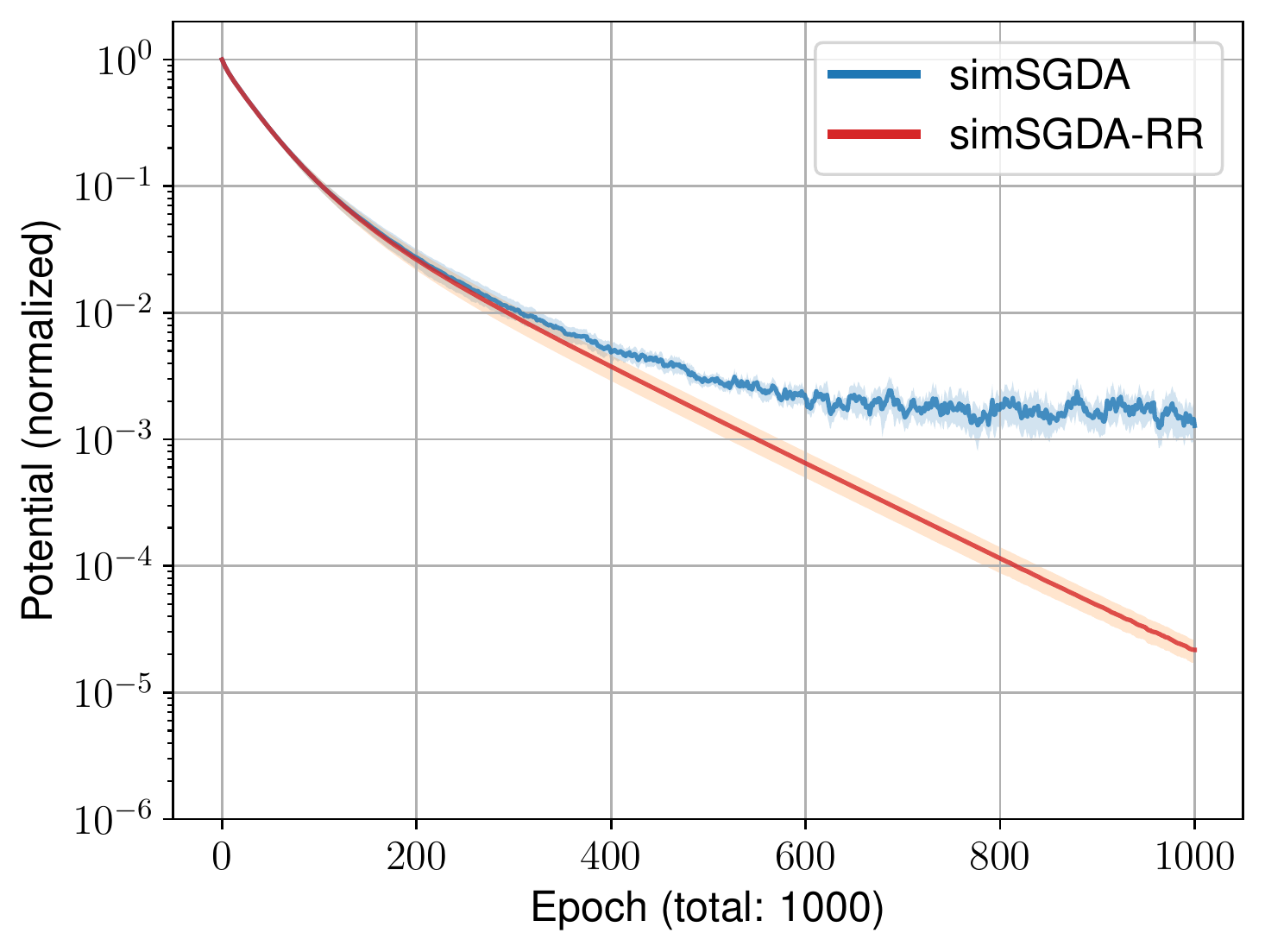}
        \caption{$\kappa_2=20$, {\color{tab:blue}simSGDA}{\color{tab:red}(-RR)}}
        \label{fig:Experiment:kappa20simSGDA}
    \end{subfigure}
    \\
    \begin{subfigure}[b]{0.3\textwidth}
        \centering
        \includegraphics[width=\textwidth]{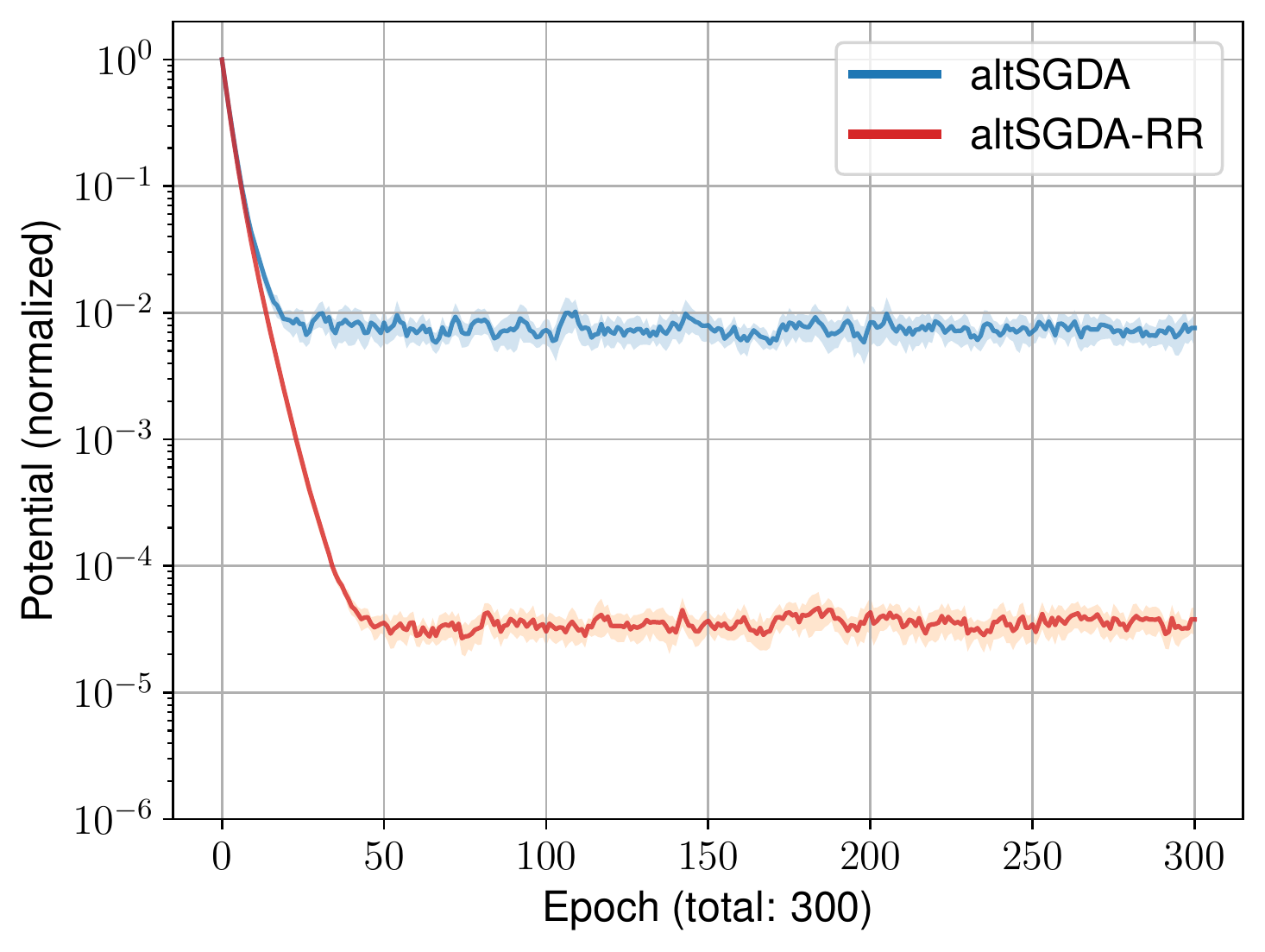}
        \caption{$\kappa_2=5$, {\color{tab:blue}altSGDA}{\color{tab:red}(-RR)}}
        \label{fig:Experiment:kappa5altSGDA}
    \end{subfigure}
    ~
    \begin{subfigure}[b]{0.3\textwidth}
        \centering
        \includegraphics[width=\textwidth]{Lyapunov_altSGDA_QuadraticPrimalPLSC_base.pdf}
        \caption{$\kappa_2=10$, {\color{tab:blue}altSGDA}{\color{tab:red}(-RR)}}
        \label{fig:Experiment:kappa10altSGDA}
    \end{subfigure}
    ~
    \begin{subfigure}[b]{0.3\textwidth}
        \centering
        \includegraphics[width=\textwidth]{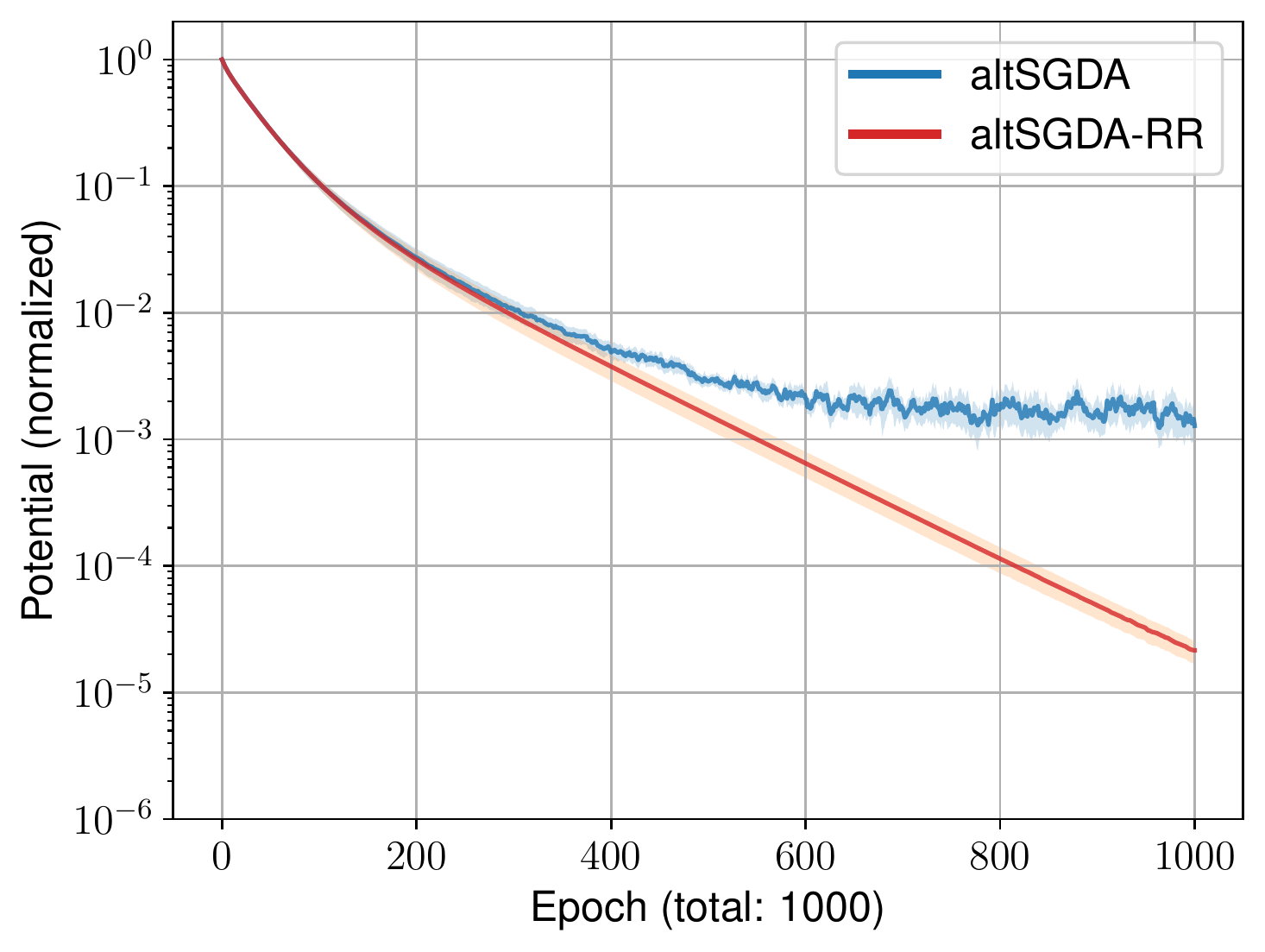}
        \caption{$\kappa_2=20$, {\color{tab:blue}altSGDA}{\color{tab:red}(-RR)}}
        \label{fig:Experiment:kappa20altSGDA}
    \end{subfigure}
    \\
    \begin{subfigure}[b]{0.3\textwidth}
        \centering
        \includegraphics[width=\textwidth]{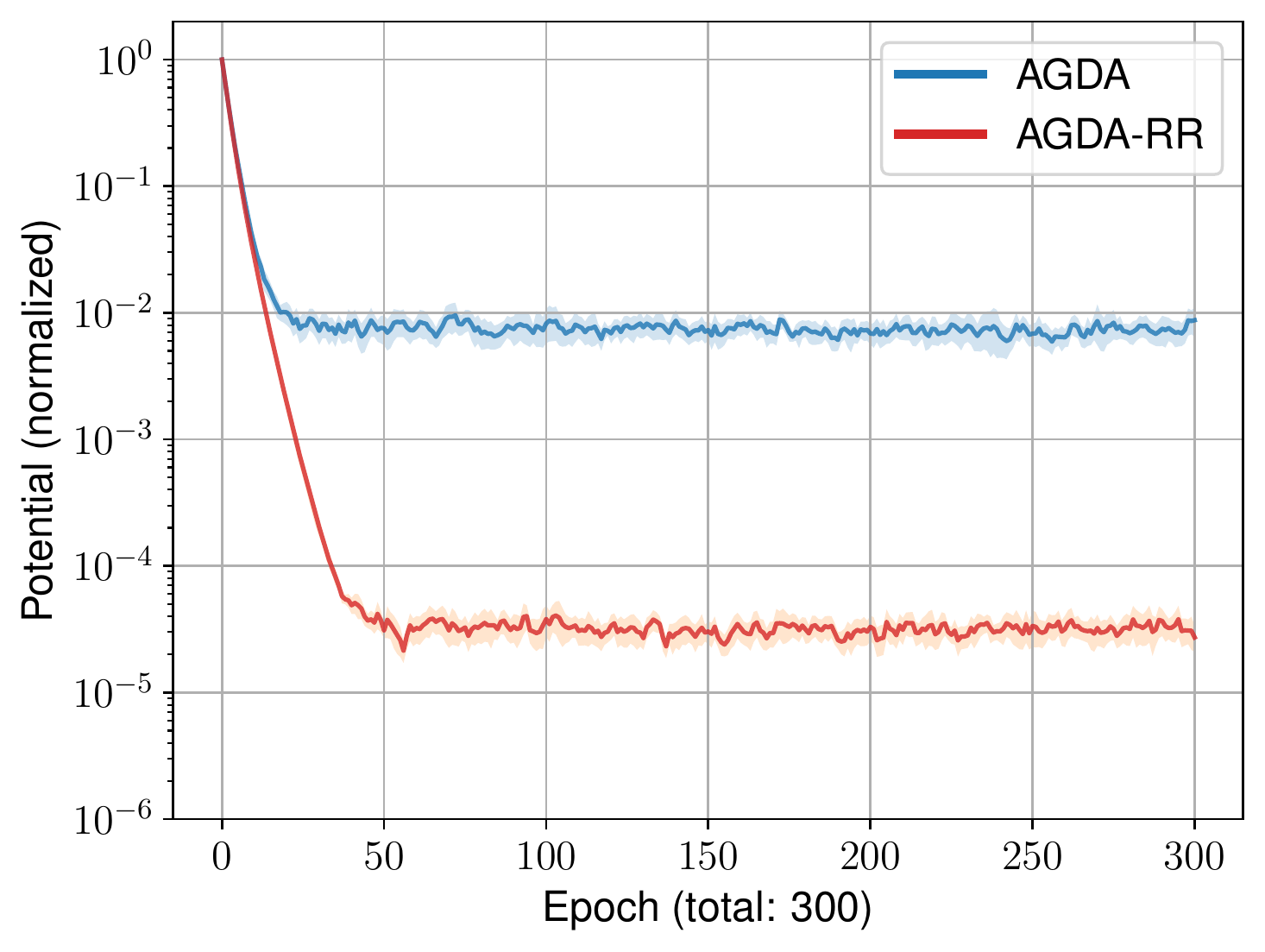}
        \caption{$\kappa_2=5$, {\color{tab:blue}AGDA}{\color{tab:red}(-RR)}}
        \label{fig:Experiment:kappa5AGDA}
    \end{subfigure}
    ~
    \begin{subfigure}[b]{0.3\textwidth}
        \centering
        \includegraphics[width=\textwidth]{Lyapunov_AGDA_QuadraticPrimalPLSC_base.pdf}
        \caption{$\kappa_2=10$, {\color{tab:blue}AGDA}{\color{tab:red}(-RR)}}
        \label{fig:Experiment:kappa10AGDA}
    \end{subfigure}
    ~
    \begin{subfigure}[b]{0.3\textwidth}
        \centering
        \includegraphics[width=\textwidth]{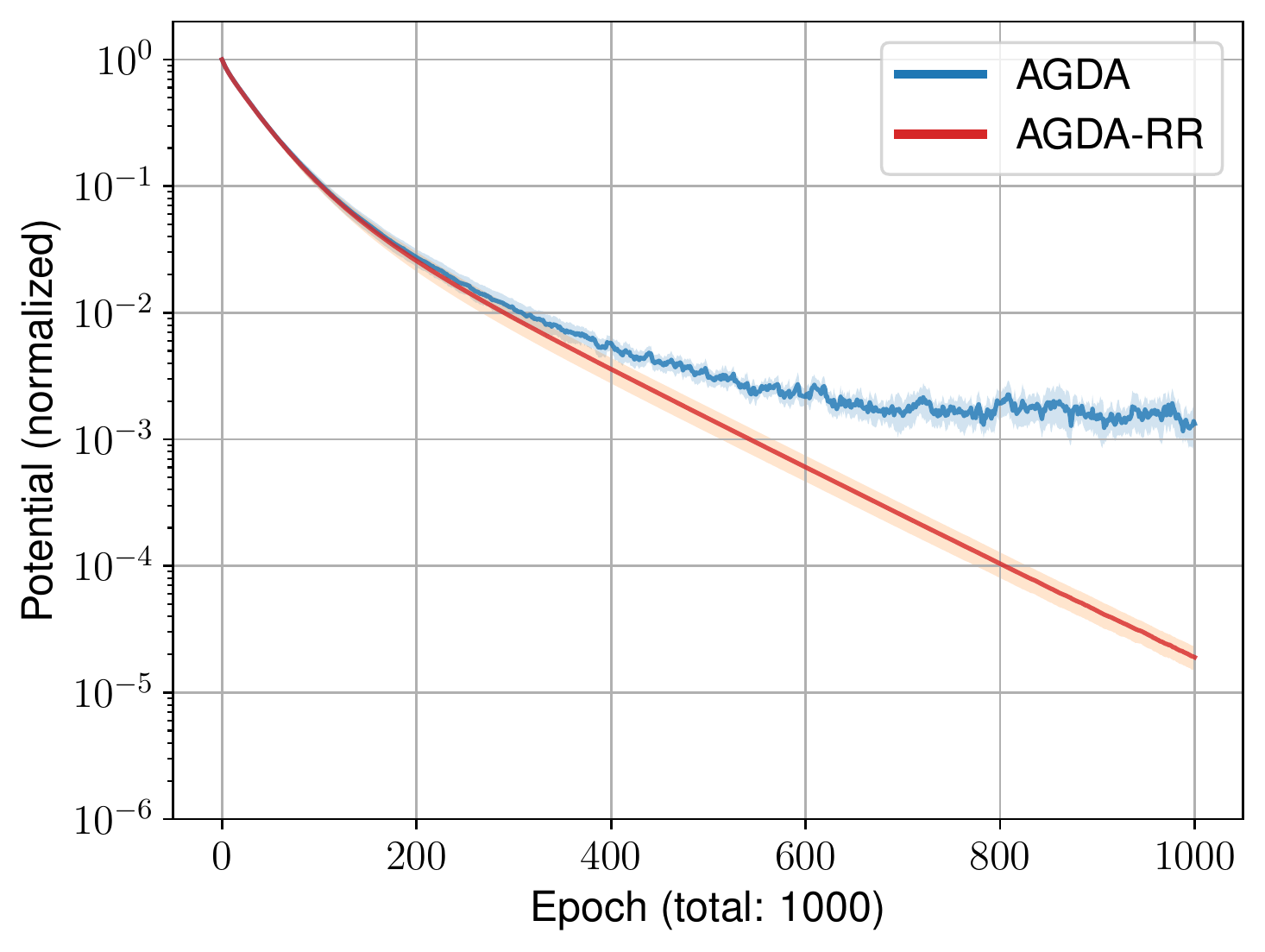}
        \caption{$\kappa_2=20$, {\color{tab:blue}AGDA}{\color{tab:red}(-RR)}}
        \label{fig:Experiment:kappa20AGDA}
    \end{subfigure}
    \caption{Comparisons by changing the value of $\kappa_2=L/\mu_C\in\{5,10,20\}$. Solid lines: average across 10 different runs. Shaded regions: 95\% confidence intervals ($\pm 1.96$ std). The vertical axes are on a \emph{logarithmic scale}. 
    Note: we run \textbf{1000 epochs for $\kappa_2=20$} (see the rightmost column), whereas we run \textbf{300 epochs for the other $\kappa_2\in\{5, 10\}$} (see the leftmost \& middle columns).}
    \label{fig:ExperimentkappaComparison}
\end{figure}

\subsection{Comparison: the effect of condition number}

Here, we present the results of experiments by varying the values of $\kappa_2$ to 5, 10, and 20, while fixing $\Delta=20$, $b=1$, and other experiment parameters. To this end, we applied the parameter settings for $L_B$ and $\mu_C$ as $(L_B, \mu_C) = (2.5,0.5), (4, 0.4), (5, 0.25)$, respectively.

The results are shown in \Figref{fig:ExperimentkappaComparison}. We observe that more epochs are required for convergence when $\kappa_2$ increases, regardless of the type of algorithm. One may think that the performance gap between RR-based/non-RR-based algorithms is small when $\kappa_2$ is huge. However, when we run the algorithm for an extended number of epochs, we observe a significant gap in convergence speeds.

\begin{figure}[t]
    \centering
    \begin{subfigure}[b]{0.45\textwidth}
        \centering
        \includegraphics[width=\textwidth]{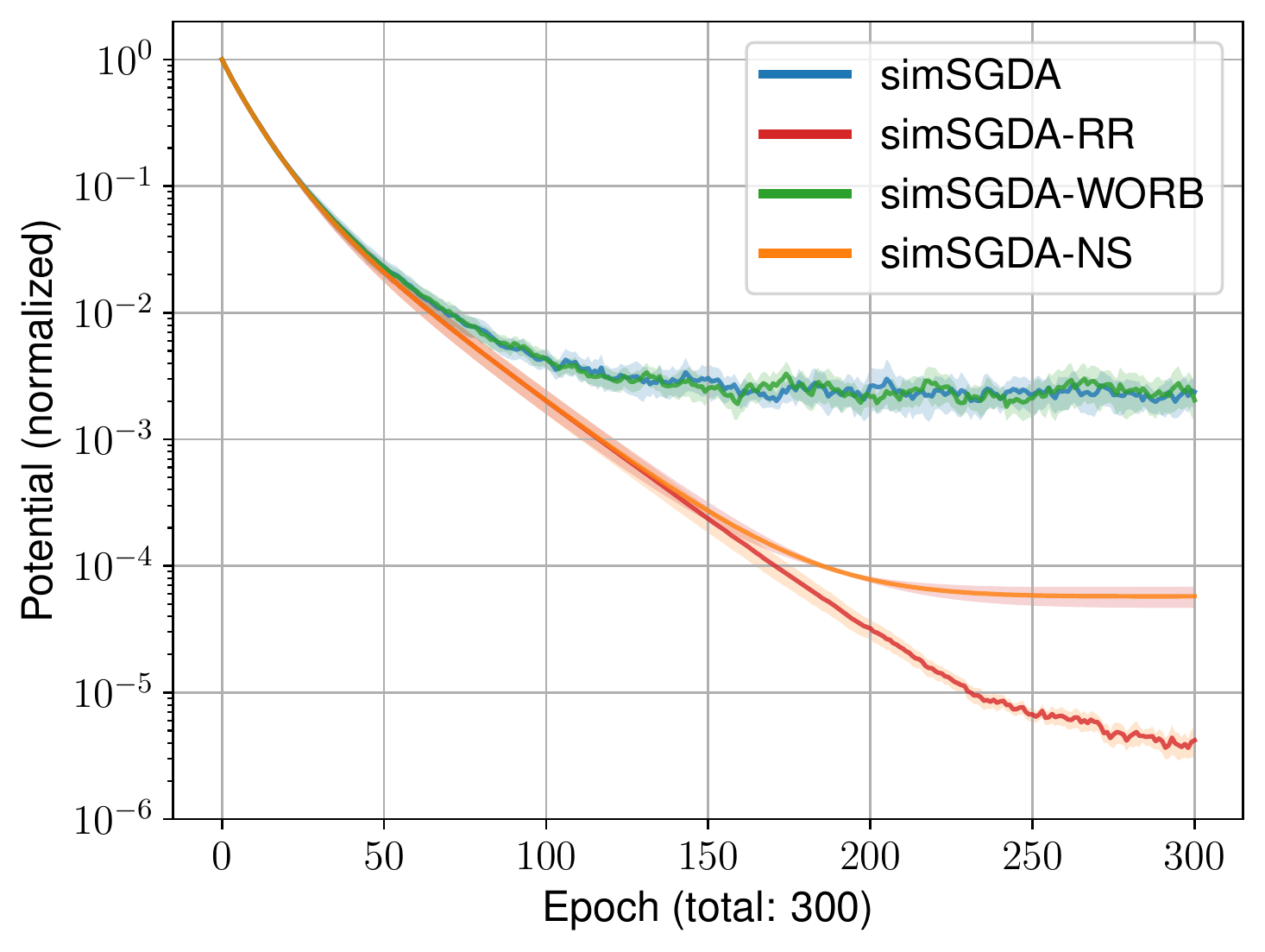}
        \caption{$b=1$}
        \label{fig:Experiment:bs1simSGDA}
    \end{subfigure}
    ~
    \begin{subfigure}[b]{0.45\textwidth}
        \centering
        \includegraphics[width=\textwidth]{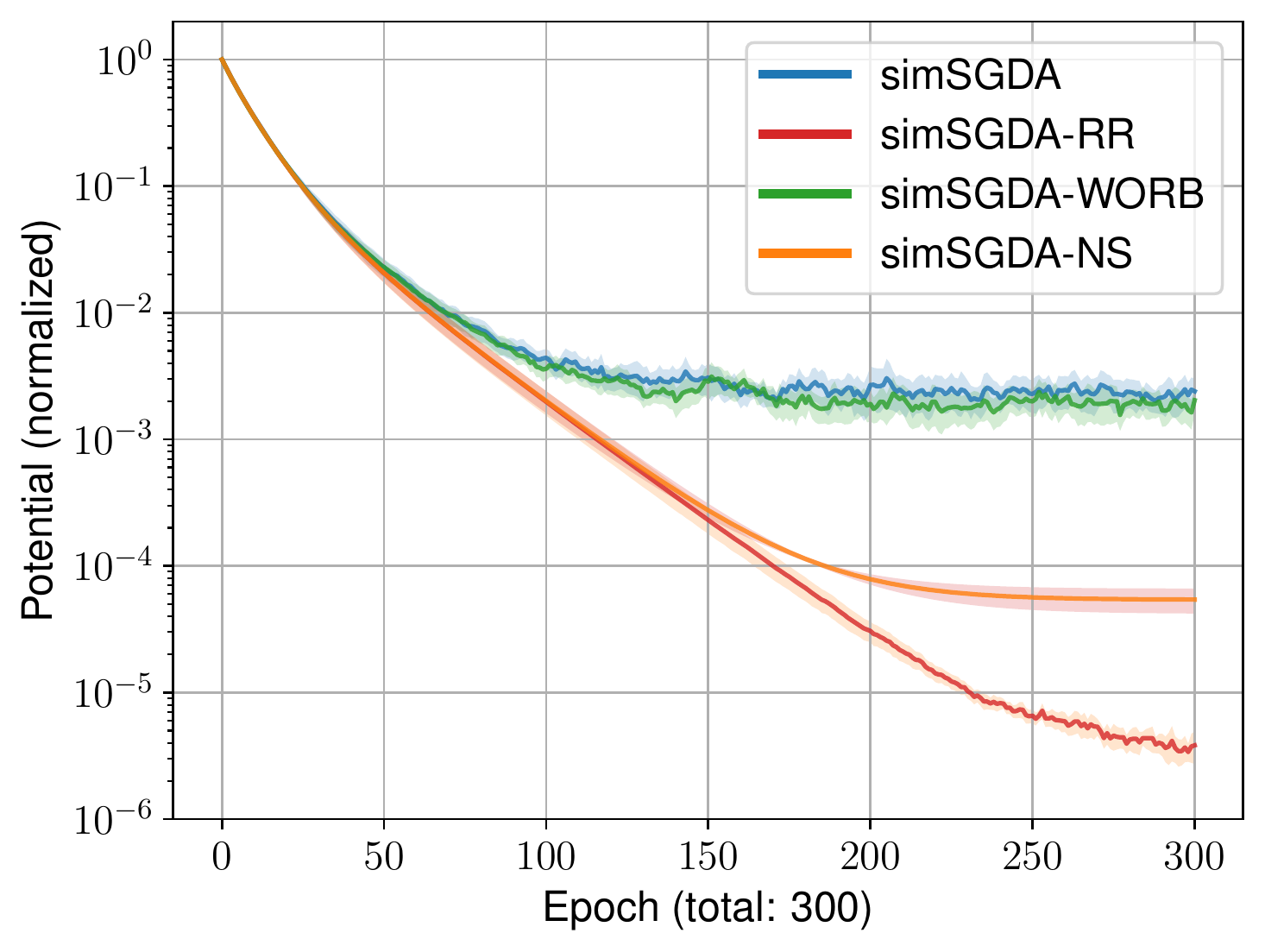}
        \caption{$b=25$}
        \label{fig:Experiment:bs25simSGDA}
    \end{subfigure}
    \\
    \begin{subfigure}[b]{0.45\textwidth}
        \centering
        \includegraphics[width=\textwidth]{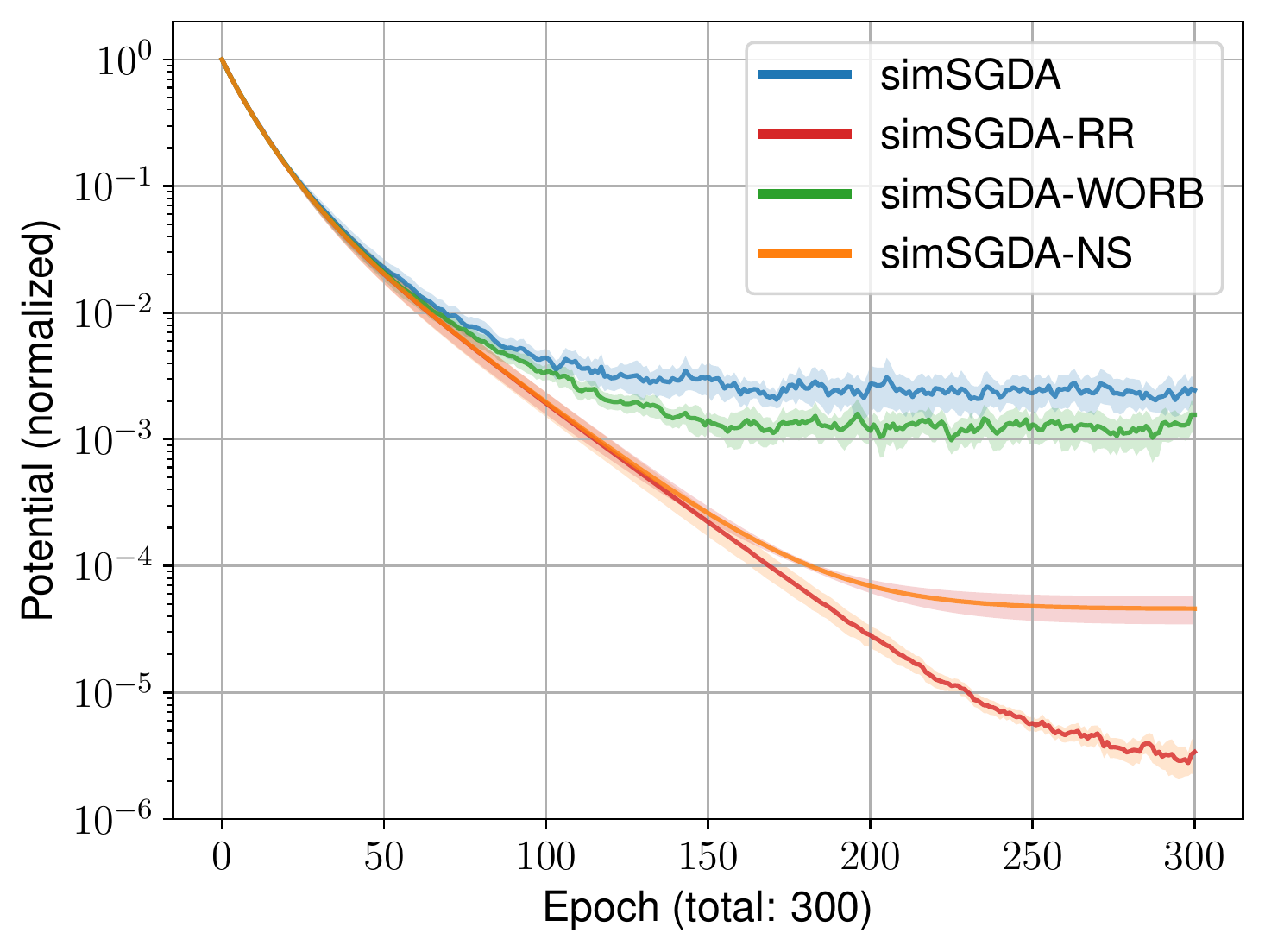}
        \caption{$b=50$}
        \label{fig:Experiment:bs50simSGDA}
    \end{subfigure}
    ~
    \begin{subfigure}[b]{0.45\textwidth}
        \centering
        \includegraphics[width=\textwidth]{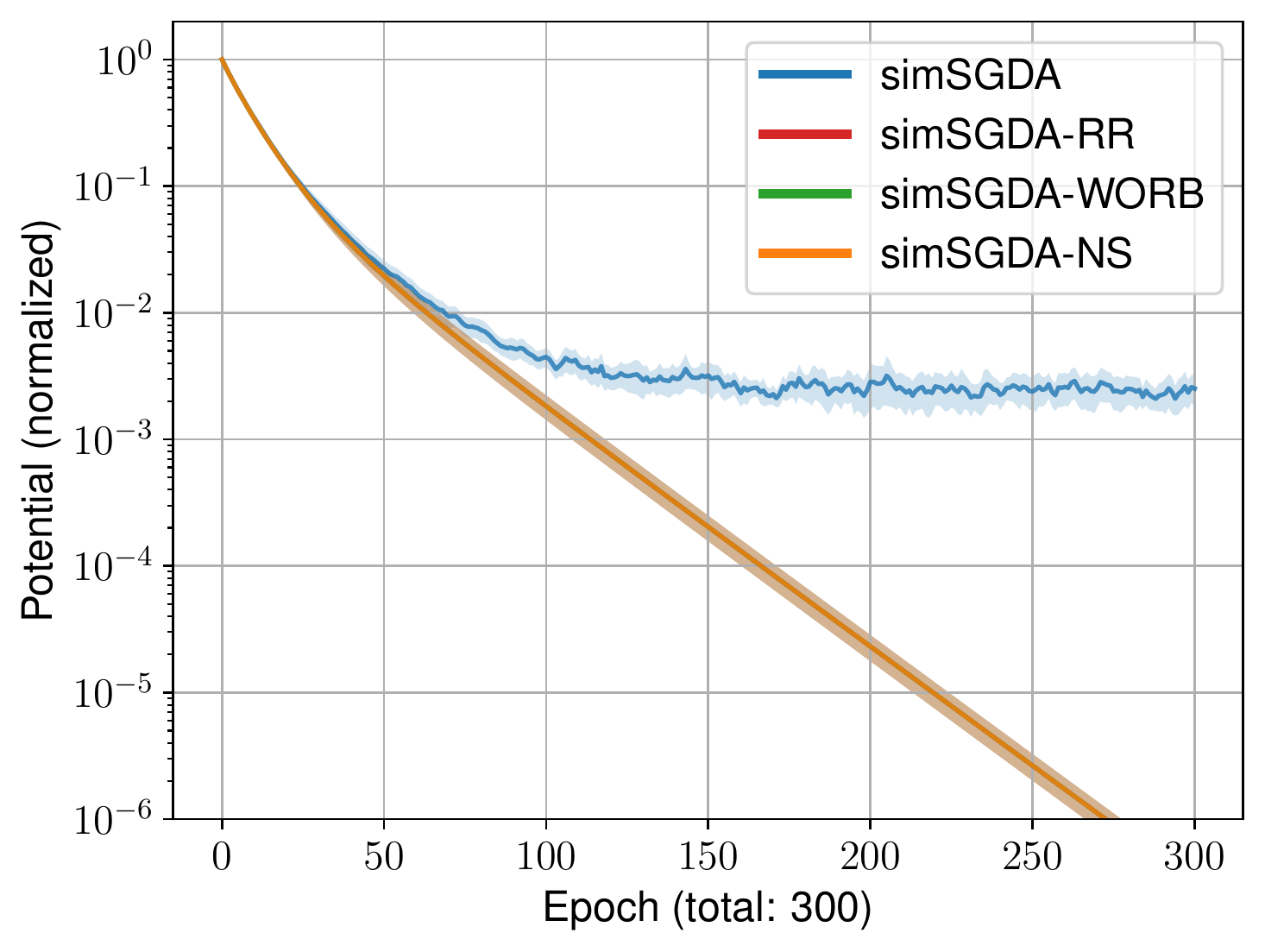}
        \caption{$b=100$}
        \label{fig:Experiment:bs100simSGDA}
    \end{subfigure}
    \\
    \caption{Comparisons of {\color{tab:blue}simSGDA}({\color{tab:red}-RR},{\color{tab:green}-WORB},{\color{tab:orange}-NS}) as changing $b\in\{1,25,50,100\}$. Solid lines: average across 10 different runs. Shaded regions: 95\% confidence intervals ($\pm 1.96$ std). The vertical axes are on a \emph{logarithmic scale}.}
    \label{fig:ExperimentBsComparison}
\end{figure}

\subsection{Comparison: the effect of batch size}

The last comparison is about the effect of batch size $b\in\{1, 25,50,100\}$. Recall that we linearly scale the step sizes as the batch size changes. However, since the number of epochs is fixed, the number of iterations decreases as $b$ gets larger.

As the readers can notice, the convergence behavior of SGDA (resp., SGDA-RR) and AGDA (resp., AGDA-RR) are similar in our construction of quadratic games. Thus, in this subsection, we only compare simSGDA and its variants. Rather, we introduce two more methods of component choice other than with-replacement uniform sampling and random reshuffling:
\begin{itemize}
    \item {\color{tab:green} WORB}(WithOut-Replacement mini-Batching): every mini-batch is without-replacement \& uniformly-randomly sampled, while any pair of mini-batches in an epoch may have some indices in common; the same as \emph{$b$-minibatch sampling} \citep{loizou2021stochastic}.
    \item {\color{tab:orange} NS}(No Shuffle): accessing $1,...,n$ in its predefined order to construct mini-batches; without-replacement but deterministic. Remark: for \emph{minimization} problems, SGD with NS is usually referred to as \emph{incremental gradient} (IG) algorithm \citep{mishchenko2020random}.
\end{itemize}

These two methods are somewhat related to without-replacement component sampling, whereas they are both different from RR which uniformly randomly samples a permutation of $[n]$ every epoch. We call simSGDA using mini-batches sampled by WORB and NS as \emph{simSGDA-WORB} and \emph{simSGDA-NS}, respectively. Remarks: If $b=1$, simSGDA-WORB becomes the same algorithm as vanilla simSGDA. Also, since we choose $n=100$, if $b=n=100$, all three algorithms simSGDA-RR/-WORB/-NS become the same as deterministic \& full-batch (simultaneous) GDA.

The results are shown in \Figref{fig:ExperimentBsComparison}. One can notice that the potential plots of simSGDA, simSGDA-RR, and simSGDA-NS are respectively the same even if we change the batch size ($b<100$). Also, if $b>1$, simSGDA-WORB has better performance than vanilla simSGDA. These imply that without-replacement mini-batches benefit the convergence speed to some extent in our quadratic game. However, the result of experiments also implies that both (i) without-replacement \emph{per epoch} (\ie, shuffling) and (ii) randomization are indeed essential for fast convergence in our quadratic game experiments. In particular, WORB requires a very large batch size but still has a much slower convergence rate than RR (see \Figref{fig:Experiment:bs50simSGDA} which is the case of using half of the total components at each iteration).

\section{Omitted comparison with related works}
\label{sec:omitted comparison}
\subsection{Comparison with \cite{xie2021efficient}}

To specialize \citet[Theorem~3]{xie2021efficient} to the single-machine setup and discuss their results in terms of our notation, we need to replace their symbols
\[(T,S,K,\sigma_1^2, \sigma_2^2, G_1^2, G_2^2,  L_{12},L_f, \mu, L_\Phi, \gL_0, \eta_t, \gamma_t)\]
with the following symbols from our notation 
\[(K,1,n,0,0,B,B,L, L,\mu_2, L(\kappa_2+1),V_\lambda(\vz_0^1), \alpha, \beta),\]
and also put $A=0$ (their analysis only applies uniformly bounded component variance per machine).
Then we can \emph{naively} translate the bound of \citet[Theorem~3]{xie2021efficient} to our language as 
\[\min_{k\in [K]} \E \left[\norm{\Phi(\vx)}^2\right] \stackrel{\text{(?)}}{\le} \gO\left(\frac{\kappa_2LV_\lambda(\vz_0^1)}{K}+\kappa_2^2\left(\frac{L^2BV_\lambda(\vz_0^1)^2}{K^2}\right)^{1/3}\right).\]
To the best of our knowledge, however, we believe there may be a mistake in the proof of \citet[Appendix~C.4]{xie2021efficient}. From the inequalities on the last page of their paper, we notice that the term $\frac{40L_{12}^2\gL_0}{\mu^2\gamma K T}$ might be missing in a step, where $\gamma$ is chosen to be the minimum of several terms including $\frac{1}{87L_fK}$. Thus, as far as we can tell, it seems inevitable that this omitted term would lead to an additional term $\frac{3480L_fL_{12}^2\gL_0}{\mu^2 T}$ in the final bound. By combining this to their bound and re-translating it, we eventually have
\[\min_{k\in [K]} \E \left[\norm{\Phi(\vx)}^2\right] {\le} \gO\left(\frac{{\color{tab:red}\kappa_2^2}LV_\lambda(\vz_0^1)}{K}+\kappa_2^2\left(\frac{L^2BV_\lambda(\vz_0^1)^2}{K^2}\right)^{1/3}\right),\]
since their $L_{12}^2/\mu^2$ translates to our $\kappa_2^2$.
Therefore, their result actually shows the same dependency on condition number $\kappa_2$ as our Theorem~\ref{thm:NCPL}. Nevertheless, comparing the terms related to the component-wise variance $B$, ours is better. In the second term in the bound above does not shrink even when the number of iterations (per machine \& per communication) grows. In our case (Theorem~\ref{thm:NCPL}), however, the dominant term (in $K$) can be briefly written as $\gO\left(\left(\frac{B}{nK^2}\right)^{1/3}\right)$ which can diminish with large $n$, \ie, the number of iterations per epoch.

\end{document}